\documentclass[12pt]{amsart}
\usepackage{amsbsy}
\usepackage{graphicx,epsfig,subfigure,psfrag}
\begin{document}

\newtheorem{theorem}{Theorem}
\newtheorem{proposition}{Proposition}
\newtheorem{lemma}{Lemma}
\newtheorem{corollary}{Corollary}
\newtheorem{definition}{Definition}
\newtheorem{remark}{Remark}
\newcommand{\tex}{\textstyle}
\numberwithin{equation}{section} \numberwithin{theorem}{section}
\numberwithin{proposition}{section} \numberwithin{lemma}{section}
\numberwithin{corollary}{section}
\numberwithin{definition}{section} \numberwithin{remark}{section}
\newcommand{\ren}{\mathbb{R}^N}
\newcommand{\re}{\mathbb{R}}
\newcommand{\n}{\nabla}
\newcommand{\p}{\partial}
\newcommand{\iy}{\infty}
\newcommand{\pa}{\partial}
\newcommand{\fp}{\noindent}
\newcommand{\ms}{\medskip\vskip-.1cm}
\newcommand{\mpb}{\medskip}
\newcommand{\AAA}{{\bf A}}
\newcommand{\BB}{{\bf B}}
\newcommand{\CC}{{\bf C}}
\newcommand{\DD}{{\bf D}}
\newcommand{\EE}{{\bf E}}
\newcommand{\FF}{{\bf F}}
\newcommand{\GG}{{\bf G}}
\newcommand{\oo}{{\mathbf \omega}}
\newcommand{\Am}{{\bf A}_{2m}}
\newcommand{\CCC}{{\mathbf  C}}
\newcommand{\II}{{\mathrm{Im}}\,}
\newcommand{\RR}{{\mathrm{Re}}\,}
\newcommand{\eee}{{\mathrm  e}}
\newcommand{\LL}{L^2_\rho(\ren)}
\newcommand{\LLL}{L^2_{\rho^*}(\ren)}
\renewcommand{\a}{\alpha}
\renewcommand{\b}{\beta}
\newcommand{\g}{\gamma}
\newcommand{\G}{\Gamma}
\renewcommand{\d}{\delta}
\newcommand{\D}{\Delta}
\newcommand{\e}{\varepsilon}
\newcommand{\var}{\varphi}
\newcommand{\lll}{\l}
\renewcommand{\l}{\lambda}
\renewcommand{\o}{\omega}
\renewcommand{\O}{\Omega}
\newcommand{\s}{\sigma}
\renewcommand{\t}{\tau}
\renewcommand{\th}{\theta}
\newcommand{\z}{\zeta}
\newcommand{\wx}{\widetilde x}
\newcommand{\wt}{\widetilde t}
\newcommand{\noi}{\noindent}
\newcommand{\uu}{{\bf u}}
\newcommand{\xx}{{\bf x}}
\newcommand{\yy}{{\bf y}}
\newcommand{\zz}{{\bf z}}
\newcommand{\aaa}{{\bf a}}
\newcommand{\cc}{{\bf c}}
\newcommand{\jj}{{\bf j}}
\newcommand{\ggg}{{\bf g}}
\newcommand{\UU}{{\bf U}}
\newcommand{\YY}{{\bf Y}}
\newcommand{\HH}{{\bf H}}
\newcommand{\GGG}{{\bf G}}
\newcommand{\VV}{{\bf V}}
\newcommand{\ww}{{\bf w}}
\newcommand{\vv}{{\bf v}}
\newcommand{\hh}{{\bf h}}
\newcommand{\di}{{\rm div}\,}
\newcommand{\ii}{{\rm i}\,}
\newcommand{\inA}{\quad \mbox{in} \quad \ren \times \re_+}
\newcommand{\inB}{\quad \mbox{in} \quad}
\newcommand{\inC}{\quad \mbox{in} \quad \re \times \re_+}
\newcommand{\inD}{\quad \mbox{in} \quad \re}
\newcommand{\forA}{\quad \mbox{for} \quad}
\newcommand{\whereA}{,\quad \mbox{where} \quad}
\newcommand{\asA}{\quad \mbox{as} \quad}
\newcommand{\andA}{\quad \mbox{and} \quad}
\newcommand{\withA}{,\quad \mbox{with} \quad}
\newcommand{\orA}{,\quad \mbox{or} \quad}
\newcommand{\atA}{\quad \mbox{at} \quad}
\newcommand{\onA}{\quad \mbox{on} \quad}
\newcommand{\ef}{\eqref}
\newcommand{\mc}{\mathcal}
\newcommand{\mf}{\mathfrak}

\newcommand{\ssk}{\smallskip}
\newcommand{\LongA}{\quad \Longrightarrow \quad}
\def\com#1{\fbox{\parbox{6in}{\texttt{#1}}}}
\def\N{{\mathbb N}}
\def\A{{\cal A}}
\newcommand{\de}{\,d}
\newcommand{\eps}{\varepsilon}
\newcommand{\be}{\begin{equation}}
\newcommand{\ee}{\end{equation}}
\newcommand{\spt}{{\mbox spt}}
\newcommand{\ind}{{\mbox ind}}
\newcommand{\supp}{{\mbox supp}}
\newcommand{\dip}{\displaystyle}
\newcommand{\prt}{\partial}
\renewcommand{\theequation}{\thesection.\arabic{equation}}
\renewcommand{\baselinestretch}{1.1}
\newcommand{\Dm}{(-\D)^m}

\title
{\bf Local bifurcation-branching analysis \\of global and
``blow-up" patterns \\for a fourth-order thin film equation}

\author{P.~\'Alvarez-Caudevilla and V.A.~Galaktionov}

\address{Centro di Ricerca Matematica Ennio De Giorgi,
Scuola Normale Superiore, 56100, Pisa, ITALY}
\email{alvcau.pablo@gmail.com}

\address{Department of Mathematical Sciences, University of Bath,
 Bath BA2 7AY, UK}
\email{vag@maths.bath.ac.uk}


\thanks{The first author is supported by the Ministry of Science and Innovation of
Spain under the grant MTM2009-08259 and the Postdoctoral
Fellowship--2008-080.}

\date{\today}





\begin{abstract}

 Countable families of {\em global-in-time} and {\em blow-up} similarity
sign-changing patterns
of the Cauchy problem for the fourth-order \emph{thin film
equation} (TFE--4)
 \begin{equation*}
 \tex{ u_t=-\n \cdot (|u|^n \n \D u) \inB \ren \times \re_+ \whereA n>0,}
  \end{equation*}
are studied.
 The similarity solutions are of standard ``forward" and ``backward"
 forms
  $$
   \tex{
 u_{\pm}(x,t)= (\pm t)^{-\a} f(y), \quad y=x/(\pm t)^\b, \quad \b= \frac {1-\a n}4,
  \quad \pm t > 0,
 \quad \mbox{where $f$ solve}
  }
  $$
   \be
   \label{02}
   \tex{
 \BB^\pm_n (\a,f) \equiv - \n \cdot (|f|^n \n \D f) \pm \b y \cdot \n f
 \pm
 \a f=0 \inB \ren,
 }
 \ee
 and $\a \in \re$ is a parameter (a ``nonlinear eigenvalue").
The sign ``$+$", i.e., $t>0$, corresponds to global asymptotics as
$t \to +\iy$, while ``$-$" ($t<0$) yields blow-up limits $t \to
0^-$ describing  possible  ``micro-scale" (multiple zero)
structures of solutions of the PDE.

 To get a countable set of nonlinear pairs $\{f_\g,
 \a_\g\}$, a bifurcation-branching analysis is performed by using a
 homotopy path $n \to 0^+$ in \ef{02}, where $\BB_0^\pm(\a,f)$
 become associated with a pair $\{\BB,\BB^*\}$ of linear
 non-self-adjoint
 operators
  $$
  \tex{
   \BB= - \D^2 + \frac 14\, y \cdot \n + \frac N4 \, I \andA
 \BB^* =- \D^2 - \frac 14\, y \cdot \n
 \quad (\mbox{so} \,\,\,(\BB)^*_{L^2}=\BB^*),
 }
 $$
 which are known to possess a discrete real spectrum,
 $\s(\BB)=\s(\BB^*)= \{ \l_ \g= - \frac{|\g|}4\}_{|\g| \ge 0}$
 ($\g$ is a multiindex in $\ren$).
 These operators occur after corresponding global and blow-up
 scaling of the classic {\em bi-harmonic equation} $u_t= - \D^2 u$.
 This allows us to trace out the origin of a countable family of
 $n$-branches of nonlinear eigenfunctions by using simple or
 semisimple eigenvalues of the linear operators $\{\BB,\BB^*\}$
 leading to important properties of oscillatory sign-changing
 nonlinear patterns of the TFE, at least, for small $n>0$.
\end{abstract}

\maketitle

\section{Introduction: TFE--4 and ``adjoint" nonlinear eigenvalue problems}
 \label{S1}

\subsection{The model, two classes of similarity solutions, and main problems}

\noindent In this paper, we study
 {\em global asymptotic behaviour} (as $t \to +\iy$)
 and
{\em finite-time blow-up behaviour} (as $t \to T^-<+\iy$)
of solutions of the fourth-order semilinear \emph{thin film
equation} (TFE-4)
\begin{equation}
\label{tfe}
    u_{t} = -\nabla \cdot(|u|^{n} \nabla \D u)
 \quad \mbox{in} \quad \ren \times \re_+, \quad n>0,
\end{equation}
where $\n={\rm grad}_x$ and $\D=\n \cdot \n$ stands for the
Laplace operator in $\ren$.


Before describing the main application of such a nonlinear
higher-order PDE model, which has become widely known in the last
decades, we state the main goal of the paper. We study
  similarity solutions of \ef{tfe} of  two ``forward" and Sturm's ``backward"  types:

 \ssk

  (i) {\em global
 similarity patterns} for $t \gg 1$, and

 \ssk

 (ii) {\em blow-up
 similarity ones} with the finite-time behaviour as $t \to
 T^-<\iy$.

 \ssk

 Both classes of such particular solutions of the TFE--4 \ef{tfe}
 can be written in the joint form as follows (here, the blow-up time $T=0$ for solutions in (ii)):
  \be
  \label{upm}
   \tex{
 u_{\pm}(x,t)= (\pm t)^{-\a} f(y), \quad y=x/(\pm t)^\b \,\,\, \mbox{for}
  \,\,\, \pm t > 0 \whereA \b= \frac {1-\a
 n}4,
  }
  \ee
  and  similarity profiles $f(y)$ satisfy the following {\em
  nonlinear eigenvalue problems}, resp.,
   \be
   \label{eigpm}
    \fbox{$
   \tex{
   {\bf (NEP)_\pm}: \quad
 \BB^\pm_n (\a,f) \equiv - \n \cdot (|f|^n \n \D f) \pm \b y \cdot \n f \pm
 \a f=0 \inB \ren.
 }
  $}
 \ee
 Here, $\a \in \re$ is  a parameter, which stands in both cases for admitted real
  {\em nonlinear eigenvalues}.
Thus, the sign ``$+$", i.e., $t>0$, corresponds to global
asymptotic as $t \to +\iy$, while ``$-$" ($t<0$) yields blow-up
limits $t \to T=0^-$ describing  a ``micro-scale" structure of
the PDE. In fact, the blow-up patterns are assumed to describe the
structures of ``multiple zeros" of solutions of the TFE--4. This
idea goes back to Sturm's analysis of  solutions of the 1D heat
equation performed in 1836 \cite{St}; see \cite[Ch.~1]{GalGeom}
for the whole history and applications of these fundamental
Sturm's ideas and two {\em zero set Theorems}.

 Being equipped with proper ``boundary conditions at infinity",
 namely,
  \be
  \label{bc1}
  \mbox{for global case}, \quad \BB^+_n(\a,f):
  \quad f \in C_0(\ren) \quad (f \,\,\,\mbox{is compactly
  supported}), \quad \mbox{and}
  \ee
   \be
  \label{bc2}
  \mbox{for blow-up case}, \quad \BB^-_n(\a,f):
  \quad f(y)\,\,\,\mbox{has a ``minimal growth" as $y \to \iy$},
  \ee
equations \ef{eigpm} become two {\em true} {\em nonlinear
eigenvalue problems} to study, which can be considered as a pair
of mutually ``adjoint" ones. All related aspects and notions used
above, remaining still somehow unclear, will be properly discussed
and specified.

Our goal is to show by using any means that, for small $n>0$,
eigenvalue problems
 \be
 \label{main1}
 \fbox{
 ${\bf (NEP)_\pm}$ admit  countable sets of
 solutions $\Phi^\pm(n)=\{\a_\g^\pm,f^\pm_\g\}_{|\g| \ge 0}$,}
  \ee
  where $\g$ is a multiindex in $\ren$ to numerate   the pairs.

The last question to address is whether these sets
 \be
 \label{main2}
 \mbox{
 $\Phi^\pm(n)$
of nonlinear eigenfunctions are {\em evolutionary complete}},
 \ee
i.e., describe {\em all} possible asymptotics as $t \to +\iy$ and
$t \to 0^-$ (on the corresponding compact subsets in the variable
$y$ in \ef{upm}) in the CP for the TFE--4 \ef{tfe} with bounded
compactly supported initial data.

Our main approach is the idea of
a ``homotopic deformation" of \ef{tfe}  as $n \to 0^+$ and
reducing it to the classic {\em bi-harmonic equation}
\begin{equation}
\label{s1}
    \tex{ u_{t} = -\D^2 u\quad \hbox{in} \quad \re^{N} \times \re_+\,,}
\end{equation}
 for which both problems \ef{main1} and \ef{main2} are solved
 positively by rather standard (but not self-adjoint)  spectral
 theory of linear operators.


\subsection{Main TFE applications: nonnegative and oscillatory
solutions}

It has been well known since the 1980s, when  higher-order parabolic
models began to be  studied more actively, that the TFEs--4
like \eqref{tfe} have many applications arisen particularly in modeling
the spreading of a liquid film along a surface, where $u$ stands
for the height of the film in this context of thin film theory.
Other physical related problems come from lubrication theory,
nonlinear diffusion, flame and wave propagation (the
Kuramoto--Sivashinsky equation and the extended Fisher--Kolmogorov
equation), phase transition at critical Lifshitz points and
bi-stable systems. We refer to a number of key survey and other
papers on TFE theory such as \cite{AGi04, Beck05, BertBern,
BerPugh1, BerPugh2, BK01N, Grun04}; see also
 Peletier--Troy \cite{PT} as a guide to higher-order ODEs and  \cite{EGK1, EGK3}
  for most recent
short surveys and long lists of references concerning physical
derivations of various models, key mathematical results and
applications of TFEs.
 Concerning mathematics of TFEs, one has to refer to the pioneering
 Bernis--Friedman paper \cite{BF1} and \cite{BPW, FB0, CarrT02} for
 the role of source-type similarity solutions of \ef{tfe}.
 On modern existence-uniqueness theory for
 the 1D TFE (for FBP setting), see \cite{GiacOtto08}, \cite[\S~6]{EGK2}, and references therein.

\par

It should be pointed out that most of the results cited above
are associated with {\em nonnegative solutions} of a {\em
free-boundary problem} (FBP) for the TFE--4 \ef{tfe}, while
 currently this equation
 is written  for {\em solutions of changing sign}. Moreover, as mentioned above,
   the development of general approaches
  to nonnegative solutions of the FBP began with the work of
 Bernis--Friedman \cite{BF1} in 1990 with such solutions having a most relevant physical
motivation and applications.

 The study of oscillatory solutions of changing sign for the TFE--4
is more recent; see \cite{BW06, EGK2, EGK4} and references therein.
It was shown in \cite{EGK1}--\cite{EGK4} (see also \cite{PV} as
the most recent publication) that such solutions
 can be attributed to the {\em Cauchy problem} (CP) in $\ren \times \re_+$, rather
than a FBP, posed in a bounded domain with moving free
boundaries. The study of the Cauchy problem is interesting from
both points of view in some biological applications as well
as  its clear  mathematical interest in PDE theory.
 We refer to \cite{PV}, where more details on the CP setting are
 available.

In this connection,
another pioneering paper of Bernis--McLeod
in 1991 \cite{BMc91} should be mentioned, where  existence and uniqueness of first
{\em three} oscillatory source-type solutions of the Cauchy
problem  for the {\em fourth-order porous medium equation} (PME--4)
\begin{equation}
 \label{pme4}
    u_t =-(|u|^{n} u)_{xxxx}\quad \hbox{in} \quad \re \times \re_+\,,
\end{equation}
are studied. Here, unlike \ef{tfe}, equation \ef{pme4} contains a
monotone operator in the metric of $H^{-2}(\re)$. By classic
theory of monotone operators \cite{LIO}, the CP for \ef{pme4} with
compactly supported initial data $u_0$ admits a unique weak
solution that is oscillatory close to the interfaces for all $n>0$
and evidently for $n=0$, where it becomes the {\em
 bi-harmonic equation} \ef{s1}, with an oscillatory kernel of the
fundamental solution; see below.

 For $n>0$, such classes of the
so-called ``oscillatory solutions" of TFE--4 \eqref{tfe} is
difficult to tackle rigorously, and even their ODE representatives
(in the radial geometry) exhibit several surprises in trying to
describe sign-changing features close to interfaces, \cite{EGK1}.
Indeed, the CP in $\ren \times \re_+$ shows compactly supported
blow-up patterns, which have infinitely many oscillations near the
interfaces and exhibit maximal regularity there (consult
\cite{EGK1} for further details). It turns out that, for a better
understanding of such singularity oscillatory properties of
solutions of \eqref{tfe}, it is quite fruitful to consider the
homotopic limit $n \to 0^+$, thanks to the spectral theory
developed for the pair $\{{\bf B},{\bf B}^*\}$ in \cite{EGKP} for
rescaled operators where $n=0$. Thus, here we perform a homotopic
approach, more rigorous than before, in order to obtain such
interplay between the CP for the TFE-4 \ef{tfe} and the
{bi-harmonic equation} \ef{s1}.

\subsection{Our  approach, problem setting, and layout of the paper}

Before giving a description of our  approaches, it is worth
mentioning again that  TFE theory  for \emph{free boundary
problems} (FBPs) with nonnegative solutions is well understood
nowadays (at least in 1D). The FBP setting assumes posing three
standard boundary conditions at the interface, and such a theory
has been developed in many papers since 1990. The mathematical
formalities and general setting of the CP is still not fully
developed and a number of problems remain open. In fact, the
concept of proper solutions of the CP is still partially obscure,
and moreover it seems that any classic or standard notions of
weak-mild-generalized-... solutions fail in the CP setting.

  Various ideas  associated with extensions
of smooth order-preserving semigroups are well known to be
effective for second-order nonlinear parabolic PDEs, when such a
construction is naturally supported by the maximum principle. The
 analysis of higher-order equations such as \eqref{tfe} is much
harder than the corresponding  second-order equations or those in
divergent form (cf. \ef{pme4}) because of the lack of the maximum
principle, comparison, order-preserving, monotone,  and potential
properties of the quasilinear operators involved.

It is clear that the CP for the \emph{bi-harmonic equation} \eqref{s1}
is well-posed and has a unique solution given by the convolution
 \be
 \label{b11}
 u(x,t)=b(x-\cdot,t)\, * \, u_0(\cdot),
  \ee
   where $b(x,t)$ is the fundamental solution of the operator $D_t + \Delta^2$. By the
   apparent
    connection
between \eqref{tfe} and \eqref{s1} (when $n=0$), intuitively at
least,
 this analysis provides us with a way to understand the CP for the TFE-4 by
 using  the fact that the proper
 solution of the CP for \eqref{tfe}, with the same initial data $u_0$, is that one which converges to the corresponding
unique solution of the CP for \eqref{s1}, as $n\rightarrow 0$.
Thus, we shall use the patterns occurring for $n=0$, as branching
points of nonlinear eigenfunctions, so some extra detailed
properties of this linear flow will be necessary.

\ssk

In Section \ref{S3}, we, more carefully, introduce two classes of
similarity solutions (the so-called nonlinear eigenfunctions),
while Section \ref{S4} is devoted to necessary properties of the
spectral pair $\{\BB,\BB^*\}$ of linear differential operators
that occur at $n=0$.

Our further analysis is as follows:


\subsection{Local bifurcation-branching analysis for global solutions
 (Section \ref{S4}): first operator theory discussion}
 \label{S1.4}

In the first part of this work, we perform a local
bifurcation-branching analysis with respect to the
continuation parameter $n >0$, when that parameter is  small
enough. Thus, we obtain the bifurcation of solutions of the
non-gradient equation $\eqref{eigpm}_+$ from the branch of the
corresponding eigenfunctions of a rescaled linear operator.  This
yields some  information and properties of the global in time
similarity  solutions $\ef{upm}_+$  of the TFE--4 \eqref{tfe}.

\par

The linear elliptic equation occurring at $n=0$,
\begin{equation}
\label{i4}
 \tex{
    {\bf B}F \equiv -\D_y^2 F + \frac{1}{4}\,y \cdot \nabla_y F +\frac{N}{4}\, F=0
    \quad \hbox{in} \quad \re^{N},\quad \int\limits_{\re^{N}} F(y) \, {\mathrm
    d}y=1,
    }
\end{equation}
 where $F$ is the rescaled kernel of the fundamental solution
 $b(x,t)$ in \ef{b11},
will be pivotal in the subsequent analysis. Indeed, the nonlinear
operator in \ef{eigpm},
\begin{equation}
\label{bf1}
 \tex{
    \BB_n^+(\a,f):= -\nabla \cdot (|f|^n \nabla \D f)
    +\frac{1-\a n}{4}\, y \cdot \nabla f + \a f,
    }
\end{equation}
 can be written in the following equivalent form:
\begin{equation}
\label{bf2}
     \tex{
   \BB_n^+(\a,f)\equiv   -\D^2 f + \frac{1-\a n}{4}\,
    y \cdot \nabla f + \a f+\nabla \cdot ((1-|f|^n) \nabla \D f).}
\end{equation}
Then, the solutions of $\BB_n^+(\a,f)=0$
are regarded as steady states
of the nonlinear evolution equation
\begin{equation}
\label{sf11}
 \tex{
 f_\t= \BB_n^+(\a,f) \inB \ren \times \re_+.
    }
\end{equation}
The bifurcation-branching point from our solutions (for $n=0$)
will be denoted by $(n,f)=(0,\psi_k)$, which is shown to occur for
some values of the nonlinear eigenvalue $\a$ written as $\a=\a_k$
($k=|\b|$ is characterized by a multiindex $\b$ in $\ren$),  and
$\psi_k$ representing the eigenfunctions of the operator ${\bf
B}$, whose expressions will be obtained in detail later on.

Firstly, we shall prove that no bifurcation from the branch of
trivial solutions $(n,f)=(0,0)$ occurs when the parameter $n$
approximates 0. Secondly, an infinite number of branches of
solutions is shown to emanate from the eigenfunctions of the
rescaled linear operator ${\bf B}$. Consequently, this analysis
provides us with a countable family of solutions pairs
\eqref{main1} for the nonlinear equation $\eqref{eigpm}_+$ for
small $n>0$.

According to classic bifurcation theory
\cite{CR, Deim, KZ, VainbergTr}, we denote
\begin{equation}
\label{funct}
    \tex{ \BB_n^+(\a,f)\equiv \mc{F}(n,f):= \mc{L}(\a,n)f+\mc{N}(n,f),}
\end{equation}
and assume that $n$ is the main continuation parameter. Then, in
order to have a branch of solutions emanating from the branch of
trivial solutions $(n,f)=(0,0)$ at certain values of the parameter
$n$ (bifurcation points), the nonlinearity in \ef{bf2}, denoted by
$$\mc{N}(n,f):=\nabla \cdot ((1-|f|^n) \nabla \D f),$$
must fulfill
the following conditions:
\begin{equation}
\label{NL1}
  {\bf (NL):} \quad
    \tex{ \mc{N}(n,0)=0,\quad D_f \mc{N}(n,0)=0 \quad \hbox{for\,\, all}\quad
    n\in\re_+.}
\end{equation}
In other words,
$$\mc{N}(n,f)=o(\left\|f\right\|) \quad \hbox{as} \quad f\rightarrow 0.$$
Recall that, here, $\mc{N}(n,f)$ serves as  a
perturbation of the operator $\BB_n^+$ defined as in \eqref{bf2}.
Thus, under the given assumptions (which are not that easy to pose
in a suitable functional setting, to say nothing of the proof),
the linear operator denoted by
\begin{equation}
 \label{NL1NN}
 \tex{
   \mc{L}(\a,n) :=  -\D^2  + \frac{1-\a n}{4}\,
    y \cdot \nabla + \a I,
        }
\end{equation}
defines an analytic semigroup in the space, where the solutions of
$\eqref{eigpm}_+$ are defined.

Note that, in any case, the
necessary assumptions for the nonlinearity of $\eqref{eigpm}_+$
are far from clearly specified, when $f$ is very close to zero, so
something else must be imposed. Let us note that the condition
 {\bf (NL)} in \ef{NL1}, roughly speaking,  assumes that the functions $f(y)$ are
sufficiently smooth and have ``transversal" zeros with a possible
accumulating point at a finite interface only. Otherwise, if
$f(y)$ exhibits vanishing inside the support at a sufficiently
``thick" nodal set, with many non-transversal zeros, this can
undermine the validity of \ef{NL1}, even in any weak sense.

As customary in nonlinear operator
theory, instead of the differential operators in \ef{funct}, one
has to deal with the equivalent integral equation
 \be
 \label{int1}
f=- (\mc{L}(\a,n)-a I)^{-1} (\mc{N}(n,f)+a f),
 \ee
  where $a>0$ is a parameter to be chosen so that the inverse
  operator (a resolvent value) is a compact one in a weighted space $L^2_\rho(\ren)$; see
  Section \ref{S3}. We will show therein that the
  spectrum of $\mc{L}$ is always discrete and, actually,
 \be
 \label{sp33}
  \tex{
 \s(\mc{L}(\a,n))=\big\{(1-\a n)\big(- \frac k4\big)+\a,
 \,k=0,1,2,...\big\},
 }
  \ee
 so that any choice of $a>0$ such that $a \not \in \s(\mc{L})$ is suitable in \ef{int1}.
 Therefore, in particular, the conditions \ef{NL1} are assumed to be valid in a
 weaker sense associated with the integral operator in \ef{int1}.

Let us explain  why a certain ``transversality" of zeros of
possible solutions $f(y)$ is of key importance. As we see from
\ef{bf2}, we have to use the expansion for small $n>0$
 \be
 \label{nn1}
 |f|^n-1 \equiv {\mathrm e}^{n \, \ln |f|}-1= 1+ n \,
 \ln|f|+...-1=n \, \ln|f|+...\,,
  \ee
  which is true pointwise on any set $\{|f| \ge \e_0\}$ for
  an arbitrarily small
 fixed constant $\e_0>0$. However, in a small neighbourhood of any
 {\em zero} of $f(y)$, the expansion \ef{nn1} is no longer true.
 Nevertheless, it remains true in a weak sense provided that this
 zero is sufficiently transversal in a natural sense, i.e.,
 \be
 \label{nn2}
  \tex{
  \frac{|f|^n-1}n \rightharpoonup \ln|f| \asA n \to 0^+
  }
  \ee
  in $L^\infty_{\rm loc}$, since then
 the singularity $\ln |f(y)|$ is not more than {\em logarithmic}
 and, hence, is locally integrable in \ef{int1}. Equivalently we are dealing with the limit
 $$n \ln^2 |f| \rightharpoonup 0, \quad \hbox{as} \quad n \downarrow 0^+,$$
 at least in a very weak sense, since by the expansion \eqref{nn1} we have that
 $$
   \tex{
    \frac{|f|^n-1}n - \ln |f| = \frac 12 \, n \, \ln^2 |f|+...
    \, .
    }
    $$

   Note also that actually we deal, in \ef{int1}, with an easier
  expansion
   \be
   \label{nn3}
   (|f|^n-1) \n \D f = (n \, \ln|f|+...) \n \D f,
    \ee
    so that even if $f(y)$ does not vanish transversally at a
    zero surface, the extra multiplier $\n \D f(y)$ in \ef{nn3},
    which is supposed to vanish as well, helps to improve the
    corresponding weak convergence.
Furthermore, it is seen from \ef{bf2} that, locally in space
variables, the operator in \ef{int1} (with $a=0$ for simplicity)
acts like a standard Hammerstein--Uryson compact integral operator
with a sufficiently smooth kernel:
 \be
 \label{nn4}
 f \sim (\n \D)^{-1}[(|f|^n-1) \n \D f].
  \ee

 Therefore, in order to justify our asymptotic branching analysis,
 one needs in fact to introduce such a functional setting and a class of solutions
 $$\mathcal P=\{f=f(\cdot,n): \,\,
  f \in H^4_{\rho}(\ren)\},$$
  for which:
 \be
 \label{nn5}
  \tex{
  {\mathcal P}: \quad \,\,
   (\n \D)^{-1}\big( \frac{|f|^n-1}n \,\n \D f\big)  \to (\n \D)^{-1}(\ln \, |f|
 \n \D f)
    \asA n
   \to 0^+
   }
   \ee
a.e. This is the precise statement on the regularity of possible solutions,
which is necessary to perform our asymptotic branching analysis.
 In 1D or in the radial
 geometry in $\ren$, \ef{nn5} looks rather constructive. However, in general, for
 complicated solutions with unknown types of compact supports in $\ren$,
 functional settings that can guarantee \ef{nn5}
 are not achievable still.
We mention again that, in particular, our formal analysis aims to
establish structures of difficult multiple zeros of the nonlinear
eigenfunctions $f_\g(y)$, at which \ef{nn5} can be violated, but
hopefully not in the a.e. sense.

\ssk

To study nonlinear integral operators it is necessary to construct a function space in
which the integral operator possesses favorable properties (continuity, compactness).
Indeed, one can apply the classical fixed point principles of Schauder's type to an operator
acting between suitable Banach spaces. In this situation we can assert the existence of such fixed
points establishing the continuity and boundedness of the integral operator. To do so, thanks to classical nonlinear integral operator
theory we should impose the continuity of the kernel function involved in our integral operator \eqref{nn4}.

Within the previous context, let us observe that the integral equation with a Hammerstein-Uryson operator-type \eqref{nn4} is equivalent
to the integral equation \eqref{int1}, for
which we know that the inverse operator $(\mc{L}(\a,n)-a I)^{-1}$ is compact. Indeed,
by the spectral theory described in Section \ref{S3},
we are able to deduce that the operator $\mc{L}$ is defined between
two exponential weighted spaces. Hence, it looks like to ascertain the existence of such fixed points for \eqref{int1} and,
equivalently for \eqref{nn4}, the suitable Banach spaces (that will provide us with the existence of solutions of the original
equation \eqref{bf1}) are precisely those exponential weighted spaces, together with the assumption of continuity of the kernels
involved in the equivalent integral equations \eqref{int1} and \eqref{nn4}.

In addition, we would like to mention that for the study of elliptic problems of order $2m$ by Schauder's inversion procedure the suitable
Banach spaces could be the typical pairs consisted of either H\"{o}lder spaces
$$\big(\mc{C}_\rho^{2m,\a}(\ren), \mc{C}_\rho^{0,\a}(\ren)\big),\quad \hbox{with} \quad  0<\a <1,$$
or as remarked in the previous paragraph Sobolev spaces
$$\big(W_\rho^{2m,p}(\ren), L_\rho^{p}(\ren)\big), \quad \hbox{with}\quad  1<p<\infty.$$

The particular
weights assumed for those Banach spaces should be consistent with the exponential ones obtained in section \ref{S3}. Thus, this enables us to
obtain {\em a priori} estimates for the solutions of the original nonlinear equation \eqref{bf1} and provides us with the compactness
of the integral operators involved in \eqref{int1} and \eqref{nn4}.


\subsection{Regularity convention}

    Overall, we observe that, unlike classic
existence bifurcation-branching theory \cite{Deim, KZ,
VainbergTr}, where sufficiently smooth expansions are used, the
present singular one \ef{nn1} dictates a special functional
setting in a subset ${\mathcal P}$ of functions (admissible
solutions), for which \ef{nn5} should be valid {\em a priori}.
 In particular,
     such an analysis of the integral
    equation \ef{int1} will always require some deep knowledge of
    admissible structure of proper solutions $f(y)$ near zero (nodal)
    sets (also unknown), which we are still not aware of.
    Recall that, as our main goal, the present branching
    analysis is going to give us a first understanding of such delicate
    properties via the known eigenfunctions of the linear rescaled operators to appear at $n=0^+$.

Thus,  since these necessary nodal properties of possible
solutions $f(y)$ are unknown entirely rigorously, we perform our
analysis under the following {\bf regularity convention}:
\begin{equation*}
    \begin{array}{c} \text{{\em we assume that, regardless of the strong degeneracy
    of the nonlinear elliptic operator }} \\ \text{{\em involved, the problems under
    consideration in both integral and differential forms}} \\ \text{{\em admit
    sufficiently regular expansions of solutions in small $n>0$}} \\
    \text{{\em in the functional class, for which \eqref{nn5} holds.}}
    \end{array}
\end{equation*}

\ssk

 As usual in bifurcation theory, the hypothesis for this to be
valid is formulated
 for the equivalent integral
representation of the operators, though, for simplicity, we
perform the $n$-expansion analysis in the simpler (but indeed
equivalent) differential form. Overall, currently, we honestly do
not think that our analysis can be justified more rigorously than
that: technicalities to arise can be extreme and a full prove
truly illusive. However, in Appendix A, we show that a suitable
justification of the branching is indeed achievable provided that
clear transversality of a.a. zeros of linear eigenfunctions of
$\BB$ is known. Nevertheless, the problem of the actual existence
of nonlinear eigenfunctions for small $n>0$ remains open still.

\subsection{Further branching discussion}

Now, once we have discussed the principal difficulties, which have arisen, we carry out a
local bifurcation analysis close to $n=0$.
Then, the change of stability from the branch of trivial solutions
$(n,f)=(n,0)$ should be determined by the spectrum of the
linearization $\mc{L}(\a,n)$ and the assumptions of the
nonlinearity. Therefore, a bifurcation would take place at some
values of the parameter $n$ if every neighbourhood of
$(n,f)=(n,0)$ in $\re \times \mc{C}(\bar\O)$ contains a nontrivial
solution $(n,f)$ of $\eqref{eigpm}_+$ under the assumptions
imposed for the linear and nonlinear part of the equation.
However, as will be proved later, such a bifurcation from the
branch of trivial solutions $(n,f)=(n,0)$ never happens at the
value of the parameter $n=0$. Hence, we obtain a branching
from the points $(n,f)=(0,\psi_k)$ only; these arguments are fully
consistent with more particular results obtained earlier in
\cite{PV}.

Moreover, since the bifurcation from the branch  of trivial
solutions depends on the eigenvalues of the linear operator
\eqref{NL1NN}, we believe that such a bifurcation does not exist
at all (any proof is also very difficult). Indeed, after some
rescaling of the type $y\mapsto a y$, with $a=(1-\a n)^{-1/4}$, it
turns out that the spectrum of \eqref{NL1NN} is directly related
to the spectrum of the linear operator ${\bf B}$. This suggests
that no bifurcation from the branch of trivial solutions ever
happens. Therefore, as also discussed in \cite{PV}, we conjecture
that if one wants to ascertain the global bifurcation analysis for
these similarity TFEs-4, a new, different operator theory approach
must be used.

\par

Next, let us obtain the values of the parameter $\a$ for which the bifurcation-branching
phenomena occurs.
The spectrum of the operator ${\bf B}$ in \eqref{i4}, which
appears for $n=0$ in $\ef{eigpm}_+$, is already well known and
will be explained  in detail in the next sections. This has the
form
\begin{equation*}
 \tex{
    \s({\bf B})=\big\{\l_k =: -\frac{k}{4}\,,\,k=0,1,2,...\big\}.
    }
\end{equation*}
Moreover, for $k=0$, i.e., for the first eigenvalue-eigenfunction
pair $\{\a_0(n),f_0\}$, from the conservation of mass condition,
denoting by $M(t)$ the mass of the solutions of \eqref{tfe}, we
have that (here $\O$ is the rescaled support of $f(y)$,
however, for the CP, one can put $\O=\ren$)
\begin{equation*}
    \tex{ M(t):=\int\limits_\O u(x,t) \, {\mathrm d}x=t^{-\a} \int\limits_\O f(\frac{x}{t^\b})
    \, {\mathrm d}x = t^{-\a+\b N} \int\limits_\O f(y)
    \, {\mathrm d}y. }
\end{equation*}
 This yields the exact values
\begin{equation}
 \label{al0}
    \tex{
-\a+\b N=0 \LongA
     \a_0(n)=\frac{N}{4+Nn} \andA \b_0(n)=\frac{1}{4+Nn}.}
\end{equation}
However, the construction of the first eigenfunction $f_0(y)$ is
not that straightforward even in 1D; see \cite[\S~7]{EGK2},
where its oscillatory properties cease to exist at a {\em
heteroclinic} bifurcation calculated numerically as
 $$
 n_{\rm h}=1.7587...\,.
  $$
   It is worth
mentioning that, fortunately, for all $n \in (0,1)$ (this interval
is of particular interest in what follows), both the existence and
the uniqueness of $f_0(y)$ follow from the results of
\cite{BMc91}, since, rather surprisingly, source-type similarity
profiles for \ef{tfe} and \ef{pme4} are reduced to each other with
the parameter change $n \mapsto \frac n{n+1}$; see a precise
statement in \cite[Prop.~9.1]{EGK2}.

Thus, it turns out that, when the parameter $n$ approximates zero,
we obtain according to \ef{al0}
 $$
  \tex{
 \a_0(0)=\frac{N}{4},
 }
 $$
  so the solutions of
$\eqref{eigpm}_+$ seem to approach the first eigenfunction
$\psi_0$ associated with the first eigenvalue of the operator
${\bf B}$, i.e., corresponding to $\l_0=0$. However, that
approximation for the solutions of $\eqref{eigpm}_+$ should also
be extended to the eigenfunctions $\psi_k$, for any $k \geq 1$,
when the parameter $\a$  reaches the following values:
\begin{equation}
\label{bf4}
 \tex{
     \a_k(0) := -\l_k + \frac{N}{4} \quad \hbox{for any} \quad k=1,2,\ldots,
     }
\end{equation}
where $\l_k$ are the eigenvalues of the operator ${\bf B}$,
so that
\begin{equation*}
 \tex{
    \a_0(0)= \frac{N}{4}, \; \a_1(0)= \frac{N+1}{4},\; \a_2(0)= \frac{N+2}{4},\ldots ,
    \a_k(0)= \frac{N+k}{4}\ldots \,.
     }
\end{equation*}
Then, we introduce the next expression for the parameter $\a$
\begin{equation}
\label{i52}
    \tex{ \a_k(n):= \frac{N}{4+Nn}-\l_k.}
\end{equation}
Hence, due to the necessary assumptions, the structure of the
bifurcating-branching set emanating at $(n,f)=(0,\psi_k)$ depends
on the spectral theory  for the operator \eqref{i4}. For the first
eigenvalue, since $\l_0=0$ is simple, we can ascertain accurately
the local bifurcation-branching. Then, at least for sufficiently
small $n$'s,  the bifurcation-branching is locally a $\mc{C}^1$
curve, which can be parameterized as $s \mapsto (n_0(s),f(s))$ in
$\re \times \mc{C}(\bar\O)$ with
\begin{equation*}
    \tex{ (n_0(0),f(0))=(0,\psi_0),\quad f'(0)=\Phi_0, \quad \Phi_0 \in
    Y_0},
\end{equation*}
where $':=\frac{\mathrm d}{{\mathrm d}s}$, emanating from the
eigenfunction $\psi_0$ at the $n=0$ in the direction of the space
$Y_0$ orthogonal to the eigenspace $\ker \,{\bf B}$, with
$\psi_0=F$ (the rescaled fundamental kernel of $b(x,t)$) being the
eigenfunction associated with the eigenvalue $\l_0=0$.
\par

However, in the case when the multiplicity of the eigenvalues
$\l_k$ with $k\geq 1$ is higher (bigger than 1), we obtain that
the continua emanating at $(n,f)=(0,\psi_k)$ are tangent to the
manifolds $Y_k$, orthogonal to $\ker\big({\bf B} + \frac{k}{4}\,
I\big)$.  And, hence, we might have more than one direction of
bifurcation-branching, depending on certain values related to the
eigenfunctions which generate the eigenspace. This certainly
agrees with the work of Rabinowitz \cite{R}, in which for
potential operators and bifurcation from the branch of trivial
solutions, one of the next alternatives for the bifurcation
structure must be obtained:

 (i) for the value of the parameter
where the bifurcation takes place, the trivial solution is not
isolated; or

(ii) for any other value of the parameter in one-sided
neighbourhood of the bifurcation point, there are at least two
nontrivial solutions; or

(iii) for any other value of the parameter in a neighbourhood of
the bifurcation point, at least one nontrivial solution exists.

 In general, the question
about how many precise branches bifurcates for any $k \geq 1$
remains an open problem, though we think it is very related to the
dimensions of the eigenspaces. As far as we know, only partial and
very specific results have been obtained for non-variational
problems with higher multiplicities.



 \subsection{Blow-up patterns via branching theory (Section \ref{S5})}

This is a natural counterpart of the global similarity analysis of
PDEs in the limits as $t \to \iy$. We next consider {\em blow-up
limits} as $t \to T^-<\iy$, or  $t \to 0^-$ as in $\ef{upm}_-$,
where $T=0$. We thus perform a detailed and systematic  analysis
of the {\em blow-up} similarity  solutions. This is done again by
using the {homotopic approach} as $n \to 0^+$ via branching
theory, but this time based on the Lyapunov--Schmidt methods in order to obtain
relevant results and properties for the solutions of the
self-similar equation $\ef{eigpm}_-$.
This homotopic-like approach is based upon the spectral properties
of the adjoint (to the $\BB$ above) operator
 \be
 \label{ad55}
  \tex{
  \BB^*=- \D^2 - \frac 14 \, y \cdot \n
  \withA \s(\BB^*)=\big\{\l_\b=- \frac{|\b|}4,\,\,
  |\b|=0,1,2,...\big\},
  }
  \ee
  which occurs after blow-up scaling of
  the linear counterpart \eqref{s1} of
the TFE-4 \eqref{tfe} for $n=0$.
 Note that \ef{ad55} admits a complete and closed set of
 eigenfunctions being {\em generalized Hermite polynomials}, which
 exhibit finite oscillatory properties.

It is curious that, in \cite{EGK1}, blow-up similarity analysis of
the related {\em unstable} TFE--4
 did not detect any stable
oscillatory behaviour
 of solutions  near the interfaces of the radially symmetric
associated equation. All the blow-up patterns turned out to be
nonnegative, which is a specific feature of the PDE under
consideration therein.
 This does not mean that blow-up similarity solutions of the CP do not change sign near the interfaces
 or inside the support.
  Actually, it was  pointed out
   that local sign-preserving property could be attributed only to the blow-up ODE and not to
   the whole PDE \eqref{tfe}. Hence, the possibility of
having oscillatory solutions cannot be ruled out for every case. Indeed,
 thanks to the polynomial expressions of the eigenfunctions for the operator
  ${\bf B}^*$, we have, in particular, that  the first eigenfunction
   $\psi^*(y)=1$ is not oscillatory, but for some other
   eigenfunctions we shall expect sign-changing behaviour.

Then, this homotopy study exhibits a typical difficulty concerning
the desired structure of the transversal zeros of solutions, at
least for small $n>0$. Proving such a transversality zero property
is still a difficult open problem, though qualitatively, this was rather
well understood in 1D and radial geometry, \cite{EGK1}.




\subsection{TFE: FBP and CP problem settings}

We recall that, for  both the FBP and the CP of \eqref{tfe}, the
solutions are assumed to satisfy standard free-boundary
conditions:
\begin{equation}
\label{i5}
    \left\{\begin{array}{ll}  \tex{ u=0,} & \tex{ \hbox{zero-height,} }  \\
    \tex{ \nabla u=0,} & \tex{ \hbox{zero contact angle,} } \\
    \tex{ -{\bf n} \cdot \nabla (\left|u\right|^{n}  \D u)=0,}  &
    \tex{ \hbox{conservation of mass (zero-flux)} } \end{array} \right.
\end{equation}
at the singularity surface (interface) $\Gamma_0[u]\equiv \p\O$, which is the lateral boundary
of
\begin{equation*}
    \tex{ \hbox{supp} \;u \subset \re^{N} \times \re_+,\quad N \geq 1\,,}
\end{equation*}
where ${\bf n}$ stands for the unit outward normal to
$\Gamma_0[u]$, which is assumed to be sufficiently smooth (the treatment of such hypotheses is
not any goal of this paper).
For  smooth interfaces, the condition on the flux can be read as
\begin{equation*}
    \lim_{\hbox{dist}(x,\Gamma_0[u])\downarrow 0}
    -{\bf n} \cdot \nabla (|u|^{n}  \D u)=0.
\end{equation*}

For the FBP, dealing with {\em nonnegative} solutions, this
setting is assumed to define a unique solution. However, this
uniqueness result is known in 1D only; see \cite{GiacOtto08},
where the interface equation was included into the problem
setting. We also refer to \cite[\S~6.2]{EGK2}, where a ``local"
uniqueness is explained via {\em von Mises} transformation, which
fixes the interface point. For more difficult, non-radial
geometries in $\ren$, there is no hope of getting any uniqueness for
the FBP, in view of possible very complicated shapes of supports
leading to various ``self-focusing" singularities of interfaces at
some points, which can dramatically change the required regularity
of solutions.

\ssk

For the CP, the assumption on nonnegativity is got rid of, and
solutions become oscillatory close to interfaces. It is then key
that
  the solutions are expected to be
``smoother" at the interface than those for the FBP, i.e., \ef{i5}
are not sufficient to define their regularity. These {\em maximal
regularity} issues for the CP, leading to oscillatory solutions,
are under scrutiny in \cite{EGK2}; see also \cite{PV}, as the most
recent source of such a study.

\ssk

Next, denote by
\begin{equation*}
 \tex{
    M(t):=\int  u(x,t) \, {\mathrm d}x
    }
\end{equation*}
the mass of the solution, where integration is performed over
smooth support ($\ren$ is allowed for the CP only). Then,
differentiating $M(t)$ with respect to $t$ and applying the
divergence theorem (under natural regularity assumptions on
solutions and free boundary), we have that
\begin{equation*}
 \tex{
  J(t):=  \frac{{\mathrm d}M}{{\mathrm d}t}= -
  \int\limits_{\Gamma_0\cap\{t\}}{\bf n} \cdot \nabla
     (|u|^{n}  \D u )\, .
     }
\end{equation*}
The mass is conserved if $ J(t) \equiv 0$,
which is assured by the flux condition in \eqref{i5}.

The problem is completed
with bounded, smooth, integrable, compactly supported initial data
\begin{equation}
\label{i6}
    \tex{ u(x,0)=u_0(x) \quad \hbox{in} \quad \Gamma_0[u] \cap \{t=0\}.}
\end{equation}

\ssk

In the CP for \eqref{tfe} in $\ren \times \re_+$, one needs to
pose bounded compactly supported initial data \eqref{i6}
prescribed in $\ren$. Then,  under the same zero flux condition at
finite interfaces (to be established separately), the mass is
preserved; however smoother regularity properties of solutions
require a separate study/understanding; see \cite{EGK2} for some
results.

\setcounter{equation}{0}
\section{Self-similar solutions:  two nonlinear eigenvalue problems}
 \label{S2}

\subsection{Global similarity solutions}

 We now more carefully derive the problem for global  self-similar
solutions of  \eqref{tfe}, which  occur due to its natural
scaling-invariant  nature.
Namely, using the following scaling in \eqref{tfe}:
\begin{equation}
 \label{sf1}
  \begin{matrix}
    x:= \mu \bar x,\quad t:= \l \bar t,\quad u:= \nu \bar u, \quad
    \mbox{with} \ssk\ssk\\
 \tex{
    \frac{\p u}{\p t}= \frac{\nu}{\l} \frac{\p \bar u}{\p \bar t},\quad
    \frac{\p u}{\p x_{i}}= \frac{\nu}{\mu} \frac{\p \bar u}{\p \bar x_{i}},\quad
    \frac{\p^2 u}{\p x_{i}^2}= \frac{\nu}{\mu^2} \frac{\p^2 \bar u}{\p
    \bar x_{i}^2},
     }
      \end{matrix}
\end{equation}
 and substituting those expressions in \eqref{tfe} yields
\begin{equation*}
 \tex{
    \frac{\nu}{\l} \frac{\p \bar u}{\p \bar t}=-
    \frac{\nu^{n+1}}{\mu^4} \nabla \cdot (|\bar u|^{n} \nabla \D \bar u)\,.
    }
\end{equation*}
To keep this equation invariant, the following equalities must be
fulfilled:
\begin{equation}
\label{sf2}
 \begin{matrix}
 \tex{
    \frac{\nu}{\l}=\frac{\nu^{n+1}}{\mu^4} \LongA
    }
    \mu := \l^\b \Longrightarrow \nu := \l^{\frac{4\b-1}{n}}, \quad
    \mbox{so that}
     \ssk\ssk\\
 \tex{
    u(x,t):= \l^{\frac{4\b-1}{n}} \bar u(\bar x,\bar t) = \l^{\frac{4\b-1}{n}}
    \bar u(\frac{x}{\mu}) \whereA t=\l.
    }
     \end{matrix}
\end{equation}

Consequently, we have to  rescale in the following way:
\begin{equation}
\label{sf10}
 \tex{
    u_+(x,t)=t^{-\a} v(y,\tau), \quad y:=\frac{x}{t^{\b}},
    \quad \tau= \ln t \,:\, \re_+ \rightarrow \re
    \whereA \,\,\, \b= \frac {1-\a n}4,
 }
\end{equation}
such that $f(y,\tau)=\bar u(\frac{x}{t^\b},\tau)$, we obtain, after substituting
\eqref{sf10} into \eqref{tfe} and rearranging terms, that $f$  solves a quasilinear
 evolution equation given by
\begin{equation}
\label{evol}
 \tex{
 v_\t=
   \BB^+_n(\a,v) \equiv  -\nabla \cdot(\left|v\right|^{n} \nabla \D v)
    +\frac{1-\a n}{4}\, y  \cdot \nabla v +\a v \inB \ren \times \re_+\,.
    }
\end{equation}

 Consider the steady-states of the parabolic equation \eqref{evol}. Thus,
we analyze the local bifurcation-branching behaviour of the {\em nonlinear eigenvalue problem}:
\begin{equation}
\label{sf5}
 \fbox{$
 \tex{
  \BB^+_n(\a,f) \equiv
    -\nabla \cdot(|f|^{n} \nabla \D f) +\frac{1-\a n}{4}\, y \cdot \nabla f +\a
    f=0,
    \quad f \in C_0(\ren)\,.
    }
    $}
\end{equation}
 Here, the ``boundary conditions at infinity" stated as $f \in C_0(\ren)$
 are naturally associated with the known properties of {\em
 finite propagation} for TFEs, which have been mathematically
 justified about two decades ago at
 least; see a survey on energy methods in PDE theory in \cite{GS1S-V}.  Then,
  any assumption stating that $f(y)$ is
 ``sufficiently small" at infinity, e.g.,
  \be
  \label{H41}
  f \in H^4(\ren) \quad \mbox{or} \quad H^4_\rho(\ren)
   \ee
  (the last space is a domain of the linear operator $\BB$ in \ef{i4}; see the next section),
 would lead to compactly supported solutions. In fact, such a
 conclusion entirely depends on asymptotic (i.e., local, not any global) properties
of the nonlinear elliptic operators involved, so will not be a
main concern in our study.

\subsection{Blow-up similarity solutions}

The blow-up similarity solutions of the TFE--4 \ef{tfe} correspond
to completely different limits and then describe a ``micro-scale"
structure of its solutions at any given point. For convenience, we
reduce the blow-up time to $T=0$. Then, similar to the global
solutions, replacing $t \mapsto (-t)$, we obtain the patterns
\begin{equation}
 \label{u-}
 \tex{
    u_-(x,t):= (-t)^{-\a} f(y), \quad y=\frac{x}{(-t)^\b}, \quad \hbox{with the same parameter}
    \quad \b= \frac {1- \a n}4.
     }
\end{equation}
Hence, substituting that expression into  \eqref{tfe} and
rearranging terms, we arrive at the following quasilinear elliptic
equation:
\begin{equation}
\label{sf4}
 \fbox{$
 \BB_n^-(\a,f) \equiv
    -\nabla \cdot(|f|^{n} \nabla \D f)-\b y \cdot \nabla f-\a f=0
     \inB \ren.
    $}
\end{equation}
In order to get the corresponding second ``adjoint" nonlinear
eigenvalue problem, one needs to specify a ``minimal growth" of
admissible nonlinear eigenfunctions $f(y)$ as $y \to \iy$. This
will be done in Section \ref{S5}. Note that, by obvious and
straightforward reasons, \ef{sf4} does not admit compactly
supported solutions. Indeed, the nature of blow-up scaling \ef{u-}
would then mean disappearance of a such a solution in finite time,
contradicting uniqueness and other easy asymptotic issues for the
TFE--4.

In general, for solutions with finite blow-up time $T \in \re$,
the full self-similar scaling
\begin{equation}
\label{bupT}
 \tex{
    u(x,t)=(T-t)^{-\a} w(y,\tau), \quad y:=\frac{x}{(T-t)^{\b}}, \quad \tau= -\ln (T-t) \,:\,
    (-\iy, T)
     \rightarrow \re,
     }
\end{equation}
yields the parabolic equation
\begin{equation}
\label{BlowT}
 \tex{
 w_\t=
    \BB^-_n(\a,w) \equiv  -\nabla \cdot(\left|w\right|^{n} \nabla \D w)
    -\frac{1-\a n}{4}\, y  \cdot \nabla w - \a w \inB \ren \times
    \re_+.
    }
\end{equation}

\setcounter{equation}{0}
\section{Spectral properties of the linear operator ${\bf{B}}$ }
\label{S3}

\noindent In this section, we describe the spectrum $\s(\bf{B})$
of the linear operator $\bf{B}$ obtained from the rescaling of the
\emph{bi-harmonic equation} \eqref{s1}. This spectral theory will
be essentially used in ascertaining the direction of the branches
bifurcating from the trivial (actually nonexistence) and other
eigenfunctions and the number of branches for the blow-up
solutions.

\subsection{Relation to a linear eigenvalue problem}

\noindent Let $u(x,t)$ be the unique solution of the CP for the linear
parabolic bi-harmonic equation \eqref{s1} with the initial data
\begin{equation}
 \label{NN1}
    \tex{ u_0 \in L_{\rho}^2(\re^{N}), \quad \mbox{where}
    \quad \rho(y)={\mathrm e}^{a |y|^{4/3}}, \quad a>0 \,\,\,\mbox{small},}
\end{equation}
given by the convolution Poisson-type integral
\begin{equation}
\label{s2}
 \tex{
    u(x,t)=b(x,t)\, * \, u_0 \equiv t^{-\frac N4}
     \int\limits_{\re^{N}} F((x-z)t^{-\frac 14}) u_0(z)\, {\mathrm d}z.
    }
\end{equation}
Here, by scaling invariance of the problem, the  unique
fundamental solution of the operator $\frac{\p}{\p t} + \D^2$ has
the self-similar structure
\begin{equation}
\label{s3}
 \tex{
    b(x,t)=t^{-\frac N4} F(y), \quad y:=\frac{x}{t^{1/4}} \quad  (x\in \ren).
    }
\end{equation}
\par
Substituting $b(x,t)$ into \eqref{s1}, we obtain that the rescaled
fundamental kernel $F$ in \ef{s3} solves the linear elliptic
problem \ef{i4}.
${\bf B}$ is a non-symmetric linear operator, which is bounded
from $H_{\rho}^4(\re^{N})$ to $L_{\rho}^2(\re^{N})$ with the
exponential weight given in \ef{NN1}.
Here, more precisely,  $a\in (0,2d)$ is any positive constant,
depending on the parameter $d>0$ characterizing the exponential
decay of the rescaled kernel:
 \begin{equation}
  \label{F11}
        \tex{ |F(y)| \le D {\mathrm e}^{-d |y|^{4/3}} \quad \mbox{in} \quad
        \ren \quad (D>0),}
   \end{equation}
 Later on, by  $F$ we denote the oscillatory rescaled kernel as
the only solution of \eqref{i4}, which has exponential decay,
oscillates as $|y|\rightarrow \infty$, and satisfies the  standard
pointwise estimate \ef{F11}.

\ssk

Thus, we need to solve the corresponding {\em linear eigenvalue
problem}:
 \be
 \label{LP1}
  \fbox{$
  \BB \psi = \l \psi \inB \ren, \quad \psi \in L^2_\rho(\ren).
   $}
   \ee
   It seems  clear that the nonlinear
   problem \eqref{tfe} formally reduces to \eqref{LP1} at $n=0$ with
   the following shifting of the corresponding eigenvalues:
     \be
     \label{a221}
    \tex{
    \l=-\a(0) + \frac N4.
    }
    \ee
 It is another reason to call  \eqref{sf5} a {\em nonlinear eigenvalue
problem}, since for $n=0$ it reduces to the classic eigenvalue one
 for a linear differential operator. Moreover, crucially, the
discreteness of the real spectrum of the linear problem
\eqref{LP1} can be apparently  inherited by the nonlinear one, but
a complete justification of this issue is far from being clear.

\subsection{Functional setting and semigroup expansion}

\noindent Thus, we solve \eqref{LP1} and calculate the spectrum of
$\s({\bf B})$ in the weighted space $L_{\rho}^2(\re^{N})$. We then
need the following Hilbert space:
\begin{equation*}
    \tex{ H_{\rho}^4(\re^{N}) \subset L_{\rho}^2(\re^{N}) \subset L^2(\re^{N}).}
\end{equation*}
The Hilbert space $H_{\rho}^4(\re^{N})$ has the following inner
product:
\begin{equation*}
 \tex{
    \left\langle v,w \right\rangle_{\rho} := \int\limits_{\ren} \rho(y)
    \sum\limits_{k=0}^{4}
     D^{k} v(y) {\bar D^{k} w(y)}
     \, {\mathrm d}y,
      }
\end{equation*}
where $D^{k} v$ stands for the vector
$\{D^{\b} v\,,\,|\b|=k\}$,
and the norm
\begin{equation*}
 \tex{
    \left\| v \right\|_{\rho}^2 := \int\limits_{\ren} \rho(y) \sum\limits_{k=0}^{4}
     |D^{k} v(y)|^2 \, {\mathrm d}y.
    }
\end{equation*}
\par
 Next, introducing the rescaled variables
\begin{equation}
\label{s6}
 \tex{
    u(x,t)=t^{- \frac N 4} w(y,\tau), \quad y:=\frac{x}{t^{1/4}},
    \quad \tau= \ln t \,:\, \re_+ \rightarrow \re,
  }
\end{equation}
we find that the rescaled solution $w$ satisfies the evolution equation
\begin{equation}
\label{s7}
    \tex{ w_{\tau} = {\bf B}w\,,}
\end{equation}
since, substituting the representation of $u(x,t)$ \eqref{s6} into
\eqref{s1} yields
\begin{equation*}
 \tex{
    -\D_y^2 w + \frac{1}{4} \, y \cdot \nabla_y w  +\frac{N}{4} \,w(y,\tau)= t \frac{\p w}{\p t} \frac{\p \tau}{\p t}.
  }
\end{equation*}
Thus, to keep this invariant it must be satisfied that
 $
 \tex{
    t  \frac{\p \tau}{\p t}=1 \Longrightarrow \tau = \ln t.
    }
 $
Hence, $w(y,\tau)$ is the solution of the Cauchy problem for the
equation \eqref{s7} and with the following initial condition at
$\tau=0$, i.e., at $t=1$:
\begin{equation}
\label{s8}
    \tex{ w_0(y) = u(y,1)\equiv b(1)\, * \, u_0 = F\, * \, u_0\, .}
\end{equation}
Thus, the linear operator $\frac{\p}{\p \tau} - {\bf B}$ is  a
rescaled version of the standard parabolic one $\frac{\p}{\p t} +
\D^2$. Therefore, the corresponding semigroup ${\mathrm e}^{{\bf
B} \tau}$ admits an explicit integral representation. This helps
to establish some properties of the operator ${\bf B}$ and
describes other evolution features of the linear flow. From
\eqref{s2}, we find the following explicit representation of the
semigroup:
\begin{equation}
\label{s9}
 \tex{
    w(y,\tau)=\int\limits_{\re^{N}} F \big(y-z{\mathrm e}^{-\frac{\tau}{4}}\big)\, u_0(z) \,
    {\mathrm d}z \equiv {\mathrm e}^{{\bf B} \tau} w_0, \quad \mbox{where}
     \quad
    x=t^\frac{1}{4}y,  \quad \tau=\ln t.
    }
\end{equation}
Subsequently, consider Taylor's power series of the analytic kernel
\begin{equation}
\label{s10}
 \tex{
    F\big(y-z {\mathrm e}^{-\frac{\tau}{4}}\big)=\sum_{(\b)} {\mathrm e}^{
    -\frac{|\b|\tau}{4}} \frac{(-1)^{|\b|}}{\b!}    D^\b F(y) z^\b
    \equiv \sum_{(\b)} {\mathrm e}^{-\frac{|\b|\tau}{4}} \frac{1}{\sqrt{\b!}} \psi_\b(y) z^\b,
    }
\end{equation}
for any $y\in \re^{N}$, where
\begin{equation*}
    \tex{ z^{\b}:=z_1^{\b_1}\cdots z_{N}^{\b_{N}},}
\end{equation*}
and $\psi_{\b}$ are the normalized eigenfunctions of the operator
$\bf{B}$. The series in \ef{s10} converges uniformly on compact
subsets in $z \in \re^{N}$. Indeed, estimating coefficients for
$|\b|=l$,
\begin{equation*}
 \tex{
    \left|\sum_{\b=l} \frac{(-1)^{l}}{\b!}  D^\b F(y) z_1^{\b_1}\cdots z_{N}^{\b_{N}}\right| \leq b_l
    |z|^l,
    }
\end{equation*}
by Stirling's formula we have that, for $l\gg 1$,
\begin{equation}
\label{s11}
 \tex{
    b_l = \frac{N^l}{l!} \sup_{y\in \re^{N},
    |\b|=l} |D^\b F(y)| \approx \frac{N^l}{l!} l^{-l/4}
    {\mathrm e}^{l/4} \approx l^{-3l/4}c^l ={\mathrm e}^{-l \ln 3l/4 +l \ln c}.
    }
\end{equation}
Note that the series
$
    \tex{ \sum b_l |z|^l}
$
 has its radius of convergence $R=\infty$.

  Thus,
we obtain the following representation of the solution:
\begin{equation}
\label{s12}
 \tex{
    w(y,\tau)= \sum_{(\b)}  {\mathrm e}^{ \l_\b {\tau}} M_\b(u_0) \psi_\b(y),
    \quad \mbox{where} \quad
    \l_\b =: -\frac{|\b|}{4}
    }
\end{equation}
 and $\{\psi_\b\}$ are the eigenvalues and
eigenfunctions of the operator ${\bf B}$, respectively, and
\begin{equation*}
 \tex{
    M_\b(u_0):= \frac{1}{\sqrt{\b!}}
    \int\limits_{\re^{N}} z_1^{\b_1}\cdots z_{N}^{\b_{N}} u_0(z) \, {\mathrm d}z
 }
\end{equation*}
are the corresponding moments of the initial datum $w_0$ defined
by \eqref{s8}.

\subsection{Main spectral properties of the pair $\{\BB,\,
\BB^*\}$}

Thus, the next results hold \cite{EGKP}:

\begin{theorem}
\label{Th s1} {\rm (i)} The spectrum of ${\bf B}$ comprises real
eigenvalues only with the form,
\begin{equation}
\label{s13}
 \tex{
    \s({\bf B}):=\big\{\l_\b =: -\frac{|\b|}{4}\,,\,|\b|=0,1,2,...\big\}.
    }
\end{equation}
Eigenvalues $\l_\b$ have finite multiplicity with eigenfunctions,
\begin{equation}
\label{s14}
 \tex{
    \psi_\b(y):= \frac{(-1)^{|\b|}}
    {\sqrt{\b!}} D^\b F(y) \equiv \frac{(-1)^{|\b|}}{\sqrt{\b!}}
    \left(\frac{\p}{\p y_1}\right)^{\b_1}\cdots \left(\frac{\p}{\p y_N}\right)^{\b_N} F(y).
    }
\end{equation}

\noi{\rm (ii)} The subset of eigenfunctions
  $  \Phi=\{\psi_\b\}$
is complete in $L^2(\re^{N})$ and in $L_{\rho}^2(\re^{N})$.

\noi{\rm (iii)} For any $\l \notin \s({\bf B})$, the resolvent
  $  ({\bf B}-\l I)^{-1}$
is a compact operator in $L_{\rho}^2(\re^{N})$.
\end{theorem}

Then, the adjoint operator $\BB^*$ of ${\bf B}$ (in the dual
metric of $L_{\rho}^2(\ren)$ takes the form \ef{ad55}
and is defined
 in the weighted space
$L_{\rho^*}^2(\re^{N})$, with the domain $H^4_{\rho^*}(\ren)$,
where the (dual) weight function is exponentially decaying:
\begin{equation*}
 \tex{
    \rho^*(y) \equiv \frac{1}{\rho(y)} = {\mathrm e}^{-a|y|^\a}
    >0.
    }
\end{equation*}
It is a bounded linear operator \cite{EGKP},
\begin{equation*}
    \tex{ {\bf B}^*:H_{\rho^*}^4(\re^{N}) \to L_{\rho^*}^2(\re^{N}),
    \,\, \mbox{so} \,\,
    \left\langle {\bf B} v, w\right\rangle = \left\langle v, {\bf B}^* w\right\rangle,
    \,\,
    v \in H_{\rho}^4(\re^{N}), \,\,
    w \in H_{\rho^*}^4(\re^{N}).}
\end{equation*}
Moreover, the following theorem establishes the spectral properties of the
adjoint operator which will be very similar to those shown in Theorem\,\ref{Th s1}
for the operator $\bf{B}$.

\begin{theorem}
\label{Th s2}
{\rm (i)} The spectrum of ${\bf B}^*$ consists of eigenvalues of finite multiplicity,
\begin{equation}
\label{s15}
 \tex{
    \s({\bf B}^*)=\s({\bf B}):=\big\{\l_\b =: -\frac{|\b|}{4}\,,\,|\b|=0,1,2,...\big\},
    }
\end{equation}
and the eigenfunctions $\psi_\b^*(y)$
are polynomials of order $|\b|$.

\noi{\rm (ii)}  The subset of eigenfunctions
$\Phi^*=\{\psi_\b^*\}$
is complete and closed $L_{\rho^*}^2(\re^{N})$.

\noi{\rm (iii)} For any $\l \notin \s({\bf B}^*)$ the resolvent
$({\bf B}^*-\l I)^{-1}$
is a compact operator in $L_{\rho^*}^2(\re^{N})$.
\end{theorem}
It should be pointed out that, since $\psi_0=F$,
\begin{equation*}
 \tex{
     \int_{\re^{N}} \psi_0\, {\mathrm
     d}y =\int_{\re^{N}} F (y) \, {\mathrm
     d}y=1.
     }
\end{equation*}
However, thanks to \eqref{s14} we have that
\begin{equation*}
 \tex{
     \int_{\re^{N}} \psi_\b=0 \quad \hbox{for any} \quad
     |\b|\neq 0.
     }
\end{equation*}
This expresses the orthogonality property to the adjoint
eigenfunctions in terms of the dual inner product. Due to
Theorem\,\ref{Th s2} the adjoint eigenfunctions are polynomials
which form a complete subset in $L_{\rho^*}^2(\re^{N})$ with
exponential decaying weight $\rho^*(y) = {\mathrm e}^{-a|y|^{4/3}}$.
\par
Note that \cite{EGKP}, for the eigenfunctions $\{\psi_\b\}$ of
$\bf{B}$ denoted by \eqref{s14}, the corresponding adjoint
eigenfunctions are {\em generalized Hermite  polynomials} of  the
form
\begin{equation}
\label{s16}
 \tex{
    \psi_\b^*(y):=\frac{1}{\sqrt{\b!}}\big[y^\b + \sum\limits_{j=1}^{[\b/4]}
    \frac{1}{j!} \D^{2j} y^\b\big].
    }
\end{equation}
Hence, the orthogonality condition holds:
\begin{equation}
\label{s17}
    \tex{ \left\langle \psi_\b,\psi_\g \right\rangle=\d_{\b,\g}
    \quad \hbox{for any} \quad \b,\;\g,}
\end{equation}
where $\left\langle \cdot,\cdot \right\rangle$ is the duality
product in $L^2(\ren)$ and $\d_{\b,\g}$ is the Kronecker's delta.
Operators $\bf{B}$ and $\bf{B}^*$ have zero Morse index (no
eigenvalues with positive real parts are available).
\par
The main spectral results  are extended \cite{EGKP} to
$2m$th-order linear {\em poly-harmonic flows}
 \begin{equation}
  \label{PP1}
  \tex{
 u_t= - (-\D)^m u \quad \mbox{in} \quad \ren \times \re_+,
  }
  \end{equation}
  where
 the elliptic equation for the rescaled kernel $F(y)$ takes the
 form
\begin{equation}
\label{s5}
  \tex{
    {\bf B} F \equiv -(-\D_y)^m F + \frac{1}{2m}\, y \cdot \nabla_y F  +\frac{N}{2m} \,F=0
    \quad \hbox{in} \quad \re^{N},\quad \int\limits_{\re^{N}} F(y) \, {\mathrm
    d}y=1.
     }
\end{equation}
In particular, if $m=1$ and $N=1$, we find the classic
second-order \emph{Hermite operator} ${\bf B}$ (see \cite{BS} for
further information)
\begin{equation*}
 \tex{
    {\bf B} F \equiv F'' + \frac{1}{2}\, F' y +\frac{1}{2} \,F=0,
    }
\end{equation*}
whose name is associated with the work of Charles Hermite of 1870,
although such equations and polynomial eigenfunctions of the
adjoint operator $\BB^*=D_y^2- \frac 12\, y D_y$ were obtained
earlier by Jacques C.F.~Sturm in 1836, \cite{St}; see
\cite[Ch.~1]{GalGeom} for history and references.

\setcounter{equation}{0}
\section{${\bf (NEP)_+}$: local bifurcation-branching analysis via a formal approach}
   \label{S4}

\noindent In this section, we show the nonexistence of local
bifurcations points from the trivial solution $(n,f)=(n,0)$ for
the nonlinear operator \eqref{bf2}, and, hence, the existence of
branching from the eigenfunctions of the linear operator ${\bf B}$
when $n$ is sufficiently close to zero. This analysis allows us to
show locally the existence of non-zero solutions of
$\eqref{eigpm}_+$ in the proximity of $n=0$.
\par
Throughout this section, we write the operator \eqref{bf2} in the form $\mc{F}(n,f)$, denoted by \eqref{funct},
where $\mc{L}(\a,n)$ and $\mc{N}(n,f)$  are the corresponding
linear and nonlinear parts of the operator \eqref{bf2} under the
abstract framework already explained above (first section).
This operator is of class $\mc{C}^r$, with $r$ sufficiently big to
make all the subsequent derivatives exist. Moreover, when $n=0$
we have the operator ${\bf B}$ defined by \eqref{i4}, for which we
showed in the previous section its complete spectral theory. It is
apparent that its eigenvalues $\l_k$ and eigenfunctions $\psi_k$,
with $k\geq 0$, will determine the precise number of branches from
$(n,f)=(0,\psi_k)$.

Let us note that the nonlinearity condition assumes that the
functions $f$ are sufficiently smooth and have ``transversal"
zeros with a possible accumulating point at a finite interface
only. Also, for each $k\geq 0$, denoting $n_{0,k}=0$, we find that
\begin{equation}
\label{bf5}
 \tex{
    \ker[\mc{L}(n_{0,k})]=
    \ker\big({\bf B} + \frac{k}{4}\, I \big)= \mathrm{span\,}\{\psi_\b, \, |\b|=k
    \} \quad \hbox{for any}
    \quad k=0,1,2,3,\cdots\,,
    }
\end{equation}
where $\mc{L}(n_{0,k}):= \mc{L}(\a_k(0),n_{0,k})$. Then,  due to
Fredholm's alternative (see e.g., \cite{Deim}),
\begin{equation*}
    \tex{R[\mc{L}(n_{0,k})]= \big\{ u\in \mc{C}(\bar \O)\;:\; \int_\O u \,\psi_k =0\big\},}
    \quad \mbox{such that}
\end{equation*}
\begin{equation*}
 \tex{
    \ker\big({\bf B}+\frac{k}{4}\, I\big) \oplus R[\mc{L}(n_{0,k})] = L_{\rho}^2(\re^{N})
    \quad \hbox{for any}\quad k=0,1,2,\cdots\, .
    }
\end{equation*}
Thus, $\psi_k \notin R[\mc{L}(n_{0,k})]$ for any $k\geq 0$. It is
clear that the operator $\mc{L}(n_{0,k})={\bf B} + \frac{k}{4}\,
I$ is Fredholm, i.e., $R[\mc{L}(\a,n)]$ is a closed subspace of
$L_{\rho}^2(\re^{N})$ and
\begin{equation*}
    \tex{ \hbox{dim} \ker(\mc{L}(\a,n))< \infty, \quad \hbox{codim}\,R[\mc{L}(\a,n)]< \infty,}
\end{equation*}
at least for each $n \approx 0^+$. Then, the operators
$\mc{L}(n_{0,k})$ are Fredholm of index zero. Indeed, we already
know that the first eigenvalue $\l_0=0$  is a simple one of the
operator $\mc{L}(n_{0,0})= {\bf B}$, so its algebraic multiplicity
is 1. Hence, we will apply the classical results of
Crandall--Rabinowitz \cite{CR} about bifurcation for simple
eigenvalues in order to prove the  nonexistence of bifurcation
points from the branch of trivial solutions at the value $n_0=0$.

Note that, due to Theorems \ref{Th s1} and \ref{Th s2}, for any $k
\ge 1$, the algebraic multiplicity is equal to the geometric ones,
so we are not dealing with the problem of introducing the
generalized eigenfunctions (no Jordan blocks are necessary for
restrictions to eigenspaces).
\par
On the other hand, when dealing with essentially non-analytic
functions of $n$, at $n=0$ as in \ef{bf2}, we cannot use standard
apparatus of bifurcation-branching theory even in the case of
finite regularity; cf. \cite{Deim, KZ, VainbergTr}. This reflects
the main partially technical but often principal difficulties of
such a branching study.
\par
Once the assumptions are established, we introduce the
following concept, which will play a  role in the forthcoming
analysis. In dynamical system theory, such concepts are typical
for characterizing various types of bifurcations.
\begin{definition}
\label{De bf1} $(n_{0,k},0)$, with $n_{0,k}=0$ for any
$k=0,1,2\cdots$, is a bifurcation point for equation \eqref{sf5}
from the curve of trivial solutions $(n,0)$, if there exists a
sequence
\begin{equation*}
    \tex{ (n_{k_m},f_m)\in \re \times (H_{\rho}^4(\ren)\setminus \{0\}),}
\end{equation*}
where $m\geq 1$, such that
\begin{equation*}
    \tex{ \lim_{m \rightarrow \infty} (n_{k_m},f_m) = (0,0)
    \,\,\,\,
    \mbox{and} \,\,\,\,
      \mc{F}(n_{k_m},f_m)=0 \quad \mbox{for each \,\,$m\geq 1$}.}
\end{equation*}
\end{definition}

Since ${\bf B} + \frac{k}{4} \,I$ is Fredholm, and
\begin{equation}
\label{bf8}
 \tex{
    D_f \mc{F}(n_{0,k},0)f= \big({\bf B} + \frac{k}{4}\, I\big) f
    }
\end{equation}
for any $k$, on the whole, as was discussed in \cite{Lo}, it is
clear that the condition
\begin{equation*}
 \tex{
    M_k=\ker\big({\bf B} + \frac{k}{4}\, I\big)\geq 1,
     }
\end{equation*}
is necessary for any $(n,f)=(n_{0,k},0)$ to be a bifurcation point
of \eqref{sf5} from $(n,0)$, the trivial solution. However, the
bifurcation does not occur depending only on the linear part.
Therefore, as was shown in \cite{Lo} through simple algebraic
examples, the nature of the nonlinearity determines the
sufficiency condition in order to have such a bifurcation. To be more
precise, the nonlinear part of the operator \eqref{bf2} must
fulfill the assumptions established in the first section of this
paper.

Therefore, according to  bifurcation theory, the points, where a
bifurcation from the branch of trivial solutions occurs, must
depend on the spectrum associated with the linear part and assuming some
conditions on the nonlinear part. Here, we prove the
nonexistence of bifurcation from the branch of trivial solutions
at the value of the parameter $n=0$. Moreover, due to the rescaled
relation between the linear operators ${\bf B}$ \eqref{i4} and
$\mc{L}(\a,n)$ defined by \eqref{NL1NN}, it turns out that {\em
there is no bifurcation point from the branch of trivial solutions
at any $n \ge 0$}.

\subsection{Bifurcation-branching for simple eigenvalues}

\noindent
 Firstly, we present a result that provides us with the nonexistence
of the branch emanating from the trivial solutions at the point
$(n,f)=(n_{0,k},0)$, with $n_{0,k}=0$, when $\l_k$ is a simple
eigenvalue. Secondly, consistent with some recent findings
\cite{PV}, we also show that there exists a branching from the
eigenfunction $\psi_0$ at the value of the parameter $n_{0,0}=0$.
We actually know that $\l_0=0$ is a simple eigenvalue of $\BB$ in
a general setting in $\ren$,
 but we cannot
assure that it is the only such one in other geometries. For
instance, in 1D and in the radial setting, {\em all eigenvalues
are simple}, so we can apply this simplified analysis.

Thus, the calculus below  will be valid for any $k$ such that
$\l_k$ is simple (under suitable restrictions). However, to avoid
excessive notation, we make all the computations for the  case
$k=0$, which is always special and simpler.

\begin{lemma}
\label{Le bf2}
  Under the  regularity convention and assumptions in Section
  $\ref{S1.4}$:

 {\rm (i)} $(n,f)=(n_{0,0},0)$, with $n_{0,0}=0$, is not a bifurcation point for
the stationary equation
\begin{equation*}
     \tex{ \mc{F}(n,f)=0 \quad (f \in C_0(\ren) \,\,\, \mbox{or} \,\,\, H^4_\rho(\ren); \,\,\mbox{cf.}
     \,\,(\ref{H41}))}.
\end{equation*}

{\rm (ii)} $(n,f)=(n_{0,0},\psi_0)$ is a branching point for the
stationary solutions of the functional $\mc{F}(n,f)$. Furthermore,
let $Y_0$ be a subspace of $H_{\rho}^4(\ren)$,
\begin{equation*}
    \tex{Y_0:= \{ u\in \mc{C}(\ren)\;:\; \int_{\ren} u \,\psi_0 =0\},\quad
    \hbox{such that} \quad \ker(\mc{L}_{0,0}) \oplus Y_0 = H_{\rho}^4(\ren).}
\end{equation*}
Then, there exists $\e>0$ and two maps of the class $\mc{C}^{r-1}$,
\begin{equation*}
     \tex{ n_0\,:\, (-\e,\e) \to \re, \quad
     \Phi_0\,:\, (-\e,\e) \to Y_0,} \quad
    \tex{ n_0(0)= 0, \quad \Phi_0(0)= 0,}
\end{equation*}
such that, for $n_{0,0}=0$
and for each $s\in (-\e,\e)$,
\begin{equation}
\label{bf12}
     \tex{ \mc{F}(n_0(s),f_0(s))=0, \quad
      f_0(s):=\psi_0+s\Phi_0(s),}
\end{equation}
where $\psi_0$ is the eigenfunction associated with the simple
eigenvalue $\l_0$ of ${\bf B}$, so that
$\dim \ker [\mc{L}(n_0)]=1$. Moreover, there exists $\rho>0$ such
that if $\mc{F}(n,f)=0$ and $(n,f)\in B_r(0,\psi_0)$, then
either $f=\psi_0$, or $(n,f)= (n_0(s),f_0(s))$ for some $s\in
(-\e,\e)$, where $B_r(0,\psi_0)$ is the a ball of the radius
r centered as $(0,\psi_0)$ in $\re^2 \times L^2(\re^{N})$.
Furthermore, if $\mc{F}$ is analytic, so are $n_0(s)$, $\a_0(s)$,
and $f_0(s)$ near $0$.
\end{lemma}

\begin{proof}
 (i)
To show that $(n,f)=(n_{0,0},0)$, with $n_{0,0}=0$, is not  a
bifurcation point, we check that the transversality condition of
Crandall--Rabinowitz \cite{CR} is not satisfied. Under the
regularity  assumptions and convention of Section \ref{S1.4}
imposed  on the linear and nonlinear parts, we obtain that
$\mc{L}(n):= D_f \mc{F}(n,0)$, where standard calculations of the
derivative are allowed:
\begin{align*}
    \tex{ D_f \mc{F}(n,f)g} & \tex{ := \lim_{h\rightarrow 0} \frac{\mc{F}(n,f + h g)-\mc{F}(n,f)}{h}}
     \\ &
    \tex{ = \lim_{h\rightarrow 0} \frac{ -\nabla \cdot (|f+hg|^n \nabla \Delta (f+hg) )
    +\frac{1- \a n}{4} y \cdot \nabla (f+hg)+\a (f+hg) - \mc{F}(n,f)}{h}} \\ &
    \tex{ = \lim_{h\rightarrow 0} \frac{ \mc{F}(n,f) + h \left[-\nabla \cdot (n |f|^{n-1} g
     \nabla \Delta f )-
    \nabla \cdot ( |f|^n \nabla \Delta g ) + \frac{1- \a n}{4} y \cdot \nabla g+ \a g\right] + o(h) - \mc{F}(n,f)}{h}}\\ &
    \tex{= - \nabla \cdot (n|f|^{n-1} g \nabla \Delta f) -
    \nabla \cdot (|f|^n \nabla \Delta g) + \frac{1- \a n}{4}\, y \cdot \nabla g + \a g\,,}
\end{align*}
for any $g \in C_0(\ren)$, or $H^4_\rho(\ren)$; cf. (\ref{H41}).
Moreover, owing to the spectral theory
shown in Section\;\ref{S3}, we find that there exists a singular
value $\l_0$, the eigenvalue of the operator $\mc{L}(n_{0,0})$,
with $n_{0,0}=0$, associated with the eigenfunction $\psi_0$. Set
$\mc{L}_{0,0}:=\mc{L}(n_{0,0})$ and $\mc{L}_{1,0}:=\frac{\mathrm
d}{{\mathrm d} n} \mc{L}(n_{0,0})$ where
\begin{equation*}
    \tex{ \frac{\mathrm d}{{\mathrm d} n} \mc{L}(n):=
    \big( -\frac{N(4+Nn)-N^2 n}{4(4+Nn)^2} + \frac{\l_k}{4} \big)
    y \cdot \nabla - \frac{N^2}{(4+Nn)^2}\, I.}
\end{equation*}
Then, $\ker(\mc{L}_{0,0})= \mathrm{span\,}\{\psi_0\}$ and the
following transversality condition does not hold:
\begin{equation}
\label{bf15}
    \tex{ \mc{L}_{1,0}\psi_0 \notin R[\mc{L}_{0,0}].}
\end{equation}
Indeed, suppose
\begin{equation*}
 \tex{
    \mc{L}_{1,0}\psi_0 =
    - \frac{N}{16}\, y \cdot \nabla \psi_0
    - \frac{N^2}{16}\,\psi_0 \in
    R[\mc{L}_{0,0}], \quad \mbox{so}
    }
\end{equation*}
\begin{equation}
\label{bf16}
 \tex{
    - \frac{N}{16}\, y \cdot \nabla \psi_0
    - \frac{N^2}{16}\,\psi_0  =
    -\D^2 v +\frac{1}{4} \,y \cdot \nabla v  + \frac{N}{4}\,  v,
    }
\end{equation}
for any $v \in R[\mc{L}_{0,0}]$.
 Hence, multiplying \eqref{bf16} by the adjoint eigenfunction $\psi_0^*$
and integrating by parts yields
\begin{align*}
 \tex{
    - \frac{N^2}{16} \int \psi_0 \psi_0^* } & \tex{ - \frac{N}{16} \int \psi_0^*
    y \cdot \nabla \psi_0 = - \frac{N^2}{16} \int \psi_0 \psi_0^* + \frac{N}{16}
    \int \hbox{div} (\psi_0^* y) \psi_0 }
    \\ & \tex{
   = - \frac{N^2}{16} \int \psi_0 \psi_0^* + \frac{N}{16} \int \psi_0
    y \cdot \nabla \psi_0^* + \frac{N^2}{16} \int \psi_0 \psi_0^* }
    \\ & \tex{= - \frac{N^2}{16} \int \psi_0 \psi_0^* + \frac{N^2}{16} \int \psi_0 \psi_0^* =  0,
    }
\end{align*}
which implies the nonexistence of bifurcation at  $n_{0,0}=0$ from
the trivial solution.

\ssk

(ii) We now prove the final statement of Lemma \ref{Le bf2}. Then,
under the same necessary regularity assumptions and the convention
in Section \ref{S1.4}, we define the auxiliary operator
\begin{equation}
\label{bf17}
    \tex{ \mc{G}(s,n_0,\Phi_0):=} \left\{\begin{array}{ll}
    \tex{ \frac{\mc{F}(n_0,\,\psi_0+ s\Phi_0)}{s},}  & \tex{ \hbox{if}\quad s\neq 0,}\\
    \tex{ D_f \mc{F}(n_0,\,\psi_0)\Phi_0,} & \tex{ \hbox{if}\quad s= 0,}
    \end{array}\right.
\end{equation}
for $s\in\re$, $s\approx 0$, $n_0\in \re$, and $\Phi_0 \in Y_0$.
Since $\mc{F}$ is $\mc{C}^r$ in both variables, $\mc{G}$ is
$\mc{C}^{r-1}$ in all its arguments. The number $r$ is
sufficiently large ensuring that the derivatives employed in the
sequel exist. Moreover, by the  definition, we have that
\begin{equation}
\label{bf18}
    \tex{\mc{G}(0,0,0)=0,}
\end{equation}
since $D_f \mc{F}(n_0,0)\Phi_0=0$ by construction. Then, if the
zeros of the eigenfunction $\psi_0$ are transversal
a.e.\footnote{For $\psi_0(y)$ which is a radial function, this is
very probable; however we do not have a fully convincing proof.},
we find that
$$1-\left|\psi_0\right|^{n_0(s)} \approx 0$$
in the weak sense (or even ``a.e.") for a sufficiently small $s$. Hence,
\begin{equation}
\label{bf19}
    \begin{split}
    \tex{
     D_{(n_0,\Phi_0)} \mc{G}(0,0,0)
     }
      &
      \tex{ (n_0,\Phi_0) =
    \lim_{h\rightarrow 0} \frac{\mc{G}(0,n_0 h,h \Phi_0)-\mc{G}(0,0,0)}{h}
     }
    \\ &
    \tex{
    =\lim_{h\rightarrow 0} \frac{D_f \mc{F}(n_0,\psi_0)(h \Phi_0)}{h}   }
    \\ &
     \tex{
     = \lim_{h\rightarrow 0} (\mc{L}_{0,0}+ h n_0 \mc{L}_{1,0})
      \Phi_0 +o(h) = \mc{L}_{0,0}\Phi_0.
    }
    \end{split}
\end{equation}

Thus, owing to
\begin{equation*}
    \tex{ Y_0 \oplus \ker(\mc{L}_{0,0})=H_{\rho}^4(\ren) ,}
\end{equation*}
the operator
\begin{equation*}
    \tex{ D_{(n_0,\Phi_0)} \mc{G}(0,0,0)\;:\;\re \times Y \to L_{\rho}^2(\ren)}
\end{equation*}
is an isomorphism. Then applying the implicit function theorem, we
deduce the existence and uniqueness of two $\mc{C}^{r-1}$
functions
\begin{equation*}
     \tex{ n_0\,:\, (-\e,\e) \to \re, \quad
     \Phi_0\,:\, (-\e,\e) \to Y_0,} \quad \mbox{such that}
\end{equation*}
\begin{equation*}
    \tex{ n_0(0)= 0, \quad \Phi_0(0)= 0, \quad \mc{G}(s,n_0(s),\Phi_0(s)).}
\end{equation*}
\end{proof}


Throughout the rest of this section, we calculate the possible
types of local bifurcations. According to our formal analysis
above, it follows that \eqref{sf5} has a local curve of solutions
\begin{equation*}
    \tex{ (n_0(s),f_0(s)), \quad f_0(s):=\psi_0 +s\Phi_0(s),}
\end{equation*}
emanating from $(n,f)=(n,\psi_0)$ at $n=0$. This shows  in some natural
sense that the component emanating at $(n,f)=(0,\psi_0)$
consists of two subcontinua, $\mf{C}_0^+$ and $\mf{C}_0^-$ in the
direction of $\Phi_0$ and  $-\Phi_0$, respectively, where $\Phi_0$ belongs to the orthogonal space $Y_0$ to
the eigenspace spanned by the eigenfunction associated with the simple eigenvalue
$\l_0=0$. Moreover, those functions $(n_0(s),f_0(s))$ admit the
next expansions of the form: as $s \to 0$,
\begin{equation}
 \label{ss1}
     \begin{array}{l}
     \tex{ n_0(s):=s \g_{0,1}+s^2 \g_{0,2}+o(s^2),} \ssk \\
     \tex{ f_0(s):=\psi_0+ s \Phi_{0,1} +s^2 \Phi_{0,2}+ o(s^2),}
     \end{array}
\end{equation}
for certain real numbers $\g_{0,l}$ and some functions
$\Phi_{0,l}\in Y_0$, with $l=1,2$. In addition, since we are
assuming that $\a$ is dependent on $n$, this eventually yields $$
 \tex{
 \a_0(s)=
\frac{N}{4} + s \eta_{0,1} + s^2 \eta_{0,2}+ o(s^2),
 }
$$
   where $\eta_{0,i} \in \re$ for any $i=1,2,\cdots$.
\par

Now, substituting  the expansions defined by \ef{ss1} and the corresponding expansion
for the parameter $\a_0(n)$ into the  nonlinear
elliptic equation \eqref{sf5} yields
\begin{align*}
    \tex{
    -\nabla \cdot(
    } &
    \tex{
    \left| \psi_0+ s  \Phi_{0,1}\right|^{(s \g_{0,1}+o(s))}
    \nabla \D \psi_0) -s \nabla \cdot(\left| \psi_0+ s
    \Phi_{0,1}\right|^{(s \g_{0,1}+o(s))}
    \nabla \D  \Phi_{0,1})
    } \\ &
    \tex{ -s^2
    \nabla \cdot(\left| \psi_0+ s \Phi_{0,1}\right|^{(s \g_{0,1}+o(s))}
    \nabla \D  \Phi_{0,2})
    + \frac{1}{4}\, y \cdot\nabla (\psi_0+ s \Phi_{0,1}+o(s))
    }  \\ &
    \tex{ - s \frac{ \left(\frac{N}{4} + s \eta_{0,1} + s^2 \eta_{0,2}+ o(s^2)\right)
     \g_{0,1}}{4}\, y \cdot \nabla   (\psi_0+ s \Phi_{0,1}+o(s))
     }
    \\ &
    \tex{
    + \left(\frac{N}{4} + s \eta_{0,1} + s^2 \eta_{0,2}+ o(s^2)\right)
    \big(\psi_0+ s \Phi_{0,1} + o(s)\big)  =0\,.
    }
\end{align*}
 Next,  passing to the limit as $s \rightarrow 0$, by the regularity convention
 (in particular, this assumes the
transversality condition for the zeros of the eigenfunction
$\psi_0$), we have that $|\psi_0|^{s \g_{0,1}}=1+ s \g_{0,1} \ln
|\psi_0| + o(s)$, and according to the spectral theory shown in Section \ref{S3} we find that
\begin{equation*}
 \tex{
    {\bf B} \psi_0 \equiv
    \big(-\D^2+\frac{1}{4}\, y \cdot \nabla +\frac{N}{4}\, I \big)\psi_0=0\,.
    }
\end{equation*}
Hence, dividing the rest of the terms by
$s$ and passing to the limit as $s \rightarrow 0$ gives
\begin{equation}
\label{bf21}
 \tex{
    {\bf B}  \Phi_{0,1}=
    \g_{0,1} \left[\frac{N}{16}\, y \cdot \nabla  \psi_0+ \nabla
     \cdot (\ln |\psi_0| \nabla \Delta \psi_0)\right] +
    \eta_{0,1} \psi_0\,.
     }
\end{equation}
Now, applying  Fredholm's theory \cite{Deim} to \eqref{bf21}
yields that there exists a function $\Phi_{0,1}$, which solves
\eqref{bf21} if and only if the right hand side is orthogonal to
${\rm ker} \,\BB$, i.e., to the eigenfunction $\psi_0^*=1$ of the
adjoint operator
 ${\bf B}^*$. This uniquely defines the coefficient:
\begin{equation}
\label{bf22}
 \tex{
    \g_{0,1}:= \frac{\eta_{0,1} \left\langle \psi_0^*,\psi_0\right\rangle}{\frac{N}{16}\,
    \left\langle \psi_0^*,y \cdot \nabla  \psi_0\right\rangle + \left\langle \psi_0^*,
    \nabla \cdot (\ln |\psi_0| \nabla \Delta \psi_0)\right\rangle}= \frac{\eta_{0,1}}{\frac{N}{16}\,
    \left\langle 1,y \cdot \nabla  \psi_0\right\rangle + \left\langle 1,
    \nabla \cdot (\ln |\psi_0| \nabla \Delta
    \psi_0)\right\rangle}\,,
    }
\end{equation}
 provided that the denominator does not vanish (notably, a difficult
 property to prove or even to verify numerically).
Therefore, any different branching-type in  the vicinity of $n=0$
from the first eigenfunction $\psi_0$ of the operator ${\bf B}$
will depend on the values of the coefficients $\g_{0,1}$ and
$\eta_{0,1}$ related by \ef{bf22}.

\subsection{Bifurcation-branching for semisimple eigenvalues}

\noindent Hereafter, in this section, we focus on the case when
the kernel is multidimensional. Note that, the question of local
bifurcation at the value of the parameter, for which the
corresponding eigenvalue is simple has been extensively studied in
literature. In particular, as was shown above, for \eqref{sf5},
under some assumptions over the nonlinearity, we proved that no
bifurcation takes place at $n=0$ from the branch of the trivial
solution, and when the eigenvalue is simple. However, for any
neighbourhood around $n=0$, a branch of solutions emanates from
the associated eigenfunction in the direction of the orthogonal
subspace $Y_0$.

On the other hand, for eigenvalues with higher multiplicity, we
prove that the nontrivial solutions emanating from the
eigenfunctions $\psi_k$ at the value of the parameter $n_{0,k}=0$,
for any $k \geq 1$, are tangent to a manifold $Y_k$.

Also, in general, it is not completely understood
how many branches emanate from the trivial solution, which remains
an open problem and it can only be obtained for some specific
examples.

For the case of bifurcation from the branch of trivial
solutions $(n,f)=(n,0)$, there exist some results supposing that
the operators are potential (see \cite{DGM} and \cite{R}) and
very few for non-gradient, non-self-adjoint operators, \cite{KHK}.

Here, we provide a number of possible branches of
bifurcation-branching when $n$ is close to $0^+$ from the
eigenfunctions $\psi_k$ under some conditions imposed over certain
values.
\par
Similarly to the case of simple eigenvalues, we already know that
\begin{equation}
 \label{Mk1}
 \tex{M_k=\dim \ker\big({\bf B} + \frac{k}{4}\, I\big)\geq 1,}
\end{equation}
for any $k\geq 1$, and the inequality is strict in dimensions $N
\ge 1$. Note that $\l_k=- \frac k4$ is not a simple eigenvalue and
\eqref{bf5} is fulfilled by $\psi_k$ as the eigenfunctions
associated with those semisimple eigenvalues of ${\bf B}$, $\l_k$,
such that $\psi_k:=\sum_{\left|\b\right|=k} c_\b \hat{\psi}_\b$,
for every $k\geq 1$ and under the natural ``normalizing"
constraint
\begin{equation}
    \label{br15}
    \tex{\sum_{\left|\b\right|=k} c_\b=1.}
\end{equation}
Subsequently, under the circumstances imposed for the
nonlinearity, it is apparent that
\begin{equation*}
    \tex{ \dim \ker\big({\bf B} + \frac{k}{4}\, I\big)=
    \hbox{codim}\,
    R\big({\bf B} + \frac{k}{4}\, I\big)= M_k,}
\end{equation*}
for any $k\geq 1$. Thus, as was discussed above, the operator
\eqref{bf8} is Fredholm of index zero since it is a compact
perturbation of the identity.
\par
It should be pointed out that the odd crossing number condition
might fail, so we are not distinguishing between odd or even
multiplicities (see the works Ambrosetti \cite{A} for gradient
operators and by Kr\"{o}mer--Healey--Kielh\"{o}fer \cite{KHK} for
more general operators). It is classically known that, when the multiplicity
is odd, there is always a bifurcation. However, when the
multiplicity is even, the bifurcation depends strongly on the
nonlinearity.
\par
The next theorem is one of the main results of this paper.

\begin{theorem}
\label{Th bf3} Let the assumptions  for the linear and nonlinear
part of the functional $\mc{F}$ be satisfied, together with the
regularity convention in Section $\ref{S1.4}$, and $\ef{Mk1}$
hold.
  Then:

  {\rm (i)}
$(n_{0,k}=0,0)$, for any $k\geq 0$, is not a bifurcation point, and;

{\rm (ii)} if $$
 \tex{
 \ker\big({\bf B} + \frac{k}{4}\, I\big) \oplus Y_k = H_{\rho}^4(\ren),
  }
  $$
where the subspace $Y_k$ is defined by
\begin{equation*}
    \tex{Y_k:= \big\{ u\in \mc{C}(\ren)\;:\; \int_{\ren} u \,\psi_k =0\big\},}
\end{equation*}

\noi there exists $\e>0$ and two maps of class $\mc{C}^{r-1}$,
\begin{equation*}
     \tex{ n_k\,:\, (-\e,\e) \to \re, \quad
     \Phi_k\,:\, (-\e,\e) \to Y_k,}
\end{equation*}
such that, for $n=0$ and for each $s\in (-\e,\e)$,
\begin{equation}
\label{bf51}
     \tex{\mc{F}(n_k(s),f_k(s))=0, \quad
      f_k(s):=\psi_k+s \Phi_k(s).}
\end{equation}
Furthermore, if $\mc{F}(n,f)$ is analytic in a neighbourhood of
$(0,\psi_k)$, $k=1,2,...$, so are $n_k(s)$, $\a_k(s)$, and
$f_k(s)$ near $s=0$,
 a countable number of branches emanate from $\psi_k$ for  $n \approx 0^+$.
\end{theorem}


Note that our earlier result for simple eigenvalues, $k=0$, is
included here, with   similar conclusions.


\begin{proof}
Firstly, we consider the following auxiliary operator:
\begin{equation}
\label{bf52}
    \tex{\mc{G}(s,n_k,\Phi_k):=\left\{\begin{array}{ll}
    \frac{\mc{F}(n_k,\psi_k+ s \Phi_k)}{s}, & \hbox{if}\quad s\neq 0,\\
    D_f \mc{F}(n_k,\psi_k)\Phi_k, & \hbox{if}\quad s= 0,
    \end{array}\right.}
\end{equation}
for $s\in\re$ and close to zero, $n_k\in \re$ and $\Phi_k \in
Y_k$. Since $\mc{F}$ is $\mc{C}^r$ in both variables, $\mc{G}$ is
$\mc{C}^{r-1}$ in all its arguments. As customary, we impose some
regularity conditions making sure that all the derivatives in the
sequel exist. Moreover, by definition we have that
\begin{equation}
\label{bf53}
    \tex{\mc{G}(0,0,0) =0,}
\end{equation}
for every $k\geq 1$. Thus, similarly as done for the case of
simple eigenvalues \eqref{bf19},
\begin{equation}
\label{bf54}
    \begin{split}
    \tex{
     D_{(n_k,\Phi_k)} \mc{G}(0,0,0)
     }
      &
      \tex{ (n_k,\Phi_k) =
    \lim_{h\rightarrow 0} \frac{\mc{G}(0,n_k h,h \Phi_k)-\mc{G}(0,0,0)}{h}
     }
    \\ &
    \tex{
    =\lim_{h\rightarrow 0} \frac{D_f \mc{F}(n_k,\psi_k) h \Phi_k}{h}
   }
    \\ &
     \tex{
     = \lim_{h\rightarrow 0} (\mc{L}_{0,k}+ h n_k \mc{L}_{1,k}) \Phi_k +o(h)
    }
    \\ & = \mc{L}_{0,k}\Phi_k.
    \end{split}
\end{equation}
Hence, if the following condition
\begin{equation*}
    \tex{ Y_k \oplus \ker(\mc{L}_{0,k})=H_{\rho}^4(\ren)}
\end{equation*}
holds, then the operator $D_{(n_k,\Phi_k)} \mc{G}(0,0,0): \re
\times Y_k \rightarrow L_{\rho}^2(\ren)$ is an isomorphism. Consequently,
we can apply the implicit function theorem. Therefore, the
existence and uniqueness of the following two $C^{r-1}$ functions
are guaranteed:
\begin{equation*}
     \tex{ n_k\,:\, (-\e,\e) \to \re, \quad
     \Phi_k\,:\, (-\e,\e) \to Y_k,} \quad \mbox{such that} \quad
\end{equation*}
\par
Now, in order to conclude the proof, we must show that if
$(n,f)=(n_{0,k},0)$, with $n_{0,k}=0$, is not a bifurcation point,
for any $k \geq 0$, then the following condition,  providing us
with the bifurcation from the branch of trivial solutions at the
value of the parameter $n_{0,k}=0$, must not be satisfied
\begin{equation*}
    \tex{ \mathrm{span\,}\{\mc{L}_{1,k}\hat{\psi}_1,
    \cdots,\mc{L}_{1,k}\hat{\psi}_{M_k}\} \oplus
    R\big({\bf B} + \frac{k}{4}\, I\big) = L_{\rho}^2 (\ren),}
\end{equation*}
where $\mc{L}_{1,k}:=\frac{\mathrm d}{{\mathrm d} n}
\mc{L}(n_{0,k})$, with $n_{0,k}=0$, and
$\{\hat{\psi}_1,\cdots,\hat{\psi}_{M_k}\}$ a basis of the subspace
$\ker\big({\bf B} + \frac{k}{4}\, I\big)$ such that
\begin{equation*}
    \tex{\psi_k= c_1 \hat{\psi}_1 +\cdots + c_{M_k} \hat{\psi}_{M_k},}
\end{equation*}
with the ``normalizing" constraint \eqref{br15}.
Thus, we
suppose that
\begin{equation*}
 \tex{
    \mc{L}_{1,k}\psi_k = \sum_{j=1}^{M_k} c_j \mc{L}_{1,k}\hat{\psi}_j =
    (- \frac{N}{16}+\frac{\l_k}{4}) y \cdot \sum_{j=1}^{M_k} c_j \nabla \hat{\psi}_j
    - \frac{N^2}{16} \sum_{j=1}^{M_k} c_j \hat{\psi}_j \in
    R[\mc{L}_{0,k}], \,\,\, \mbox{so}
    }
\end{equation*}
\begin{equation}
\label{bf55} \begin{split}
 \tex{
    \sum_{j=1}^{M_k} c_j \mc{L}_{1,k}\hat{\psi}_j
    }
    &
    \tex{
    (= - \frac{N}{16}+\frac{\l_k}{4}) y \cdot \sum_{j=1}^{M_k} c_j \nabla \hat{\psi}_j
    - \frac{N^2}{16} \sum_{j=1}^{M_k} c_j \hat{\psi}_j
    }
    \\ &
    \tex{
    =
    -\D^2 v +\frac{1}{4} \,y \cdot \nabla v  + \frac{N}{4} v,
    }
    \end{split}
\end{equation}
for any $v \in R[\mc{L}_{0,k}]$. Now, we restrict ourselves to the
case when $k=1$ and $M_k=2$ (i.e., $N=2$) to avoid excessive
calculations. Hence, multiplying \eqref{bf55} by the associated
adjoint eigenfunctions $\hat{\psi}_1^*$ and $\hat{\psi}_2^*$ and
integrating by parts, we obtain the following system:
\begin{equation*}
    \tex{ \left\{\begin{array}{l} \frac{N+1}{16} \int
    \hat{\psi}_1^* y \cdot (c_1 \nabla \hat{\psi}_1+
    c_2 \nabla \hat{\psi}_2) + \frac{N^2}{16}  \int \hat{\psi}_1^*
    (c_1 \hat{\psi}_1+
    c_2 \hat{\psi}_2) =0,\ssk \\
    \frac{N+1}{16} \int
    \hat{\psi}_2^* y \cdot (c_1 \nabla \hat{\psi}_1+
    c_2 \nabla \hat{\psi}_2) + \frac{N^2}{16}  \int \hat{\psi}_2^*
    (c_1 \hat{\psi}_1+
    c_2 \hat{\psi}_2) =0,\end{array}\right.}
\end{equation*}
and, hence,
\begin{equation}
\label{bf56}
    \tex{ \left\{\begin{array}{l} c_1 \left[ \frac{N+1}{16} \int
    \hat{\psi}_1^* y \cdot \nabla \hat{\psi}_1 +
    \frac{N^2}{16}  \int \hat{\psi}_1^* \hat{\psi}_1\right] +
    c_2 \left[\frac{N+1}{16} \int
    \hat{\psi}_1^* y \cdot \nabla \hat{\psi}_2)
    + \frac{N^2}{16}  \int \hat{\psi}_1^* \hat{\psi}_2\right] =0, \ssk \\
    c_1\left[ \frac{N+1}{16} \int
    \hat{\psi}_2^* y \cdot \nabla \hat{\psi}_1+
    \frac{N^2}{16}  \int \hat{\psi}_2^* \hat{\psi}_1\right] +
    c_2 \left[\frac{N+1}{16} \int
    \hat{\psi}_2^* y \cdot \nabla \hat{\psi}_2 +
    \frac{N^2}{16}  \int \hat{\psi}_2^*\hat{\psi}_2\right] =0.\end{array}\right.}
\end{equation}
Consequently, if the determinant of the system \eqref{bf56} for
the unknowns $c_1$ and $c_2$ is different from zero,
\begin{equation*}
    \tex{ \left|\begin{array}{ll} (N+1) \int
    \hat{\psi}_1^* y \cdot \nabla \hat{\psi}_1 +
    N^2  \int \hat{\psi}_1^* \hat{\psi}_1 &
    (N+1) \int  \hat{\psi}_1^* y \cdot \nabla \hat{\psi}_2)
    + N^2  \int \hat{\psi}_1^* \hat{\psi}_2 \\
    (N+1) \int  \hat{\psi}_2^* y \cdot \nabla \hat{\psi}_1+
    N^2  \int \hat{\psi}_2^* \hat{\psi}_1 &
    (N+1) \int \hat{\psi}_2^* y \cdot \nabla \hat{\psi}_2 +
    N^2  \int \hat{\psi}_2^*\hat{\psi}_2\end{array}\right| \neq
    0},
\end{equation*}
we arrive at the desired result with the normalizing constraint
\eqref{br15}. Similar computations can be done for any finite $k$
and $M_k$; see below.
\end{proof}

\vspace{0.2cm}

Furthermore, as was performed for the case with simple
eigenvalues, we ascertain the conditions that provide us with how
the branching from the eigenfunctions at $n=0$ is and how many
branches we actually have. Hence, by Theorem\;\ref{Th bf3}, for
any $k\geq 1$, the local curve of solutions
\begin{equation*}
    \tex{ (n_k(s),f_k(s)), \quad f_k(s):=\psi_k +s \Phi_k(s),}
\end{equation*}
emanates from the branch of solutions $(n,f)=(0,\psi_k)$, for any $k\geq 1$. That curve of
solutions is defined locally by two maps of class $\mc{C}^{r-1}$,
\begin{equation*}
     \tex{ n_k\,:\, (-\e,\e) \to \re, \quad
     \Phi_k\,:\, (-\e,\e) \to Y_k,}
\end{equation*}
such that, for $s=0$,
\begin{equation*}
    \tex{ n_k(0)= 0, \quad \Phi_k(0)= 0,}
\end{equation*}
and the eigenfunction $\psi_k$ of the subspace $\ker\big({\bf B} +
\frac{k}{4}\, I\big)$, as well as the expansion of the parameter
$\a(n)$ depending on $n$.

Moreover, since the eigenvalues
associated with the eigenfunctions $\psi_k$ are semisimple, the
dimension of the kernel $M_k$ will be greater than 1.
 Thus,
the component emanating at $n=0$ will do it in the direction of
the orthogonal manifold to the one generated by
$\{\hat{\psi}_1,\cdots,\hat{\psi}_{M_k}\}$, in such a way that
$\psi_k = c_1 \hat{\psi}_1+\cdots+c_{M_k}\hat{\psi}_{M_k}$. In
other words, depending on the coefficients $c_1,\cdots,c_{M_k}$,
we shall obtain different directions of the bifurcation-branching.
Then, those functions $(n_k(s),f_k(s))$ admit the following
expansions of the form: as $s \to 0$,
\begin{equation*}
     \tex{ \begin{array}{l}
     n_k(s):=s \g_{k,1}+s^2 \g_{k,2}+o(s^2),\ssk\ssk\\
     f_k(s):=\psi_k+ s \Phi_{k,1} +s^2 \Phi_{k,2}+ o(s^2),
     \end{array} 
     }
\end{equation*}
for certain real numbers $\g_{k,l}$ and some functions
$\Phi_{k,l}\in Y_K$, with $l=1,2$, and $k\geq 1$. Furthermore, we
set $\a_k(s)= \frac{N+k}{4} + s \eta_{k,1} + s^2 \eta_{k,2}+
o(s^2)$, where $\eta_{k,i} \in \re$ for any $i=1,2,\cdots$ and any
$k\geq 0$.
\par
Hence, substituting those expansions into the equation \eqref{sf5}
and dividing by $s$ gives
\begin{align*}
    \tex{ -\nabla \cdot( } &
    \tex{  \left|\psi_k+ s  \Phi_{k,1}\right|^{(s \g_{k,1}+o(s))}
    \nabla \D \psi_k) -s \nabla \cdot(\left| \psi_k+ s
    \Phi_{k,1}\right|^{(s \g_{k,1}+o(s))}
    \nabla \D  \Phi_{k,1}) }
    \\ &
    \tex{ -s^2
    \nabla \cdot(\left|s \psi_k+ s^2 \Phi_{k,1}\right|^{(s \g_{k,1}+o(s))}
    \nabla \D  \Phi_{k,2})
    + \frac{1}{4}\, y \cdot\nabla (\psi_k+ s \Phi_{k,1}+ o(s))}
    \\ &
    \tex{ -s \frac{\left(\frac{N+k}{4} + s \eta_{k,1} + s^2 \eta_{k,2}+ o(s^2)\right)
     \g_{k,1}}{4}\, y \cdot \nabla   (\psi_k+ s \Phi_{k,1}+ o(s)) }
    \\ &
    \tex{ +\left(\frac{N+k}{4} + s \eta_{k,1} + s^2 \eta_{k,2}+ o(s^2)\right)
    \big(\psi_k+ s \Phi_{k,1} + o(s)\big)  =0\,.}
\end{align*}
Then, passing to the limit as $s \rightarrow 0$ in a similar way
as was done for the case of simple eigenvalues (assuming a
``sufficient transversality"  of a.a. zeros of the eigenfunctions
$\psi_k$ and, hence, the expansion $|\psi_k|^{s \g_{k,1}}=1+ s
\g_{k,1} \ln |\psi_k| + o(s)$), we have that
\begin{equation*}
 \tex{
    \big({\bf B}+\frac{k}{4}\, I\big) \psi_k \equiv
    \big(-\D^2+\frac{1}{4}\, y \cdot \nabla +\frac{N+k}{4}\, I
    \big)\psi_k=0\,,
    }
\end{equation*}
which is true by Section \ref{S3}. Dividing the rest by $s$ and
letting $s \rightarrow 0$ give
\begin{equation}
\label{bf57}
 \tex{
    \big({\bf B}+\frac{k}{4} \, I\big)  \Phi_{k,1}= \g_{k,1} \left[\frac{N+k}{16}\, y
    \cdot \nabla  \psi_k+ \nabla \cdot (\ln |\psi_k| \nabla \Delta \psi_k)\right] +
    \eta_{k,1} \psi_k\,,
     }
\end{equation}
 where  $\psi_k =
c_1 \hat{\psi}_1+\cdots+c_{M_k}\hat{\psi}_{M_k}$. Hence,
multiplying by the associated adjoint eigenfunctions
$\{\hat{\psi}_1^*,\cdots,\hat{\psi}_{M_k}^*\}$, integrating over
$\ren$ and applying the Fredholm alternative \cite{Deim},  we
obtain an algebraic system with the coefficients $\{c_j, \,j=
1,\cdots,M_k\}$, $\g_{k,1}$, and $\eta_{k,1}$ as the unknowns.
Again, to avoid excessive calculations, we will restrict ourselves
to the simplest case in which $k=1$ and $M_1=2$ ($N=2$). Then  we
arrive at the following algebraic system:
\begin{equation}
\label{bf58}
    \left\{\begin{array}{l}
    \tex{
    c_1 \big[\frac{\g_{1,1}(N+k)}{16}\,\int \hat{\psi}_1^*\, y \cdot \nabla  \hat{\psi}_1+ \g_{1,1}
    \,\int \hat{\psi}_1^*\, \nabla \cdot (\ln |c_1\hat{\psi}_1+c_2 \hat{\psi}_2| \nabla \Delta
    \hat{\psi}_1) }
     \ssk \\
 \tex{
      +
    \eta_{1,1} \,\int \hat{\psi}_1^*\, \hat{\psi}_1 \big]
    }
    \tex{+ c_2[\frac{\g_{1,1}(N+k)}{16}\,\int \hat{\psi}_1^*\, y \cdot \nabla  \hat{\psi}_2
    }
 \ssk \\
  \tex{
    + \g_{1,1}
    \,\int \hat{\psi}_1^*\, \nabla \cdot (\ln |c_1\hat{\psi}_1+c_2 \hat{\psi}_2|
    \nabla \Delta \hat{\psi}_2) +
    \eta_{1,1} \,\int \hat{\psi}_1^*\, \hat{\psi}_2]=0,
    }
    \ssk \\
    \tex{ c_1\big[\frac{\g_{1,1}(N+k)}{16}\,\int \hat{\psi}_2^*\, y \cdot \nabla  \hat{\psi}_1+ \g_{1,1}
    \,\int \hat{\psi}_2^*\, \nabla \cdot (\ln |c_1\hat{\psi}_1+c_2 \hat{\psi}_2| \nabla \Delta
     \hat{\psi}_1)
      }
 \ssk\\
 \tex{
      +
    \eta_{1,1} \,\int \hat{\psi}_2^*\, \hat{\psi}_1\big]}
    \tex{+ c_2[\frac{\g_{1,1}(N+k)}{16}\,\int \hat{\psi}_2^*\, y \cdot \nabla  \hat{\psi}_2
    }
 \ssk \\
  \tex{
    + \g_{1,1}
    \,\int \hat{\psi}_2^*\, \nabla \cdot (\ln |c_1\hat{\psi}_1+c_2 \hat{\psi}_2|
     \nabla \Delta \hat{\psi}_2) +
    \eta_{1,1} \,\int \hat{\psi}_2^*\, \hat{\psi}_2]=0,}
     \ssk \\
  \tex{ c_1+c_2=1.}\end{array}\right.
\end{equation}
We will achieve the existence of solutions due to standard fixed
point theory arguments (see \cite{PV} for further details). Then,
in order  to ascertain how many possible solutions we might
have, one can fix the value of $\eta_{k,1}$ and solve the
nonlinear algebraic system \eqref{bf58} for the remaining unknowns.
Thus, from the third equation of \eqref{bf58} we find that
$c_1=1-c_2$. Then, setting $c_1 \hat{\psi}_1 +c_2 \hat{\psi}_2 =
\hat{\psi}_1 +c_2 (\hat{\psi}_2-\hat{\psi}_1)$, substituting it in
the other two equations, and integrating by parts in some of the
terms (the ones with the logarithm), we obtain
\begin{equation*}
    \tex{ \begin{array}{l}
    \g_{1,1} \big[\frac{(N+k)}{16}\,\int \hat{\psi}_1^*\, y \cdot \nabla  \hat{\psi}_1+
    \,\int \nabla \hat{\psi}_1^*\, h_1 \big]+
    \eta_{1,1} \,\int \hat{\psi}_1^*\, \hat{\psi}_1
    \ssk \\
    \tex{ \,\,\,\,\,\,\, +c_2 \big[\frac{\g_{1,1}(N+k)}{16}\,\int \hat{\psi}_1^*\,
     y \cdot (\nabla  \hat{\psi}_2-\nabla  \hat{\psi}_1)+ \g_{1,1}
    \,\int \nabla \hat{\psi}_1^*\, (h_2-h_1) +
    \eta_{1,1} \,\int \hat{\psi}_1^*\, (\hat{\psi}_2-\hat{\psi}_1)\big]=0,}
    \ssk \\
    \tex{ \g_{1,1}\big[\frac{(N+k)}{16}\,\int \hat{\psi}_2^*\, y \cdot \nabla  \hat{\psi}_1+
    \,\int \nabla \hat{\psi}_2^*\, h_1 \big]+
    \eta_{1,1} \,\int \hat{\psi}_2^*\, \hat{\psi}_1}
    \ssk \\
    \tex{ \,\,\,\,\,\,\, +c_2 \big[\frac{\g_{1,1}(N+k)}{16}\,\int \hat{\psi}_2^*\, y \cdot (\nabla  \hat{\psi}_2-\nabla  \hat{\psi}_1)+ \g_{1,1}
    \,\int \nabla \hat{\psi}_2^*\, (h_2-h_1) +
    \eta_{1,1} \,\int \hat{\psi}_2^*\, (\hat{\psi}_2-\hat{\psi}_1) \big]=0,}
    \end{array}}
\end{equation*}
\begin{equation*}
 \mbox{where} \quad
    \tex{ h_1\,:=\, \ln |\hat{\psi}_1+c_2 (\hat{\psi}_2-\hat{\psi}_1)| \nabla \Delta \hat{\psi}_1\,, \quad
    h_2\,:=\, \ln |\hat{\psi}_1+c_2 (\hat{\psi}_2-\hat{\psi}_1)| \nabla \Delta \hat{\psi}_2\,.}
\end{equation*}
 Hence, we next solve the system without considering the extra perturbation
terms, which $h_1$ and $h_2$ are involved in, i.e.,
\begin{equation*}
    \tex{ \o_i(\g_{1,1},c_2) := \g_{1,1} \,\int \nabla \hat{\psi}_i^*\, h_1 + c_2 \g_{1,1}
    \,\int \nabla \hat{\psi}_i^*\, (h_2-h_1)\,, \quad \hbox{with} \quad i=1,2.}
\end{equation*}
Hence, we need to solve the system,
\begin{equation}
\label{sysgo}
    \tex{ \begin{array}{l}
    c_2\big[\frac{\g_{1,1}(N+k)}{16}\,\int \hat{\psi}_1^*\, y \cdot (\nabla  \hat{\psi}_2-\nabla  \hat{\psi}_1)
    +
    \eta_{1,1} \,\int \hat{\psi}_1^*\, (\hat{\psi}_2-\hat{\psi}_1)\big]
    \ssk \\
    \tex{ \qquad \qquad +\g_{1,1}\frac{(N+k)}{16}\,\int \hat{\psi}_1^*\, y \cdot \nabla  \hat{\psi}_1+
    \eta_{1,1} \,\int \hat{\psi}_1^*\, \hat{\psi}_1=0 ,}
    \ssk \\
    \tex{ c_2\big[\frac{\g_{1,1}(N+k)}{16}\,\int \hat{\psi}_2^*\, y \cdot (\nabla  \hat{\psi}_2-\nabla  \hat{\psi}_1) +
    \eta_{1,1} \,\int \hat{\psi}_2^*\, (\hat{\psi}_2-\hat{\psi}_1)\big]
    }
    \ssk \\
    \tex{ \qquad \qquad +\g_{1,1}\frac{(N+k)}{16}\,\int \hat{\psi}_2^*\, y \cdot \nabla  \hat{\psi}_1+
    \eta_{1,1} \,\int \hat{\psi}_2^*\, \hat{\psi}_1 =0.}
    \end{array}}
\end{equation}
At this point it is quite easy to prove that, for example, after
substituting the expression for $c_2$, obtained from the first
equation, into the second equation, we arrive at a quadratic form
which depends on the unknown $\g_{1,1}$,
\begin{equation*}
 \tex{ \mf{F}(\g_{1,1}):= A_1 \g_{1,1}^2 +B_1 \g_{1,1} +C_1 =0,}
\end{equation*}
with at most two possible solutions. Note we are assuming
that at least one  of the following conditions is fulfilled:
\begin{equation}
\label{bf60}
    \tex{\frac{\g_{1,1}(N+k)}{16}\,\int \hat{\psi}_i^*\, y \cdot (\nabla  \hat{\psi}_2-\nabla  \hat{\psi}_1)
    + \eta_{1,1} \,\int \hat{\psi}_i^*\, (\hat{\psi}_2-\hat{\psi}_1)\neq 0\, \quad \hbox{with} \quad i=1,2. }
\end{equation}
Those solutions will correspond to two possible values of $c_2$
which are the roots of the following quadratic form:
\begin{equation*}
 \tex{ \mf{G}(c_2):= A_2 c_2^2 +B_2 c_2 +C_2 =0.}
\end{equation*}
Moreover, owing to the ``normalizing" constraint \eqref{br15}, we
have that $c_2 \in [0,1]$. Hence, for that quadratic form the
following is ascertained:
\begin{enumerate}
\item[(i)] $c_2 = 0 \Rightarrow \mf{G}(0)= C_2$;
\item[(ii)] $c_2 = 1 \Rightarrow \mf{G}(1) = A_2+B_2+C_2$;
\item[(iii)] Differentiating $\mf{G}$ with respect to $c_2$, we obtain that
$\mf{G}'(c_2)= 2 c_2 A_2+B_2$. Then, the critical point of the function $\mf{G}$ is
$c_2^* = -\frac{B_2}{2A_2}$ and $\mf{G}(c_2^*)= -\frac{B_2}{4A_2}+C_2$.
\end{enumerate}
Once the solutions for $c_2$ are established, that we know they are
between 0 and 1 according to the ``normalizing" constraint, we are
able to ascertain the solutions for $\g_{2,1}$. Although, they can
reach any value in the real line, and not only values between 0
and 1, they must fulfill the quadratic form $\mf{F}$ as well, so
that, we will have at most two solutions corresponding to this
unknown $\g_{2,1}$. Therefore, going back to \eqref{bf60} and
supposing it is true,  we shall obtain two solutions after
imposing the following conditions:
\begin{enumerate}
\item[(a)] $C  (A_2+B_2+C_2) >0$;
\item[(b)] $C \big(-\frac{B_2}{4A_2}+C_2\big)<0 $; and
\item[(c)] $0<-\frac{B_2}{2A_2}<1$.
\end{enumerate}
On the other hand, if $-\frac{B_2}{4A_2}+C_2=0$ then we have just
one solution of the quadratic form.

However, in the  case when condition \eqref{bf60} is not
fulfilled, we obtain a single unique solution for $\g_{2,1}$,
which satisfies the following equality:
\begin{equation*}
    \tex{ \g_{1,1}=- \frac{\eta_{1,1} \,\int \hat{\psi}_1^*\, \hat{\psi}_1}
    {\frac{(N+k)}{16}\,\int \hat{\psi}_1^*\, y \cdot \nabla  \hat{\psi}_1}=
       - \frac{\eta_{1,1} \,\int \hat{\psi}_2^*\, \hat{\psi}_1}
        {\frac{(N+k)}{16}\,\int \hat{\psi}_2^*\, y \cdot \nabla  \hat{\psi}_1}\,.}
\end{equation*}
Unfortunately, in this case nothing can be said about the number
of solutions for $c_2$ and, hence,  for the other unknowns
appearing in the system \eqref{bf58}, unless just one of them is
satisfied, in which case we ascertain one unique solution for
all the unknowns. Observe that, for any solution pair
$(\g_{1,1},c_2)$, there  corresponds a value of $\eta_{1,1}$, that
we fixed above.

Therefore, we will obtain at most two solutions of the system
\eqref{sysgo}, and, eventually, imposing some conditions on the
extra nonlinear terms $\o_i(\g_{1,1},c_2)$ such as
\begin{equation*}
    \left\| \o_i (\g_{1,1},c_2) \right\|_{L^\infty} \leq \min \{\mf{F}(\g_{1,1}^*),\mf{G}(c_2^*)\}, \quad \hbox{for any} \quad i=1,2,
\end{equation*}
where, $\g_{1,1}^*,c_2^*$ are the values where the quadratic forms catch the critical points, we finally
obtain at most two solutions for the original nonlinear algebraic system \eqref{bf58}.

\vspace{0.4cm}

For the sake of completion, we extend these results to the
case in which $k=2$ and the dimension of the kernel is $M_2=3$
(again, $N=2$). In other words, the kernel will be generated by
$\{\hat{\psi}_1^*,\hat{\psi}_2^*,\hat{\psi}_{3}^*\}$ such that
$\psi_2=c_1 \hat{\psi}_1^*+c_2\hat{\psi}_2^*+c_3\hat{\psi}_{3}^*$.
Hence, for this particular case, multiplying again \eqref{bf57} by
the adjoint eigenfunctions $\hat{\psi}_i^*$, with $i=1,2,3$, and
imposing the ``normalizing" constraint \eqref{br15}, the following
algebraic system is obtained:
\begin{equation}
\label{bf62}
   \begin{array}{l}
    c_1\big[\frac{\g_{2,1}(N+k)}{16}\,\int \hat{\psi}_1^*\, y \cdot \nabla  \hat{\psi}_1+ \g_{2,1}
    \,\int \hat{\psi}_1^*\, \nabla \cdot (\ln |c_1\hat{\psi}_1+c_2 \hat{\psi}_2
    +c_3 \hat{\psi}_3| \nabla \Delta \hat{\psi}_1)
 \ssk\\
 \tex{
     +
    \eta_{2,1} \,\int \hat{\psi}_1^*\, \hat{\psi}_1\big]}
    \tex{
    + c_2\big[\frac{\g_{2,1}(N+k)}{16}\,\int \hat{\psi}_1^*\, y \cdot \nabla
     \hat{\psi}_2
      }
 \ssk\\
    \tex{ + \g_{2,1}
    \,\int \hat{\psi}_1^*\, \nabla \cdot (\ln |c_1\hat{\psi}_1+c_2 \hat{\psi}_2
    +c_3 \hat{\psi}_3| \nabla \Delta \hat{\psi}_2) +
    \eta_{2,1} \,\int \hat{\psi}_1^*\, \hat{\psi}_2\big]}
    \ssk \\
    \tex{+ c_3\big[\frac{\g_{2,1}(N+k)}{16}\,\int \hat{\psi}_1^*\, y \cdot \nabla  \hat{\psi}_3
    }
 \ssk\\
 \tex{
    + \g_{2,1}
    \,\int \hat{\psi}_1^*\, \nabla \cdot (\ln |c_1\hat{\psi}_1+c_2 \hat{\psi}_2+c_3 \hat{\psi}_3| \nabla \Delta \hat{\psi}_3) +
    \eta_{2,1} \,\int \hat{\psi}_1^*\, \hat{\psi}_3\big]=0,}
    \ssk\ssk \\
    \tex{ c_1\big[\frac{\g_{2,1}(N+k)}{16}\,\int \hat{\psi}_2^*\, y \cdot \nabla  \hat{\psi}_1+
     \g_{2,1}
    \,\int \hat{\psi}_2^*\, \nabla \cdot (\ln |c_1\hat{\psi}_1+c_2 \hat{\psi}_2+c_3 \hat{\psi}_3|
     \nabla \Delta \hat{\psi}_1)}
 \ssk\\
      +
    \eta_{2,1} \,\int \hat{\psi}_2^*\, \hat{\psi}_1\big]
    \tex{+ c_2\big[\frac{\g_{2,1}(N+k)}{16}\,\int \hat{\psi}_2^*\, y \cdot \nabla  \hat{\psi}_2
    }
 \ssk\\
    + \g_{2,1}
    \,\int \hat{\psi}_2^*\, \nabla \cdot (\ln |c_1\hat{\psi}_1+c_2 \hat{\psi}_2+c_3 \hat{\psi}_3|
    \nabla \Delta \hat{\psi}_2) +
    \eta_{2,1} \,\int \hat{\psi}_2^*\, \hat{\psi}_2\big]
    \ssk \\
    \tex{+ c_3\big[\frac{\g_{2,1}(N+k)}{16}\,\int \hat{\psi}_2^*\,
     y \cdot \nabla  \hat{\psi}_3
     }
 \ssk\\
     + \g_{2,1}
    \,\int \hat{\psi}_2^*\, \nabla \cdot (\ln |c_1\hat{\psi}_1+c_2 \hat{\psi}_2+c_3
    \hat{\psi}_3| \nabla \Delta \hat{\psi}_3) +
    \eta_{2,1} \,\int \hat{\psi}_2^*\, \hat{\psi}_3\big]=0,
    \ssk\ssk \\
    \tex{ c_1\big[\frac{\g_{2,1}(N+k)}{16}\,\int \hat{\psi}_3^*\, y \cdot \nabla  \hat{\psi}_1
    + \g_{2,1}
    \,\int \hat{\psi}_3^*\, \nabla \cdot (\ln |c_1\hat{\psi}_1+c_2 \hat{\psi}_2+c_3
     \hat{\psi}_3| \nabla \Delta \hat{\psi}_1)
     }
     \ssk\\
      +
    \eta_{2,1} \,\int \hat{\psi}_3^*\, \hat{\psi}_1\big]
    \tex{+ c_2\big[\frac{\g_{2,1}(N+k)}{16}\,\int \hat{\psi}_3^*\, y \cdot \nabla  \hat{\psi}_2
    }
    \ssk\\
    + \g_{2,1}
    \,\int \hat{\psi}_3^*\, \nabla \cdot (\ln |c_1\hat{\psi}_1+c_2 \hat{\psi}_2
    +c_3 \hat{\psi}_3| \nabla \Delta \hat{\psi}_2) +
    \eta_{2,1} \,\int \hat{\psi}_3^*\, \hat{\psi}_2\big]
    \ssk \\
    \tex{+ c_3\big[\frac{\g_{2,1}(N+k)}{16}\,\int \hat{\psi}_3^*\,
     y \cdot \nabla  \hat{\psi}_3
     }
     \ssk\\
     + \g_{2,1}
    \,\int \hat{\psi}_3^*\, \nabla \cdot (\ln |c_1\hat{\psi}_1+c_2
    \hat{\psi}_2+c_3 \hat{\psi}_3| \nabla \Delta \hat{\psi}_3) +
    \eta_{2,1} \,\int \hat{\psi}_3^*\, \hat{\psi}_3\big]=0,
    \ssk\ssk \\
  \tex{ c_1+c_2+c_3=1.}\end{array}
\end{equation}
As mentioned above, for the case $k=1$, the existence of
non-degenerate solutions is guaranteed by standard fixed point
theory. Moreover, since by the third equation $c_1=1-c_2-c_3$,
substituting it into the other three equations of the system
\eqref{bf62} yields
\begin{equation}
\label{beast}
    \tex{ \begin{array}{l}
    \g_{2,1}\big[\frac{(N+k)}{16}\,\int \hat{\psi}_1^*\, y \cdot \nabla  \hat{\psi}_1+
    \,\int \nabla \hat{\psi}_1^*\, h_1\big]+
    \eta_{2,1} \,\int \hat{\psi}_1^*\, \hat{\psi}_1
    \ssk \\
    \tex{ \,\,\,\,\,\,\, +c_2\big[\frac{\g_{2,1}(N+k)}{16}\,\int \hat{\psi}_1^*\,
     y \cdot (\nabla  \hat{\psi}_2-\nabla  \hat{\psi}_1)+ \g_{2,1}
    \,\int \nabla \hat{\psi}_1^*\, (h_2-h_1)
 }
 \ssk\\
     +
    \eta_{2,1} \,\int \hat{\psi}_1^*\, (\hat{\psi}_2-\hat{\psi}_1)\big]
    \tex{+c_3\big[\frac{\g_{2,1}(N+k)}{16}\,\int \hat{\psi}_1^*\, y \cdot
     (\nabla  \hat{\psi}_3-\nabla  \hat{\psi}_1)
     }
     \ssk\\
     + \g_{2,1}
    \,\int \nabla \hat{\psi}_1^*\, (h_3-h_1) +
    \eta_{2,1} \,\int \hat{\psi}_1^*\, (\hat{\psi}_3-\hat{\psi}_1)\big]=0,
    \ssk \\
    \tex{ \g_{2,1}\big[\frac{(N+k)}{16}\,\int \hat{\psi}_2^*\, y \cdot \nabla  \hat{\psi}_1+
    \,\int \nabla \hat{\psi}_2^*\, h_1\big]+
    \eta_{2,1} \,\int \hat{\psi}_2^*\, \hat{\psi}_1}
    \ssk \\
    \tex{ \,\,\,\,\,\,\, +c_2\big[\frac{\g_{2,1}(N+k)}{16}\,\int \hat{\psi}_2^*\,
    y \cdot (\nabla  \hat{\psi}_2-\nabla  \hat{\psi}_1)+ \g_{2,1}
    \,\int \nabla \hat{\psi}_2^*\, (h_2-h_1)
    }
    \ssk\\
     +
    \eta_{2,1} \,\int \hat{\psi}_2^*\, (\hat{\psi}_2-\hat{\psi}_1)\big]
    \tex{+c_3\big[\frac{\g_{2,1}(N+k)}{16}\,\int \hat{\psi}_2^*\, y \cdot (\nabla  \hat{\psi}_3-
    \nabla  \hat{\psi}_1)
    }
    \ssk\\
    + \g_{2,1}
    \,\int \nabla \hat{\psi}_2^*\, (h_3-h_1) +
    \eta_{2,1} \,\int \hat{\psi}_2^*\, (\hat{\psi}_3-\hat{\psi}_1)\big]=0,
    \ssk \\
    \tex{ \g_{2,1}\big[\frac{(N+k)}{16}\,\int \hat{\psi}_3^*\, y \cdot \nabla  \hat{\psi}_1+
    \,\int \nabla \hat{\psi}_2^*\, h_1\big]+
    \eta_{2,1} \,\int \hat{\psi}_3^*\, \hat{\psi}_1}
    \ssk \\
    \tex{ \,\,\,\,\,\,\, +c_2\big[\frac{\g_{2,1}(N+k)}{16}\,\int \hat{\psi}_3^*\,
     y \cdot (\nabla  \hat{\psi}_2-\nabla  \hat{\psi}_1)+ \g_{2,1}
    \,\int \nabla \hat{\psi}_2^*\, (h_2-h_1)
    }
    \ssk\\
     +
    \eta_{2,1} \,\int \hat{\psi}_3^*\, (\hat{\psi}_2-\hat{\psi}_1)\big]
    \tex{ \,\,\,\,\,\,\, +c_3\big[\frac{\g_{2,1}(N+k)}{16}\,\int \hat{\psi}_3^*\, y
    \cdot (\nabla  \hat{\psi}_3-\nabla  \hat{\psi}_1)
    }
    \ssk\\
    + \g_{2,1}
    \,\int \nabla \hat{\psi}_2^*\, (h_3-h_1) +
    \eta_{2,1} \,\int \hat{\psi}_3^*\, (\hat{\psi}_3-\hat{\psi}_1)\big]=0,
    \end{array}}
\end{equation}
where,
\begin{align*}
    & \tex{ h_1\,:=\, \ln |\hat{\psi}_1+c_2 (\hat{\psi}_2-\hat{\psi}_1)+c_3(\hat{\psi}_3-\hat{\psi}_1)| \nabla \Delta \hat{\psi}_1\,,}
    \\ & \tex{ h_2\,:=\, \ln |\hat{\psi}_1+c_2 (\hat{\psi}_2-\hat{\psi}_1)+c_3(\hat{\psi}_3-\hat{\psi}_1)| \nabla \Delta \hat{\psi}_2\,,}
    \\ & \tex{h_3\,:=\, \ln |\hat{\psi}_1+c_2 (\hat{\psi}_2-\hat{\psi}_1)+c_3(\hat{\psi}_3-\hat{\psi}_1)| \nabla \Delta \hat{\psi}_3\,.}
\end{align*}
Subsequently,  as previously done for the particular case  $k=1$,
we solve the nonlinear algebraic system \eqref{beast} without
including   complicated nonlinear perturbations,  which the terms
$h_1$, $h_2$, and $h_3$ are involved in:
\begin{equation*}
    \tex{ \o_i(\g_{2,1},c_2,c_3) := \g_{2,1} \,\int \nabla \hat{\psi}_i^*\, h_1 + c_2 \g_{2,1}
    \,\int \nabla \hat{\psi}_i^*\, (h_2-h_1)+ c_3 \g_{2,1}
    \,\int \nabla \hat{\psi}_i^*\, (h_3-h_1)\,,}
\end{equation*}
with $i=1,2$.
Thus, ascertaining the number of possible solutions for the system
\begin{equation*}
    \tex{ \begin{array}{l}
    \g_{2,1}\frac{(N+k)}{16}\,\int \hat{\psi}_1^*\, y \cdot \nabla  \hat{\psi}_1+
    \eta_{2,1} \,\int \hat{\psi}_1^*\, \hat{\psi}_1
    \ssk \\
    \tex{ \,\,\,\,\,\,\, +c_2\big[\frac{\g_{2,1}(N+k)}{16}\,\int \hat{\psi}_1^*\, y \cdot (\nabla  \hat{\psi}_2-\nabla  \hat{\psi}_1) +
    \eta_{2,1} \,\int \hat{\psi}_1^*\, (\hat{\psi}_2-\hat{\psi}_1)\big]}
    \ssk \\
    \tex{ \,\,\,\,\,\,\, +c_3\big[\frac{\g_{2,1}(N+k)}{16}\,\int \hat{\psi}_1^*\, y \cdot (\nabla  \hat{\psi}_3-\nabla  \hat{\psi}_1) +
    \eta_{2,1} \,\int \hat{\psi}_1^*\, (\hat{\psi}_3-\hat{\psi}_1)\big]=0,}
    \ssk \\
    \tex{ \g_{2,1}\frac{(N+k)}{16}\,\int \hat{\psi}_2^*\, y \cdot \nabla  \hat{\psi}_1+
    \eta_{2,1} \,\int \hat{\psi}_2^*\, \hat{\psi}_1}
    \ssk \\
    \tex{ \,\,\,\,\,\,\, +c_2\big[\frac{\g_{2,1}(N+k)}{16}\,\int \hat{\psi}_2^*\, y \cdot (\nabla  \hat{\psi}_2-\nabla  \hat{\psi}_1)+
    \eta_{2,1} \,\int \hat{\psi}_2^*\, (\hat{\psi}_2-\hat{\psi}_1)\big]}
    \ssk \\
    \tex{ \,\,\,\,\,\,\, +c_3\big[\frac{\g_{2,1}(N+k)}{16}\,\int \hat{\psi}_2^*\, y \cdot (\nabla  \hat{\psi}_3-\nabla  \hat{\psi}_1) +
    \eta_{2,1} \,\int \hat{\psi}_2^*\, (\hat{\psi}_3-\hat{\psi}_1)\big]=0,}
    \ssk \\
    \tex{ \g_{2,1}\frac{(N+k)}{16}\,\int \hat{\psi}_3^*\, y \cdot \nabla  \hat{\psi}_1+
    \eta_{2,1} \,\int \hat{\psi}_3^*\, \hat{\psi}_1}
    \ssk \\
    \tex{ \,\,\,\,\,\,\, +c_2\big[\frac{\g_{2,1}(N+k)}{16}\,\int \hat{\psi}_3^*\, y \cdot (\nabla  \hat{\psi}_2-\nabla  \hat{\psi}_1)+
    \eta_{2,1} \,\int \hat{\psi}_3^*\, (\hat{\psi}_2-\hat{\psi}_1)\big]}
    \ssk \\
    \tex{ \,\,\,\,\,\,\, +c_3\big[\frac{\g_{2,1}(N+k)}{16}\,\int \hat{\psi}_3^*\, y \cdot (\nabla  \hat{\psi}_3-\nabla  \hat{\psi}_1)+
    \eta_{2,1} \,\int \hat{\psi}_3^*\, (\hat{\psi}_3-\hat{\psi}_1)\big]=0,}
    \end{array}}
\end{equation*}
and controlling the oscillations of the extra  nonlinear
perturbations $\o_i(\g_{2,1}c_2,c_3)$, for any $i=1,2,3$, as above
for  $k=1$, we achieve the desired results imposing the conditions
\begin{equation*}
    \tex{\frac{\g_{1,1}(N+k)}{16}\,\int \hat{\psi}_i^*\, y \cdot (\nabla  \hat{\psi}_j-\nabla  \hat{\psi}_1)
    + \eta_{1,1} \,\int \hat{\psi}_i^*\, (\hat{\psi}_j-\hat{\psi}_1)\neq 0\,, \; \hbox{with} \;\; i=1,2,3\;\;\hbox{and}\;\; j=2,3. }
\end{equation*}

Our system can be reduced to the study  of two perturbed quadratic
forms. Therefore, we arrive at the problem of studying the number
of intersections of two {\em conic  surfaces}, which provides us
with the number of solutions between zero and four.
 We postpone explaining how this approach works until Section
 \ref{S5.5}, where it is applied to the blow-up nonlinear eigenvalue problem $\ef{eigpm}_-$.
This approach is quite similar for both the cases.


As a preliminary but a key conclusion, it is worth mentioning now
that we believe that, since we are dealing with  a kernel of the
dimension 3, in this  case, we have four solutions. It seems then
that two of them should coincide. This is a principal issue, owing
to the fact that, somehow, the number of solutions depends on the
coefficients we have for the system and, at the same time, on the
eigenfunctions that generate the subspace $\ker\big({\bf
B}+\frac{k}{4}\, I\big)$, so its dimension. However, a full
justification is not proved here and, due to the difficult nature
of the problem, perhaps it will never be possible to justify it
completely.

\subsection{A short discussion on global behaviour of $n$-branches}

After performing a very precise analysis about the
bifurcation-branching analysis in the proximity of $n=0$, we
intend to explain how the global behaviour of the branch of
solutions that emanates from the eigenfunction $\psi_0$ at the
value of the parameter $n=0$ can be. It is well known that the
existence of solutions is not guaranteed for any $n$. Indeed, it
was discussed in \cite{EGK2} that the existence of oscillatory
solutions ends when $n=1.7587...\,$. The value where the existence
of oscillatory solutions ceases obviously depends on the type of
the thin film equation we are dealing with (the pure TFE, with
extra absorption terms, stable, unstable, etc).

Due to the analysis performed in this section, we already know
that there is no bifurcation from the branch of trivial solutions
at  $n=0$ for the TFE $\eqref{eigpm}_+$,
\begin{equation*}
   \tex{
 \BB^+_n (\a,f) \equiv - \n \cdot (|f|^n \n \D f) + \b y \cdot \n f +
 \a f=0 \inB \ren.
 }
 \end{equation*}
Moreover, after some rescaling the spectrum of the linear counterpart of
$\eqref{eigpm}_+$ is directly related with the spectrum of the operator ${\bf B}$ denoted by \eqref{i4}
\begin{equation*}
 \tex{
    {\bf B}F \equiv -\D_y^2 F + \frac{1}{4}\,y \cdot \nabla_y F +\frac{N}{4}\, F=0
    \quad \hbox{in} \quad \re^{N},\quad \int\limits_{\re^{N}} F(y) \, {\mathrm
    d}y=1.
    }
\end{equation*}
\begin{figure}[!htb]
\begin{center}
\includegraphics[width=8cm]{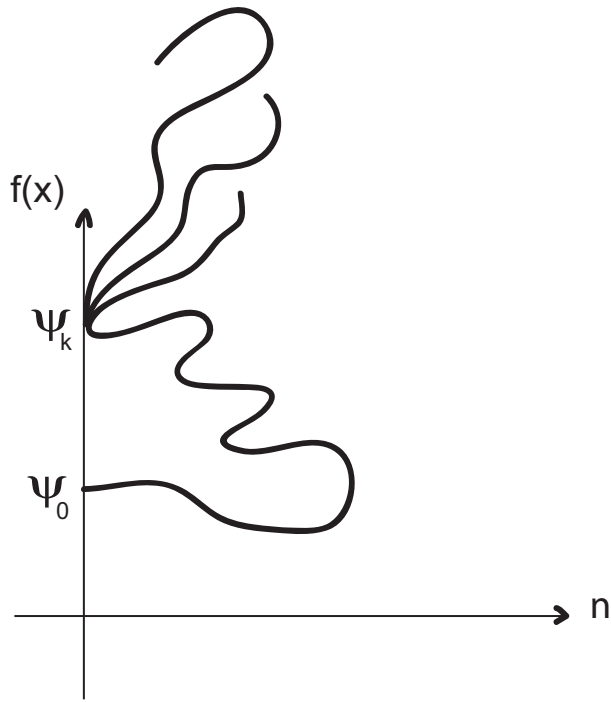}
\caption{A formal global branching $n$-diagram.} \label{fig1}
\end{center}
\end{figure}
Therefore, according to this, we believe that there is no
bifurcation from the branch of trivial solutions at any value of
the parameter $n>0$. However, a rigorous proof of this fact is
difficult, since our nonlinear operators are not monotone and,
hence, the main techniques usually used in the analysis of
second-order operators are not applicable here. However, if that
was the situation, we could obtain the existence of certain
turning points, or even a connection with another branch among the
ones emanating from some other eigenfunction $\psi_k$, with $k
\geq 1$; cf. Figure \ref{fig1}. We plan to carry out this work in a subsequent paper.

\setcounter{equation}{0}
\section{Blow-up similarity profiles for the Cauchy problem via  $n$-branching}
 \label{S5}

\subsection{Preliminaries: homotopy and nodal sets}

 In this section, we  describe the behaviour of the blow-up similarity solutions
 $\ef{upm}_-$ of the TFE-4 \eqref{tfe} through the same homotopic approach
 by setting
$n\downarrow 0$ in $\ef{eigpm}_-$ and, hence, arriving at the
linear adjoint operator \ef{ad55}. Then, we shall use the
eigenfunction patterns occurring for $n=0$ (those are generalized
Hermite polynomials \ef{s16})  as branching points for nonlinear
eigenfunctions providing us with a straightforward and  practical
$n$-continuity approach to the self-similar equation \eqref{sf4}
associated with the TFE-4 \eqref{tfe} from the equation \eqref{i4}
associated with the bi-harmonic equation \eqref{s1}.

It is worth recalling now that homotopic approaches are well-known
in the theory of vector fields and nonlinear operator theory (see
\cite{Deim, KZ} for details).
In our case, a ``homotopic path" just declares the existence of a
continuous connection (a curve) of some nonlinear eigenfunctions
$f=f_k^-(y)$ satisfying \ef{sf4} that ends up at $n=0^+$ at the
linear adjoint polynomials $\psi_k^*(y)$ given in \ef{s16}.
Due to Section\;\ref{S3}, we already know that those profiles
correspond to generalized Hermite polynomials given by
\eqref{s16}, which have {\em finite oscillatory properties}. For
instance, for any even $|\b|$, the polynomials \ef{s16} do not
have any zero nodal surface at all. However, for $k \ge 1$, linear
combinations of such eigenfunctions do have nodal sets of known
and relatively simple structure.

For odd $|\b|$ (or on multi-dimensional eigenspaces), arbitrary
linear combinations of Hermite polynomials for such a fixed
$k=|\b|  \ge 1$ explain  all possible structures of nodal sets and
(see \cite{2mSturm} for a full formulation)
 \be
 \label{cero1}
 \fbox{
 multiple zeros of solutions of the
bi-harmonic equation (\ref{s1}).}
 \ee


Furthermore, it turns out that, using classical branching theory,
 ``nonlinear eigenfunctions" $f_k(y)$ of changing sign, which
satisfies the {\em nonlinear eigenvalue problem} \eqref{sf4} (with
an extra ``radiation-minimal-like" condition at infinity to be
specified shortly),
 at least, for
sufficiently small $n > 0$, can be connected with the adjoint
polynomials $\psi^*_k(y)$ in \ef{s16}, or their linear
combinations from the eigenspace. We are capable of justify this
through the corresponding Lyapunov--Schmidt branching equation,
trying to be as rigorous as possible in supporting and
deriving the critical nonlinear eigenvalues $\a_k(n)$.

\subsection{Towards ``minimal growth at infinity"}

This is about the ``minimal" (a ``radiation-like") condition at
infinity, announced in \ef{bc2}, which makes the equation \ef{sf4}
to be a nonlinear eigenvalue problem. We recall that, for $n=0$,
the ${\bf (NEP)_-}$ in \ef{eigpm} reduces to the linear one for
the operator $\BB^*$ in \ef{ad55}, with the straightforward
correspondence
 \be
 \label{kk1}
  \tex{
  \a_k(0)= \l_k \equiv - \frac k4, \quad k=0,1,2,...\,.
  }
  \ee

Equation \ef{sf4} admits two kinds of asymptotics at infinity. The
first one is nonlinear and is given by the first two operators:
assuming simple radial behaviour ($f=y^{\g}$)  yields
 \be
 \label{kk2}
  \tex{
  -\n \cdot (|f|^n \n \D f)- \frac{1-\a n}4\, y\cdot \n f +...=0
  \LongA
  f(y) \sim |y|^{\frac 4n} \asA y \to \iy.
  }
 \ee
 Note that, as $n \to 0^+$, precisely this behaviour leads to an
 exponentially growing bundle, which is prohibited in Theorem
 \ref{Th s2} by specifying the proper weighted space
 $L^2_{\rho^*}(\ren)$ and eventually leading to the polynomial
 eigenfunctions \ef{s16}.

 The second asymptotics is linear: as $y \to \iy$,
  \be
  \label{kk3}
  \tex{
  - \frac{1-\a n}4 \, y \cdot \n f - \a f+...=0
  \LongA f(y) \sim |y|^\g \whereA \g= - \frac {4\a}{1-\a n}>0,
   }
   \ee
   since by \ef{kk1} we have to assume that $\a_k(n)<0$ (the first eigenvalue $\a_0(n)=0$
   is not of particular interest; see below) and always  $\a_k(n) <
   \frac 1n$. Note that then
   \be
   \label{kk4}
   \tex{
   \g \equiv  \frac {4|\a|}{1+|\a| n} < \frac 4n,
    }
    \ee
 so that the linear behaviour \ef{kk3} is the actual {\em minimal} one in
 comparison with  \ef{kk2}.

Overall, this allows to formulate such a ``radiation-like"
condition at infinity, which now takes a clear ``minimal nature":
 \be
 \label{kk5}
  \fbox{$
 \mbox{find solutions $f(y)$ of (\ref{sf4}) bounded at infinity by functions as in
 (\ref{kk3}).}
 $}
 \ee

In self-similar approaches and ODE theory, such conditions are
known to define similarity solutions of the {\em second kind}, a
term, which was introduced by Ya.B.~Zel'dovich in 1956
 \cite{Zel56}. Many of such ODE problems (but indeed, easier) have been rigorously
 solved since then.
  For quasilinear elliptic equations such as
 \ef{sf4}, the condition \ef{kk5} is more subtle and delicate indeed.
We cannot somehow rigorously justify that the problem \ef{sf4},
\ef{kk5} is well posed and admits a countable family of solutions
and nonlinear eigenvalues $\{\a_\b(n)\}$. We recall that using the
homotopy deformation as $n \to 0^+$ was our original intention
{\em in order to avoid} such a difficult ``direct" mathematical
attack of this nonlinear blow-up eigenvalue problem.

\ssk

We begin our actual study by noting that the first nonlinear
pair for \ef{sf4}, \ef{kk5} is trivial: for any $n>0$,
 \be
 \label{kk6}
 \a_0(n)=0 \andA f^-_0(y) \equiv 1,
  \ee
so that this well corresponds to the first Hermite polynomial from
\ef{s16} with $|\b|=0$, where $\psi_0^*(y) \equiv 1$. However,
similarity solutions \ef{u-} with the first eigenfunction in
\ef{kk6} are trivial and do not change sign, so, to understand
formation of nonlinear ``multiple zeros", we will study branching
of eigenfunctions $f_k^-(y)$ for $k \ge 1$.

\subsection{Technical bifurcation calculus}

Thus, the critical values $\a_k(n)$ are obtained for small $n>0$
according to spectral theory established in Section \ref{S3}. As
was noticed, the explicit expression for the eigenvalues and
eigenfunctions for the operator ${\bf B}^*$ in \ef{ad55}
are  known; see Theorem \ref{Th s2}. Moreover, supposing the
corresponding linear counterpart from \eqref{s16} with $n=0$, we
find, at least formally, that
\begin{equation}
\label{lin}
 \tex{
    \mathcal{L}(\a)f:=-\D^2 f -\frac{1}{4}\, y \cdot\nabla f  -\a f=0.
     }
\end{equation}
This equation can be considered as a linear perturbation in terms of the parameter $\a$
 of that for the adjoint operator $\BB^*$ in \eqref{ad55}.
From that equation combined with the eigenvalue expressions
obtained for the operator ${\bf B}^*$, we derive the  critical
values for the parameter $\a$ given in \ef{kk1},
where $\l_k$ are the eigenvalues  defined in Theorem\,\ref{Th s2}.
Note that those eigenvalues coincide with the eigenvalues of the
operator ${\bf B}$. In particular, when $k=0$, we have that $\a_0=
0= \l_0$ and the eigenfunction is $\psi_0^*=1$, satisfying
\begin{equation*}
    {\bf B}^* \psi_0^*=0, \quad \mbox{so that} \quad
    \ker \mathcal{L}(\a_0) = \mathrm{span\,}\{\psi_0^*=1\}.
\end{equation*}
 Hence,  $\l_0=0$ is a simple eigenvalue for the operator
$\mathcal{L}(\a_0)= {\bf B}^*$ and its algebraic multiplicity is
1. In general, we find that (note that $k=0$ is trivial)
\begin{equation}
\label{adjo}
 \tex{
    \ker\big({\bf B}^* + \frac{k}{4}\, I \big)= \mathrm{span\,}\{\psi_\b, \, |\b|=k
    \}, \quad \hbox{for any}
    \quad k=0,1,2,3,\dots\,,
    }
\end{equation}
where the operator ${\bf B}^* + \frac{k}{4}\, I$ is Fredholm of
index zero. In other words,
$R[\mathcal{L}(\a_k)]$ is a closed subspace of $L_{\rho}^2(\ren)$
and
\begin{equation*}
    \hbox{dim} \ker(\mathcal{L}(\a_k))< \infty, \quad \hbox{codim}\,R[\mathcal{L}(\a_k)]< \infty
\end{equation*}
for each $\a_k$. Moreover, $\dim \ker\big({\bf B} + \frac{k}{4}\,
I\big)=  M_k^*  \ge 1$,  for any $k=0,1,2,3,\dots$.

Then, once the relation between \eqref{lin} and the
linear operator ${\bf B}^*$ has been established,  for which we know its spectral
theory, by regularity issues in Section \ref{S1.4}, we can assume
for small $n>0$ in \eqref{s16} the following expansions:
\begin{equation}
\label{br3}
    \a_k(n):= \a_k+ \mu_{1,k} n+ o(n),\quad
    |f|^n \equiv |f|^n= {\mathrm e}^{n\ln |f|}:=
     1 +n \ln |f|+o(n),
\end{equation}
 where the last one is assumed to be understood in a weak sense.
Again, it is convenient to discuss further the last one. Indeed, the
second expansion cannot be interpreted pointwise for oscillatory
changing sign solutions $f(y)$, though now these functions are
assumed to have {\em finite} number of zero surfaces (as the
generalized Hermite polynomials for $n=0$ do). However, as usual,
this, of course, imposes some restrictions on the possible zeros
of the eigenfunctions $\psi_\b^*(y)$. According to the spectral
theory in Section\;\ref{S3} we already know that those
eigenfunctions and their linear combinations for the adjoint
operator ${\bf B}^*$ are generalized Hermite polynomials given by
\eqref{s16}. Hence, they are analytic functions with isolated
zeros.

Since the possible zeros are isolated, they can be localized in
arbitrarily  small neighbourhoods. Indeed, it is clear that when
$|f|>\d>0$ for any $\d>0$, there is no problem in approximating of
$|f|^n$ as in \eqref{br3}, i.e., $|f|^n = O(n)$ as $n \rightarrow
0^+$. However, when $|f| \leq \d $ for  $\d \geq 0$ sufficiently
small, the proof of such an approximation is far from clear unless
the zeros of the $f$'s are all transversal in a natural sense.
In view of the expected finite oscillatory nature of solutions
$f^-(y)$, this should allow one to obtain a weak convergence as in
\ef{nn2} to be used in the integral equation similar to \ef{int1}
(with $\BB$ replaced by $\BB^*$).

However, let us stress again that, in the present ``blow-up"
case, we do not need such subtle oscillatory properties of
solutions close to interfaces, which are not known in complicated
geometries. The point is that, due to the condition \ef{kk5}, we
are looking for solutions $f(y)$ exhibiting finite oscillatory and
sign changing properties, which are similar to those for linear
combinations of Hermite polynomials \ef{s16}.
 Hence, we can suppose that
their zeros (zero surfaces) are transversal a.e.,
so we find that, for  $n>0$  and any $\d= \d(n)
>0$ sufficiently small,
\begin{equation*}
    n| \ln |f| | \gg 1, \quad \hbox{if} \quad |f| \leq \d(n),
\end{equation*}
and, hence, on such subsets, $f(y)$ must be exponentially small:
\begin{equation*}
 \tex{
    | \ln |f| | \gg \frac{1}{n}\; \Longrightarrow \;\ln |f| \ll -\frac{1}{n}\;
    \Longrightarrow \; |f|  \ll {\mathrm e}^{-\frac{1}{n}}.
    }
\end{equation*}
Thus, we can control the singular coefficients in \eqref{br3}, and, in particular, see that
\begin{equation}
 \label{f1loc}
    \ln |f| \in L^1_{\rm loc} (\ren).
\end{equation}
Recall that this happens also in exponentially small
neighbourhoods of the transversal zeros.

It is worth recalling again that our computations below are to be
understood as those dealing with the equivalent integral equation
similar to \ef{int1} and operators, so, in particular, we can use
the powerful facts on compactness of the resolvent $(\BB-\l
I)^{-1}$ and the adjoint one $(\BB^*-\l I)^{-1}$ in  the
corresponding weighted $L^2$-spaces.

Note that, in such an
equivalent integral representation, the singular term in
(\ref{br3}) satisfying (\ref{f1loc}) makes no principal
difficulty, so the last expansion in (\ref{br3}) makes rather
usual sense for applying standard nonlinear operator theory.
Overall, the above analysis somehow justifies further branching
study. We must admit that this is not a rigorous one, but is
indeed sufficient for our formal expansions as $n \to 0^+$.

Thus, substituting \eqref{br3} into the nonlinear eigenvalue
problem \eqref{sf4} and omitting $o(n)$ terms when necessary, we
obtain the following expression:
\begin{equation*}
 \tex{
    -\nabla \cdot[(1 +n \ln |f|) \nabla \D f]
    -\frac{1-\a_k n - \mu_{1,k} n^2}{4}\, y \cdot\nabla  f
    -(\a_k+ \mu_{1,k} n) f=0\,,
    }
\end{equation*}
for any $k=0,1,2,3,\dots$\,. Hence, rearranging terms yields
\begin{equation*}
 \tex{
    -\D^2 f -n\nabla \cdot( \ln |f| \nabla \D f)
    -\frac{1}{4}\, y \cdot\nabla  f
    +\frac{\a_k n + \mu_{1,k} n^2}{4}\, y \cdot\nabla  f
    -\a_k f- \mu_{1,k} n f=0\,.
    }
\end{equation*}
In addition, using the expression of the operator ${\bf B}^*$ yields
\begin{equation}
\label{br5}
 \tex{
    \big({\bf B}^* + \frac{k}{4}\,I \big) f +n \mathcal{N}_k (f)+o(n)   =0\,,
     }
\end{equation}
with the operator
\begin{equation}
\label{br6}
 \tex{
    \mathcal{N}_k (f):=-\nabla \cdot( \ln |f| \nabla \D f)
    +\frac{\a_k}{4}\, y \cdot\nabla  f
    - \mu_{1,k} f\,.
    }
\end{equation}
Subsequently, we shall compute the coefficients involved in the
expansions \eqref{br3} applying the classical Lyapunov--Schmidt
method to \eqref{br5}  (branching approach when $n\downarrow 0$),
and, hence, describing the behaviour of the blow-up solutions for
at least small values of the parameter $n>0$. Two cases are
distinguished. The first one in which the eigenvalue is simple and
the second for which the eigenvalues are semisimple. Note that due
to Theorems\;\ref{Th s1} and \ref{Th s2}, for any $k\geq 0$, the
algebraic multiplicities are equal to the geometric ones, so we do
not deal with the problem of introducing the generalized
eigenfunctions (no Jordan blocks are necessary for restrictions to
eigenspaces).

\subsection{Simple eigenvalue}

Recall that this always happens for $k=0$ (not interesting) and
also in 1D and radial geometry, when all the eigenvalues of such
ordinary differential operators are simple.

As a typical example,
we perform the analysis as for $k=0$, bearing in mind the above
other more interesting applications.

Thus,
 since the first eigenvalue $\l_0=0$ of ${\bf B}^*$ is
simple, the dimension  of the eigenspace is $M_0^*=1$, the
analysis of this particular case presents less difficulties than
the corresponding ones for any other $k\geq 1$. Hence, denoting
$\ker\,{\bf B}^*= \mathrm{span\,}\{\psi_0^*=1\}$ and by $Y_0^*$
the complementary invariant subspace, orthogonal to $\psi_0$, we
set
\begin{equation}
\label{fad}
    f=\psi_0^*+V_0^*,
\end{equation}
where $V_0^* \in Y_0^*$. We define $P_0^*$ and $P_1^*$ such that
$P_0^*+P_1^*=I$, to be the projections onto $\ker\,{\bf B}^*$ and
$Y_0^*$ respectively. We next set
\begin{equation}
\label{br8}
 V_0^*:=n \Phi_{1,0}^* + o(n).
\end{equation}
Then, after substituting \eqref{fad} into \eqref{br5} and passing
to the limit as $n\rightarrow 0^+$, we arrive at a linear
inhomogeneous equation for $\Phi_{1,0}^*$
\begin{equation}
\label{br9}
    {\bf B}^*\Phi_{1,0}^*=- \mathcal{N}_0 (\psi_0^*),
\end{equation}
since ${\bf B}^* \psi_0^*=0$. By  Fredholm's theory \cite{Deim}
(spectral theory of the pair $\{\BB,\BB^*\}$ from Section \ref{S3}
does also matter), a unique solution $V_0^* \in Y_0^*$ of \ef{br9}
exists if and only if the right-hand side is orthogonal to the one
dimensional kernel of the adjoint operator, in this case ${\bf
B}$. In other words, in the topology of the dual space $L^2$, the
following holds:
\begin{equation}
\label{br10}
    \big\langle \mathcal{N}_0 (\psi_0^*), \psi_0 \big\rangle=0.
\end{equation}
Therefore, \eqref{br9} has a unique solution $\Phi_{1,0}^* \in
Y_0^*$ determining by \eqref{br8} a bifurcation branch for small
$n>0$. In addition, we obtain the following explicit expression
for the coefficient $\mu_{1,0}$ of the corresponding nonlinear
eigenvalue $\a_0(n)$ denoted by \eqref{br3}:
\begin{equation*}
    \mu_{1,0}
    :=
 \tex{
    \frac{\langle \nabla \cdot( \ln |\psi_0^*| \nabla \D \psi_0^*),\psi_0\rangle}
    {\langle \psi_0^*,\psi_0\rangle}
 }
     =
 \tex{
    \langle \nabla \cdot( \ln |\psi_0^*| \nabla \D \psi_0^*),\psi_0\rangle.
 }
\end{equation*}

\subsection{Multiple eigenvalues for  ${\bf k \ge 1}$}
 \label{S5.5}

For any
$k\geq 1$, we  know that
\begin{equation*}
 \tex{
    \dim \ker\big({\bf B}^* + \frac{k}{4} \,I\big)=  M_k^* >1
     \quad(\mbox{actually}, \,\,\,M_k^*=M_k).
    }
\end{equation*}
  Hence, we have to take the representation
\begin{equation}
\label{br12}
 \tex{
    f=\sum_{|\b|=k} c_\b \hat{\psi}_\b^* +V_k^*,
 }
\end{equation}
for every $k\geq 1$. Currently, for convenience, we denote
$\{\hat{\psi}_\b^*\}_{|\b|=k}=\{\hat \psi_1^*,...,\hat
\psi_{M_k}^*\}$, the natural basis of the $M_k^*$-dimensional
eigenspace $\ker\big({\bf B}^* + \frac{k}{4}\, I\big)$ and set
 $\psi_k^* =
\sum_{|\b|=k} c_\b \hat{\psi}_\b^*$.  Moreover, $V_k^* \in Y_k^*$
and $V_k^*=\sum_{|\b|>k} c_\b {\psi}_\b^*$, where $Y_k^*$ is the
complementary invariant subspace of $\ker\big({\bf B}^* +
\frac{k}{4}\, I\big)$. Furthermore, in the same way as we did for
the case $k=0$, we define the $P_{0,k}^*$ and $P_{1,k}^*$, for
every $k\geq 1$, to be the projections of $\ker\big({\bf B}^* +
\frac{k}{4}\, I\big)$ and $Y_k^*$ respectively. We also denote
$V_k^*$ by
\begin{equation}
\label{br13}
    V_k^*:=n \Phi_{1,k}^* + o(n).
\end{equation}
Subsequently, substituting \eqref{br12} into \eqref{br5} and
passing to the limit as $n\downarrow 0^+$, we obtain the following
equation:
\begin{equation}
\label{br14}
 \tex{
    \big({\bf B}^*+ \frac{k}{4}\,I\big)\Phi_{1,k}=- \mc{N}_k \big(\sum_{|\b|=k} c_\b {\psi}_\b^*\big),
    }
\end{equation}
under the natural ``normalizing" constraint
\begin{equation}
\label{nor}
 \tex{
    \sum\limits_{|\b|=k} c_\b=1.
    }
\end{equation}
Therefore, applying the Fredholm alternative \cite{Deim}, a unique
$V_k^* \in Y_k^*$ exists if and only if  the right-hand side of \eqref{br14}
is orthogonal to $\ker\,\big({\bf B}^*+ \frac{k}{4}\,I\big)$.
Multiplying the right-hand side  of \eqref{br14} by $\psi_\b$, for
every $|\b|$, in the topology of the dual space $L^2$, we obtain
an algebraic system of $M_k^*+1$ equations and the same number of
unknowns, $\{c_\b, \, |\b|=k\}$ and $\mu_{1,k}$:
\be
 \label{alg1}
  \tex{
\big\langle \mc{N}_k (\sum_{|\b|=k} c_\b {\psi}_\b^*), \psi_\b
\big\rangle=0 \quad \mbox{for all} \quad  |\b|=k,
 }
 \ee
 which is indeed the Lyapunov--Schmidt branching equation
 \cite{VainbergTr}.
Through that algebraic system we shall ascertain the coefficients
of the expansions \eqref{br3} and, hence, eventually the
directions of branching, as well as the number of branches.
However, a full solution of the {\em non-variational} algebraic
system \eqref{alg1} is a very difficult issue, though we claim
that the number of branches is expected to be related to the
dimension of the eigenspace $\ker\,\big({\bf B}^*+
\frac{k}{4}\,I\big)$.

In order to obtain the number of possible branches and with  the
objective of avoiding excessive notation, we analyze two typical
cases.

\ssk

\noi\underline{\sc Computations for branching of dipole solutions
in 2D}.  Firstly, we ascertain some expressions for those
coefficients in the case when $|\b|=1$, $N=2$ and $M_1^*=2$, so
that, in our notations, $\{\psi_\b\}_{|\b|=1}=\{\hat \psi_1^*,
\hat \psi_2^*\}$ such that $\psi_1^*=c_1 \hat \psi_1^*+c_2 \hat
\psi_2^*$. Consequently, in this case, we obtain the following
algebraic system:
\begin{equation}
\label{br16}
    \left\{\begin{array}{l}
    c_1  \langle \hat \psi_1,h_1 \rangle+ \frac{c_1 \a_1}{4}\,
     \langle \hat \psi_1,y \cdot\nabla  \hat{\psi}_1^* \rangle -c_1 \mu_{1,1}
    +c_2  \langle \hat \psi_1,h_2 \rangle+\frac{c_2 \a_1}{4}\,
     \langle \hat \psi_1,y \cdot\nabla  \hat{\psi}_2^* \rangle = 0,\ssk\\
    c_1  \langle \hat \psi_2,h_1 \rangle +\frac{c_1 \a_1}{4}\,
     \langle \hat \psi_2,y \cdot\nabla  \hat{\psi}_1^* \rangle
    + c_2  \langle \hat \psi_2,h_2 \rangle  + \frac{c_2 \a_1}{4}\,
     \langle \hat \psi_2,y \cdot\nabla  \hat{\psi}_2^* \rangle - c_2 \mu_{1,1}=0,\ssk\\
    c_1+c_2=1,
    \end{array}\right.
\end{equation}
where
\begin{equation*}
   h_1:= -\nabla \cdot  [ \ln (c_1 \hat{\psi}_1^*+c_2\hat{\psi}_2^*) \nabla \D
   \hat{\psi}_1^* ],\,\,
    h_2:= -\nabla \cdot
     [ \ln (c_1 \hat{\psi}_1^*+c_2\hat{\psi}_2^*) \nabla \D \hat{\psi}_2^* ],
\end{equation*}
and, $c_1$, $c_2$ and $\mu_{1,1}$ are the coefficients that we
want to calculate. Also, $\a_1=\l_1$ is regarded as the value of
the parameter $\a$ denoted by \eqref{kk1} such that $\hat
\psi_{1,2}^*$ are the corresponding associated adjoint
eigenfunctions. Now,
 from the third equation we have $c_2=1-c_1$, so that
substituting it into the first two equations of \eqref{bf16} gives
\begin{equation}
\label{br17}
    \left\{\begin{array}{l}
    0=N_1(c_1,\mu_{1,1})+ c_1\frac{\a_1}{4} \,\big[
     \langle \hat \psi_1,y \cdot\nabla  \hat{\psi}_1^* \rangle
    -  \langle \hat \psi_1,y \cdot\nabla  \hat{\psi}_2^* \rangle\big], \ssk\\
    0=N_2(c_1,\mu_{1,1}) + c_1\frac{\a_1}{4}\,\big[
     \langle \hat \psi_2,y \cdot\nabla  \hat{\psi}_1^* \rangle
    - \langle \hat \psi_2,y \cdot\nabla  \hat{\psi}_2^* \rangle\big]- \mu_{1,1},
    \end{array}\right.
\end{equation}
where
\begin{equation*}
  \begin{split} &
   \tex{
   N_1(c_1,\mu_{1,1}):= c_1  \langle \hat \psi_1,h_1 \rangle
   + \langle \hat \psi_1,h_2 \rangle
   + \frac{\a_1}{4}\,
     \langle \hat \psi_1,y \cdot\nabla  \hat{\psi}_2^* \rangle
   -c_1  \langle \hat \psi_1,h_2 \rangle -c_1 \mu_{1,1},
 }
   \\ &
    \tex{
   N_2(c_1,\mu_{1,1}):=c_1  \langle \hat \psi_2,h_1 \rangle
   +  \langle \hat \psi_2,h_2 \rangle
   + \frac{\a_1}{4} \,\langle \hat \psi_2,y \cdot\nabla  \hat{\psi}_2^* \rangle-
   c_1 \langle \hat \psi_2,h_2 \rangle + c_1 \mu_{1,1}
    }
   \end{split}
\end{equation*}
represent the nonlinear parts of the algebraic system, with $h_0$
and $h_1$ just depending on $c_1$ in this case.
\par
To detect  solutions for the system \eqref{br16} we apply
the Brouwer fixed point theorem to
\eqref{br17} (see \cite{PV} for
further details). Then, we suppose that the values $c_1$ and
$\mu_{1,1}$ are the unknowns in a sufficiently big disc
$D_R(\hat{c}_1,\hat{\mu}_{1,1})$, centered in a possible
nondegenerate zero $(\hat{c}_1,\hat{\mu}_{1,1})$. Therefore, if
one of the next two conditions are satisfied
\begin{equation*}
     \langle \hat \psi_1,y \cdot\nabla  \hat{\psi}_1^* \rangle-
     \langle \hat \psi_1,y \cdot\nabla  \hat{\psi}_2^* \rangle\neq 0
     \quad \hbox{or} \quad \langle \hat \psi_2,y \cdot\nabla  \hat{\psi}_2^* \rangle-
     \langle \hat \psi_2,y \cdot\nabla  \hat{\psi}_1^* \rangle\neq 0,
\end{equation*}
the nonlinear algebraic system \eqref{br17} has at least one
non-degenerate solution. Note that multiplicity results are
extremely difficult to obtain. So, to ascertain the number of
solutions for those nonlinear finite-dimensional algebraic
problems is rather complicated. However, we expect, and in fact
compute it in some cases, that this is  somehow related to the
dimension of the corresponding eigenspace $\ker\,\big({\bf B}^*+
\frac{k}{4}\,I\big)$, $k\geq 1$.

Firstly, we calculate the number of solutions for the nonlinear
algebraic system \eqref{br16}.
 Integrating by parts the terms in which $h_1$ and $h_2$ are involved in the first two
  equations and rearranging terms, we arrive at (as usual, all
  integrals are over $\ren$)
\begin{align*}
 \tex{
   \int \nabla \psi_1 \cdot \ln (c_1 \hat{\psi}_1^*+c_2\hat{\psi}_2^*) \nabla \D
   (c_1 \hat{\psi}_1^*+ c_2 \hat{\psi}_2^*)
 }
    &
    \tex{
     + c_1 \frac{\a_1}{4}\,
    \int \hat \psi_1 y \cdot\nabla  \hat{\psi}_1^*-c_1 \mu_{1,1}
    }
    \\ &
 \tex{
    + c_2  \frac{\a_1}{4}\,
    \int \hat\psi_1 y \cdot\nabla  \hat{\psi}_2^* = 0,
    }
\end{align*}
\begin{align*}
 \tex{
   \int \nabla \hat \psi_2\cdot \ln (c_1 \hat{\psi}_1^*+c_2\hat{\psi}_2^*) \nabla \D
   (c_2 \hat{\psi}_1^*+ c_2 \hat{\psi}_2^*)
 }
    &
 \tex{
     +c_1  \frac{\a_1}{4}\,
   \int \hat \psi_2 y \cdot\nabla  \hat{\psi}_1^*
    - c_2 \mu_{1,1}
 }
     \\ &
      \tex{
    +c_2  \frac{\a_1}{4}\,
    \int \hat\psi_2 y \cdot\nabla  \hat{\psi}_2^* =0.
     }
\end{align*}
Then, substituting the third equation with the expression $c_1 =
1- c_2$ and, hence, putting $c_1 \hat{\psi}_1^*+c_2\hat{\psi}_2^*
= \hat{\psi}_1^*+(\hat{\psi}_2^*- \hat{\psi}_1^*)c_2$ into those
two equations obtained above yields
\begin{align*}
 \tex{
   \int \nabla \hat \psi_1 \cdot
   \ln (\hat{\psi}_1^*+
 }
    &
     \tex{
     (\hat{\psi}_2^*-\hat{\psi}_1^*)c_2) \nabla \D
   (\hat{\psi}_1^*+(\hat{\psi}_2^*- \hat{\psi}_1^*)c_2) -\mu_{1,1}
    +c_2 \mu_{1,1}
     }
    \\ &
 \tex{
     + \frac{\a_1}{4}\,
    \int \hat \psi_1 y \cdot\nabla  \hat{\psi}_1^*
    + c_2  \frac{\a_1}{4}\, \int \hat{\psi}_1 y
    \cdot(\nabla  \hat{\psi}_2^*-\nabla \hat{\psi}_1^*)  = 0,
     }
\end{align*}
\begin{equation}
\label{br59}
\begin{split}
 \tex{
   \int \nabla \hat \psi_2\cdot
   \ln (\hat{\psi}_1^*+
   }
    &
     \tex{
     (\hat{\psi}_2^*-\hat{\psi}_1^*)c_2) \nabla \D
   (\hat{\psi}_1^*+(\hat{\psi}_2^*- \hat{\psi}_1^*)c_2)
    - c_2 \mu_{1,1}
 }
     \\ &
 \tex{
     + \frac{\a_1}{4}\,
    \int \hat \psi_2 y \cdot\nabla  \hat{\psi}_1^*
    +c_2  \frac{\a_1}{4}\,
    \int \hat \psi_2 y
    \cdot(\nabla  \hat{\psi}_2^*-\nabla \hat{\psi}_1^*) =0.
    }
\end{split}
\end{equation}
Subsequently, adding both equations, we have that
\begin{align*}
   \mu_{1,1} &  =
 \tex{
     \int  (\nabla \hat \psi_1+ \nabla \hat \psi_2) \cdot
   \ln (\hat{\psi}_1^*+  (\hat{\psi}_2^*-\hat{\psi}_1^*)c_2) \nabla \D
   (\hat{\psi}_1^*+(\hat{\psi}_2^*- \hat{\psi}_1^*)c_2)
   }
    \\ &
 \tex{
     + \frac{\a_1}{4}\,
    \int (\hat{\psi}_1+ \hat{\psi}_2) y
    \cdot \nabla  \hat{\psi}_1^*
    + c_2  \frac{\a_1}{4}\,
    \int (\hat \psi_1+ \hat \psi_2) y
    \cdot(\nabla  \hat{\psi}_2^*-\nabla \hat{\psi}_1^*).
    }
\end{align*}
Substituting it into the second equation of \eqref{br59}, we find
the  following equation with the single unknown $c_2$:
\begin{equation}
\label{br60}
\begin{split}
   &
    \tex{
    -c_2^2 \frac{\a_1}{4}\,
    \int (\hat \psi_1+ \hat \psi_2) y
    \cdot (\nabla  \hat{\psi}_2^*-\nabla \hat{\psi}_1^*) +
    c_2 \frac{\a_1}{4}(\,
    \int \hat\psi_2 y
    \cdot \nabla  \hat{\psi}_2^*-
    \,
    \int (\hat\psi_1+ 2 \hat\psi_2) y
    \cdot \nabla  \hat{\psi}_1^*)
    }
    \\ &
 \tex{
     + \frac{\a_1}{4}\,
    \int \hat\psi_2 y \cdot\nabla  \hat{\psi}_1^*
    +\int \nabla \hat{\psi}_2\cdot
   \ln (\hat{\psi}_1^*+ (\hat{\psi}_2^*-\hat{\psi}_1^*)c_2) \nabla \D
   (\hat{\psi}_1^*+(\hat{\psi}_2^*- \hat{\psi}_1^*)c_2)
   }
   \\ &
 \tex{
    -c_2 \int (\nabla \hat \psi_1 +\nabla \hat\psi_2)\cdot
   \ln (\hat{\psi}_1^*+ (\hat{\psi}_2^*-\hat{\psi}_1^*)c_2) \nabla \D
   (\hat{\psi}_1^*+(\hat{\psi}_2^*- \hat{\psi}_1^*)c_2) =0,
   }
\end{split}
\end{equation}
which can be written in the following way:
\begin{equation}
 \label{FF1}
    c_2^2 A + c_2 B + C + \o(c_2) \equiv \mf{F} (c_2) + \o
    (c_2)=0.
\end{equation}
 Here $\o(c_2)$ can be considered as a perturbation of the quadratic
form $\mf{F} (c_2)$ with the coefficients of such a quadratic form
defined by
\begin{align*}
    &
     \tex{
     A:=
  -\frac{\a_1}{4}\,
    \int (\hat\psi_1+ \hat\psi_2) y
    \cdot (\nabla  \hat{\psi}_2^*-\nabla \hat{\psi}_1^*),
 }
    \\ &
    \tex{
  B:= \frac{\a_1}{4}(\,
    \int  \hat\psi_2 y
    \cdot \nabla  \nabla \hat{\psi}_2^* -
    \,
    \int (\hat\psi_1+ 2 \hat\psi_2) y
    \cdot \nabla  \hat{\psi}_1^*),
    }
     \\ &
  C:=
 \tex{
   \frac{\a_1}{4}\,
    \int \hat\psi_2 y \cdot\nabla  \hat{\psi}_1^*,
    }
     \\ &
  \o(c_2):=
  \tex{
   \int \nabla \hat\psi_2 \cdot
   \ln (\hat{\psi}_1^*+ (\hat{\psi}_2^*-\hat{\psi}_1^*)c_2) \nabla \D
   (\hat{\psi}_1^*+(\hat{\psi}_2^*- \hat{\psi}_1^*)c_2)
   }
   \\ &
    \tex{
    -c_2 \int (\nabla \hat\psi_1+\nabla \hat\psi_2)\cdot
   \ln (\hat{\psi}_1^*+ (\hat{\psi}_2^*-\hat{\psi}_1^*)c_2) \nabla \D
   (\hat{\psi}_1^*+(\hat{\psi}_2^*- \hat{\psi}_1^*)c_2).
   }
\end{align*}
Hence, due to the normalizing constraint \eqref{nor}, $c_2 \in
[0,1]$, solving the quadratic equation $ \mf{F} (c_2)=0$ yields:
\begin{enumerate}
\item[(i)] $c_2=0 \Longrightarrow \mf{F}(0) = C $;
\item[(ii)] $c_2 =1 \Longrightarrow \mf{F}(1)= A+B+C $; and
\item[(iii)] differentiating $\mf{F}$ with respect to $c_2$, we obtain that
$\mf{F}'(c_2) = 2 c_2 A+B$. Then, the critical point of the
function $\mf{F}$ is $c_2^* = -\frac{B}{2A}$ and its image is $
\mf{F}(c_2^*)=- \frac{B}{4A}+C$.
\end{enumerate}

Therefore, since we know about the existence of at least one
solution, different from zero, in this particular case we impose
some conditions in order to have at most two solutions:
\begin{enumerate}
\item[(a)] $C  (A+B+C) >0$;
\item[(b)] $C \big(-\frac{B}{4A}+C\big)<0 $; and
\item[(c)] $0<-\frac{B}{2A}<1$.
\end{enumerate}
Note that, for $-\frac{B}{4A}+C= 0$, we have just a single
solution.
\par
Consequently, going back again to the equation \eqref{FF1}, we
need to control somehow the perturbation of the quadratic form to
maintain the number of solutions. Therefore, controlling the
possible oscillations of the perturbation $\o (c_2)$ in such a way
that
\begin{equation*}
    \left\| \o (c_2) \right\|_{L^\infty} \leq \mf{F}(c_2^*),
\end{equation*}
we can assure that the number of solutions for \eqref{br16} is
exactly two (or at most two). This is actually the dimension of
the kernel for the operator ${\bf B}+\frac{1}{4}\, I$ as we
conjectured. Note  that, in general and for large $k \gg 1$, to
solve such multiplicity problems for those types of
non-variational equations is a rather difficult open problem.

\ssk

\noi\underline{\sc Branching computations for $|\b|=2$}.
Subsequently,  we shall extend those results for the case in which
the dimension of the eigenspace is greater than 1. Again the
calculus are rather tedious. For that reason we find it easier to
make such calculations for the particular case when $|\b|=2$ and
$M_2^*=3$ ($N=2$), so that $\{\psi_\b^*\}_{|\b|=2}=\{\hat
\psi_1^*, \hat \psi_2^*, \hat \psi_3^*\}$ stands for a basis of
the eigenspace $\ker\big({\bf B}^* + \frac{1}{2}\, I \big)$, with
$k=2$ and $\l_k=- \frac k4$ as the associated eigenvalue. Observe
that $\a_k(0) =\l_k$.

Thus, in this case, performing in a similar way as was done for
\eqref{br16} with $\psi_2^*= c_1 \hat \psi_1^*+c_2 \hat \psi_2^*+c_3
\hat \psi_3^*$, we arrive at the following algebraic system:
\begin{equation}
\label{br61}
    \left\{\begin{array}{l}
    \begin{array}{r}
     c_1  \langle \hat \psi_1,h_1 \rangle
    +c_2  \langle \hat \psi_1,h_2 \rangle
    +c_3  \langle \hat \psi_1,h_3 \rangle + \frac{c_1 \a_2}{4}\,
     \langle \hat \psi_1,y \cdot\nabla  \hat{\psi}_1^* \rangle
    + \frac{c_2 \a_2}{4}\,
     \langle \hat \psi_1,y \cdot\nabla  \hat{\psi}_2^* \rangle
    \\+ \frac{c_3 \a_2}{4}\,
     \langle \hat \psi_1,y \cdot\nabla  \hat{\psi}_3^* \rangle -c_1 \mu_{1,2}= 0,
    \ssk \\
     c_1  \langle \hat \psi_2,h_1 \rangle
    + c_2  \langle \hat \psi_2,h_2 \rangle
    + c_2  \langle \hat \psi_2,h_3 \rangle  + \frac{c_1
    \a_2}{4}\,
     \langle \hat \psi_2,y \cdot\nabla  \hat{\psi}_1^* \rangle
    + \frac{c_2 \a_2}{4}\,
     \langle \hat \psi_2,y \cdot\nabla  \hat{\psi}_2^* \rangle
     \\  + \frac{c_3 \a_2}{4}\,
     \langle \hat \psi_2,y \cdot\nabla  \hat{\psi}_3^* \rangle - c_2 \mu_{1,2}=0,
     \ssk \\
     c_1  \langle \hat \psi_3,h_1 \rangle
    + c_2  \langle \hat \psi_3,h_2 \rangle
    + c_2  \langle \hat \psi_3,h_3 \rangle + \frac{c_1
    \a_2}{4}\,
     \langle \hat \psi_3,y \cdot\nabla  \hat{\psi}_1^* \rangle
    + \frac{c_2 \a_2}{4}\,
     \langle \hat \psi_3,y \cdot\nabla  \hat{\psi}_2^* \rangle
     \\  + \frac{c_3 \a_2}{4}\,
     \langle \hat \psi_3,y \cdot\nabla  \hat{\psi}_3^* \rangle - c_3 \mu_{1,2}=0,
     \end{array}\\
    c_1+c_2+c_3=1,
    \end{array}\right.
\end{equation}
where
\begin{equation*}
   h_1:= -\nabla \cdot  [ \ln (c_1 \hat{\psi}_1^*+c_2\hat{\psi}_2^*
   +c_3 \hat{\psi}_3^*) \nabla \D
   \hat{\psi}_1^* ],\,\,
    h_2:= -\nabla \cdot
     [ \ln (c_1 \hat{\psi}_1^*+c_2\hat{\psi}_2^*
     +c_3 \hat{\psi}_3^*) \nabla \D \hat{\psi}_2^* ],
\end{equation*}
\begin{equation*}
  \hbox{and}\quad
   h_3:= -\nabla \cdot  [ \ln (c_1 \hat{\psi}_1^*+c_2\hat{\psi}_2^*
   +c_3 \hat{\psi}_3^*) \nabla \D
   \hat{\psi}_3^* ],\,\,
\end{equation*}
and $c_1$, $c_2$, $c_3$, and $\mu_{1,2}$ are
 unknowns.
   Here,
$\hat{\psi}_1,\hat{\psi}_2,\hat{\psi}_3$ represent the
eigenfunctions associated with the eigenvalue $\l_2=\a_2(0)$ and
$\hat{\psi}_1^*,\hat{\psi}_2^*,\hat{\psi}_3^*$ are the
corresponding adjoint eigenfunctions, which  are associated with
the same eigenvalue $\l_2$.

As for the case  $|\b|=1$, the application of the Brouwer fixed
point theorem  and the topological degree provide us with the
existence of a non-degenerate solution for the nonlinear algebraic
system \eqref{br61} under certain conditions.

Furthermore, in the subsequent analysis, we shall show  a possible
way  to ascertain the number of solutions of the nonlinear
algebraic system \eqref{br61}. Obviously, since the dimension of
the eigenspace is bigger than the corresponding one in the case
$|\b|=1$, the difficulty in obtaining multiplicity results increases.

We proceed as in the previous case. Firstly, we integrate by parts
those terms in which  $h_1$, $h_2$, and $h_3$ are involved. After
rearranging terms, this yields
\begin{align*}
 \tex{
   \int \nabla \psi_1 \cdot
   \ln (c_1 \hat{\psi}_1^*+c_2\hat{\psi}_2^*+ c_3 \hat{\psi}_3^*)}
   & \tex{ \nabla \D
   (c_1 \hat{\psi}_1^*+ c_2 \hat{\psi}_2^*+ c_3 \hat{\psi}_3^*)
 }
    \tex{
     + c_1 \frac{\a_2}{4}\,
    \int \hat \psi_1 y \cdot\nabla  \hat{\psi}_1^* - c_1 \mu_{1,2}
    }
    \\ &
 \tex{
    + c_2  \frac{\a_2}{4}\,
    \int \hat\psi_1 y \cdot\nabla  \hat{\psi}_2^*
    + c_3  \frac{\a_2}{4}\,
    \int \hat\psi_1 y \cdot\nabla  \hat{\psi}_3^* =
    0;
    }
\end{align*}
\begin{align*}
 \tex{
   \int \nabla \hat \psi_2\cdot
   \ln (c_1 \hat{\psi}_1^*+c_2\hat{\psi}_2^*+ c_3 \hat{\psi}_3^*)} &
   \tex{ \nabla \D
   (c_2 \hat{\psi}_1^*+ c_2 \hat{\psi}_2^*+ c_3 \hat{\psi}_3^*)
 }
 \tex{
     +c_1  \frac{\a_2}{4}\,
   \int \hat \psi_2 y \cdot\nabla  \hat{\psi}_1^*
    - c_2 \mu_{1,2}
 }
     \\ &
      \tex{
    +c_2  \frac{\a_2}{4}\,
    \int \hat\psi_2 y \cdot\nabla  \hat{\psi}_2^*
    +c_3  \frac{\a_2}{4}\,
   \int \hat \psi_2 y \cdot\nabla  \hat{\psi}_3^*=0;
     }
\end{align*}
\begin{align*}
 \tex{
   \int \nabla \hat \psi_3\cdot
   \ln (c_1 \hat{\psi}_1^*+c_2\hat{\psi}_2^*+ c_3 \hat{\psi}_3^*)} &
   \tex{ \nabla \D
   (c_2 \hat{\psi}_1^*+ c_2 \hat{\psi}_2^*+ c_3 \hat{\psi}_3^*)
 }
 \tex{
     +c_1  \frac{\a_2}{4}\,
   \int \hat \psi_3 y \cdot\nabla  \hat{\psi}_1^*
    - c_3 \mu_{1,2}
 }
     \\ &
      \tex{
    +c_2  \frac{\a_2}{4}\,
    \int \hat\psi_3 y \cdot\nabla  \hat{\psi}_2^*
    +c_3  \frac{\a_2}{4}\,
   \int \hat \psi_3 y \cdot\nabla  \hat{\psi}_3^*=0.
     }
\end{align*}
Next, by the fourth equation in \eqref{br61}, we have that $c_1 =
1-c_2-c_3$. Then, setting
\begin{equation*}
    c_1 \hat{\psi}_1+c_2\hat{\psi}_2+ c_3 \hat{\psi}_3 =
    \hat{\psi}_1+c_2(\hat{\psi}_2-\hat{\psi}_1)+ c_3
    (\hat{\psi}_3-\hat{\psi}_1)
\end{equation*}
and substituting this into  three equations above yields  a
nonlinear algebraic system:
\begin{align*}
 \tex{
   \int \nabla \hat \psi_1 \cdot
   \ln (\hat{\psi}_1^*+
 }
    &
     \tex{
     (\hat{\psi}_2^*-\hat{\psi}_1^*)c_2+(\hat{\psi}_3^*-\hat{\psi}_1^*)c_3) \nabla \D
   (\hat{\psi}_1^*+(\hat{\psi}_2^*- \hat{\psi}_1^*)c_2+
   (\hat{\psi}_3^*- \hat{\psi}_1^*)c_3)
     }
     \\ &
 \tex{
     -\mu_{1,2}
    +c_2 \mu_{1,2}+ c_3 \mu_{1,2} + \frac{\a_2}{4}\,
    \int \hat \psi_1 y \cdot\nabla  \hat{\psi}_1^*
     }
    \\ &
 \tex{
    + \frac{\a_2}{4} \,\int \hat{\psi}_1 y
    \cdot( (\nabla  \hat{\psi}_2^*-\nabla \hat{\psi}_1^*)c_2 +
    (\nabla \hat{\psi}_3^* - \nabla\hat{\psi}_1^*)c_3)  = 0;
     }
\end{align*}
\begin{align*}
 \tex{
   \int \nabla \hat \psi_2 \cdot
   \ln (\hat{\psi}_1^*+
 }
    &
     \tex{
     (\hat{\psi}_2^*-\hat{\psi}_1^*)c_2+(\hat{\psi}_3^*-\hat{\psi}_1^*)c_3) \nabla \D
   (\hat{\psi}_1^*+(\hat{\psi}_2^*- \hat{\psi}_1^*)c_2+
   (\hat{\psi}_3^*- \hat{\psi}_1^*)c_3)
     }
     \\ &
 \tex{
     -c_2 \mu_{1,2}
    + \frac{\a_2}{4}\,
    \int \hat \psi_2 y \cdot\nabla  \hat{\psi}_1^*
     }
    \\ &
 \tex{
    + \frac{\a_2}{4}\, \int \hat{\psi}_2 y
    \cdot( (\nabla  \hat{\psi}_2^*-\nabla \hat{\psi}_1^*)c_2 +
    (\nabla \hat{\psi}_3^* - \nabla\hat{\psi}_1^*)c_3)  = 0;
     }
\end{align*}
\begin{equation}
\begin{split}
\label{br63}
 \tex{
   \int \nabla \hat \psi_3 \cdot
   \ln (\hat{\psi}_1^*+
 }
    &
     \tex{
     (\hat{\psi}_2^*-\hat{\psi}_1^*)c_2+(\hat{\psi}_3^*-\hat{\psi}_1^*)c_3) \nabla \D
   (\hat{\psi}_1^*+(\hat{\psi}_2^*- \hat{\psi}_1^*)c_2+
   (\hat{\psi}_3^*- \hat{\psi}_1^*)c_3)
     }
     \\ &
 \tex{
     -c_3\mu_{1,2}
    + \frac{\a_2}{4}\,
    \int \hat \psi_3 y \cdot\nabla  \hat{\psi}_1^*
     }
    \\ &
 \tex{
    + \frac{\a_2}{4} \,\int \hat{\psi}_3 y
    \cdot( (\nabla  \hat{\psi}_2^*-\nabla \hat{\psi}_1^*)c_2 +
    (\nabla \hat{\psi}_3^* - \nabla\hat{\psi}_1^*)c_3)  = 0.
     }
\end{split}
\end{equation}
Now, adding the first equation of \eqref{br63} to the other two
ones, we have that
\begin{align*}
 \tex{
   \int (\nabla \hat \psi_1+ \nabla \hat{\psi}_2) \cdot
   \ln (\hat{\psi}_1^*+
 }
    &
     \tex{
     (\hat{\psi}_2^*-\hat{\psi}_1^*)c_2+(\hat{\psi}_3^*-\hat{\psi}_1^*)c_3) \nabla \D
   (\hat{\psi}_1^*+(\hat{\psi}_2^*- \hat{\psi}_1^*)c_2+
   (\hat{\psi}_3^*- \hat{\psi}_1^*)c_3)
     }
     \\ &
 \tex{
     -\mu_{1,2}
    + c_3 \mu_{1,2} + \frac{\a_2}{4}\,
    \int (\hat \psi_1+ \hat{\psi}_2) y \cdot\nabla  \hat{\psi}_1^*
     }
    \\ &
 \tex{
    + \frac{\a_2}{4}\, \int (\hat \psi_1+ \hat{\psi}_2) y
    \cdot( (\nabla  \hat{\psi}_2^*-\nabla \hat{\psi}_1^*)c_2 +
    (\nabla \hat{\psi}_3^* - \nabla \hat{\psi}_1^*)c_3)  = 0,
     }
\end{align*}
\begin{align*}
 \tex{
   \int (\nabla \hat \psi_1+ \nabla \hat{\psi}_3) \cdot
   \ln (\hat{\psi}_1^*+
 }
    &
     \tex{
     (\hat{\psi}_2^*-\hat{\psi}_1^*)c_2+(\hat{\psi}_3^*-\hat{\psi}_1^*)c_3) \nabla \D
   (\hat{\psi}_1^*+(\hat{\psi}_2^*- \hat{\psi}_1^*)c_2+
   (\hat{\psi}_3^*- \hat{\psi}_1^*)c_3)
     }
     \\ &
 \tex{
     -\mu_{1,2}
    + c_2 \mu_{1,2} + \frac{\a_2}{4}\,
    \int (\hat \psi_1+ \hat{\psi}_3) y \cdot\nabla  \hat{\psi}_1^*
     }
    \\ &
 \tex{
    + \frac{\a_2}{4}\, \int (\hat \psi_1+ \hat{\psi}_3) y
    \cdot( (\nabla  \hat{\psi}_2^*-\nabla \hat{\psi}_1^*)c_2 +
    (\nabla \hat{\psi}_3^* - \nabla \hat{\psi}_1^*)c_3)  = 0.
     }
\end{align*}
Subsequently, subtracting those equations yields
\begin{align*}
 \tex{
    \mu_{1,2}
 }
 &
 \tex{
     = \frac{1}{c_2-c_3} \,\big[ \int
      (\nabla \hat \psi_2- \nabla \hat{\psi}_3) \cdot
   \ln \Psi^* \nabla \D \Psi^* +\frac{\a_2}{4}\,
    \int (\hat \psi_2- \hat{\psi}_3) y \cdot\nabla  \hat{\psi}_1^*
  }
     \\ &
  \tex{
    +\frac{\a_2}{4}\, \int (\hat \psi_2- \hat{\psi}_3) y
    \cdot( (\nabla  \hat{\psi}_2^*-\nabla \hat{\psi}_1^*)c_2 +
    (\nabla \hat{\psi}_3^* - \nabla \hat{\psi}_1^*)c_3) \big],
  }
\end{align*}
where $\Psi^* = \hat{\psi}_1^*+
(\hat{\psi}_2^*-\hat{\psi}_1^*)c_2+(\hat{\psi}_3^*-\hat{\psi}_1^*)c_3$.
Thus, substituting it into \eqref{br63}, we arrive at the following system, with
$c_2$ and $c_3$ as the unknowns:
\begin{align*}
 \tex{
   c_3 \int (\nabla \hat{\psi}_1-
   \nabla \hat{\psi}_2  }
    &
     \tex{
     +\nabla \hat{\psi}_3) \cdot
   \ln \Psi
     \nabla \D \Psi^* - c_2 \int (\nabla \hat{\psi}_1+
   \nabla \hat{\psi}_2-\nabla \hat{\psi}_3) \cdot
   \ln \Psi \nabla \D \Psi^*} \\ & \tex{ + \int (
   \nabla \hat{\psi}_2-\nabla \hat{\psi}_3) \cdot
   \ln \Psi^* + \frac{\a_2}{4} \int (\hat{\psi}_2-\hat{\psi}_3)
   y \cdot\nabla \hat{\psi}_1^* \,} \\ & \tex{
    +c_2 \frac{\a_2}{4}\,
    \big[\int (\hat{\psi}_2-\hat{\psi}_3) y \cdot\nabla
    (\hat{\psi}_2^* -2 \hat{\psi}_1^*)- \int
    \hat{\psi}_1 y \cdot\nabla  \hat{\psi}_1^*\big]
     }
     \\ &
 \tex{
    + c_3 \frac{\a_2}{4}\, \big[\int (\hat{\psi}_2-\hat{\psi}_3) y \cdot\nabla
    (\hat{\psi}_3^* -2 \hat{\psi}_1^*)- \int
    \hat{\psi}_1 y \cdot\nabla  \hat{\psi}_1^*\big]
     }
     \\ &
 \tex{
    + c_2 c_3  \frac{\a_2}{4} \big[\, \int \hat{\psi}_1 y
    \cdot (\nabla  \hat{\psi}_2^*-\nabla \hat{\psi}_3^*)
    }
     \\
    &
 \tex{
    +\, \int
    (\hat{\psi}_2- \hat{\psi}_3) y
    \cdot (2\nabla \hat{\psi}_1^*-\nabla\hat{\psi}_2^*-\nabla\hat{\psi}_3^*))\big] } \\ &
    \tex{
    + c_3^2  \frac{\a_2}{4}\, \int (\hat{\psi}_1-
    \hat{\psi}_2+\hat{\psi}_3) y
    \cdot( \nabla \hat{\psi}_3^* - \nabla\hat{\psi}_1^*)}\\ &
    \tex{
    - c_2^2  \frac{\a_2}{4}\, \int (\hat{\psi}_1+
    \hat{\psi}_2-\hat{\psi}_3) y
    \cdot(\nabla  \hat{\psi}_2^*-\nabla \hat{\psi}_1^*)     = 0,
     }
\end{align*}
\begin{align*}
 \tex{
   c_3 \int \nabla \hat{\psi}_2 \cdot
   \ln \Psi^*
 }
    &
     \tex{
     \nabla \D \Psi^* - c_2 \int \nabla \hat{\psi}_3 \cdot
   \ln \Psi^* \nabla \D \Psi^* + c_3 \frac{\a_2}{4}\,
    \int \hat{\psi}_2 y \cdot\nabla  \hat{\psi}_1^*
    }
     \\ &
 \tex{
    -c_2 \frac{\a_2}{4}\,
    \int \hat{\psi}_3 y \cdot\nabla  \hat{\psi}_1^*
    + c_3 \frac{\a_2}{4}\, \int \hat{\psi}_2 y
    \cdot( (\nabla  \hat{\psi}_2^*-\nabla \hat{\psi}_1^*)c_2 +
    (\nabla \hat{\psi}_3^* - \nabla\hat{\psi}_1^*)c_3)
     }
     \\ &
 \tex{
    - c_2 \frac{\a_2}{4}\, \int \hat{\psi}_3 y
    \cdot( (\nabla  \hat{\psi}_2^*-\nabla \hat{\psi}_1^*)c_2 +
    (\nabla \hat{\psi}_3^* - \nabla\hat{\psi}_1^*)c_3)  = 0.
     }
\end{align*}
 These can be re-written in the following form:
\begin{equation}
\label{br64}
    \begin{split}
    &
    \mf{F}_1(c_2,c_3) +\o_1 (c_2,c_3) \equiv A_1 c_2^2+ B_1 c_3^2
    + C_1 c_2 +D_1 c_3+ E_1 c_2 c_3 +\o_1 (c_2,c_3)=0,
    \\ &
    \mf{F}_1(c_2,c_3)+\o_2 (c_2,c_3) \equiv A_2 c_2^2+ B_2 c_3^2
    + C_2 c_2 +D_2 c_3+ E_2 c_2 c_3 +\o_2 (c_2,c_3)=0,
    \end{split}
\end{equation}
\begin{align*}
    \tex{
    \hbox{where} \quad \o_1 (c_2,c_3)} & \tex{
   := c_3 \int (\nabla \hat{\psi}_1-
   \nabla \hat{\psi}_2
     +\nabla \hat{\psi}_3) \cdot
   \ln \Psi^*
     \nabla \D \Psi^*}
     \\ & \tex{
      - c_2 \int (\nabla \hat{\psi}_1+
   \nabla \hat{\psi}_2-\nabla \hat{\psi}_3) \cdot
   \ln \Psi^* \nabla \D \Psi^*
   } \\ & \tex{
   + \int (
   \nabla \hat{\psi}_2-\nabla \hat{\psi}_3) \cdot
   \ln \Psi^*- \frac{\a_2}{4} \int (\hat{\psi}_2-\hat{\psi}_3)
   y \cdot\nabla \hat{\psi}_1^*, \,
   }
\end{align*}
\begin{align*}
    \tex{
    \hbox{and} \quad \o_2 (c_2,c_3)} & \tex{
    :=c_3 \int \nabla \hat{\psi}_2 \cdot
   \ln \Psi^* \nabla \D \Psi^* - c_2 \int \nabla \hat{\psi}_3 \cdot
   \ln \Psi^* \nabla \D \Psi^*
   }
\end{align*}
are the perturbations of the quadratic polynomials
 $$
 \mf{F}_i(c_2,c_3) := A_i c_2^2+ B_i c_3^2 + C_i c_2 +D_i c_3+ E_i c_2 c_3,
\quad \hbox{with} \quad i=1,2.
 $$
 The system \ef{br64} can be re-written in a matrix form with two
 quadratic forms involved:
\begin{equation*}
\begin{cases}
    & \tex{ (c_2\,\,\,c_3) P_1 \binom{c_2}{c_3} + Q_1 \binom{c_2}{c_3}+F_1=0,} \\ &
    \tex{ (c_2\,\,\,c_3) P_1 \binom{c_2}{c_3} + Q_1 \binom{c_2}{c_3} +F_2=0,}
\end{cases}
\end{equation*}
where the matrices $P_j$ and $Q_j$ of the quadratic forms with
$j=1,2$ have the corresponding coefficients $A_j$ to $E_j$ as
entries, plus the perturbations of the quadratic forms denoted,
under this notation, by $F_j$, with $j=1,2$.

Then, owing to the conic classification, we are able to solve
\eqref{br64} (without the nonlinear perturbation) and obtain an
estimate  for the number of solutions of the original nonlinear
algebraic system \eqref{br61}.

Hence, according to the conic classification, we will have the
following conditions that will provide us with conic section of
each equation of the system \eqref{br64} (without the nonlinear
perturbation):
\begin{enumerate}
\item[(i)] If $B_j^2 - 4A_jC_j < 0$, the equation represents an {\em ellipse},
unless the conic is {\em degenerate}, for example $c_2^2 + c_3^2 +
a = 0$ for some positive constant $a$. So, if $A_j=B_j$ and
$C_j=0$ the equation represents a {\em circle};
\item[(ii)] If $B_j^2 - 4A_jC_j = 0$, the equation represents a
{\em parabola}; and
\item[(iii)] If $B_j^2 - 4A_jC_j > 0$, the equation represents a {\em hyperbola}.
If we also have $A_j + C_j = 0$ the equation represents a
hyperbola (a rectangular one).
\end{enumerate}

Therefore, taking into account the ``normalizing" constraint, the
zeros of our system will depend on the coefficients we have for
the system, so on the eigenfunctions that generate each eigenspace
$\ker\big({\bf B}^*+\frac{k}{4}\, I\big)$.

Observe that the number of
intersections between two conics oscillates from one to four.
Hence, this will be the possible number of branches that are
obtained for our problem. However, since the dimension of the
eigenspace in this particular case is three, it seems  that, in
this case, we have four branches, so  two of them should coincide,
though this claim remains uncertain.

Moreover, as was done for the previous case when $|\b|=1$, we need
to control the oscillations of the perturbation functions in order
to maintain the number of solutions. Consequently, assuming that
\begin{equation*}
    \left\| \o_i (c_2,c_3) \right\|_{L^\infty} \leq \mf{F}_i(c_2^*,c_3^*),
    \quad \hbox{with}\quad i=1,2,
\end{equation*}
we conclude  that the number of solutions must be between one and
four. This again gives us an idea of the difficulty of more
general multiplicity results.


\section{Final   comments}

\subsection{A first comment: towards evolutionary completeness}

According to \cite{CompG}, evolutionary completeness of the
nonlinear eigenfunction subsets $\Phi^-(n)$ simply means that
those functions describe all possible types of finite time blow-up
asymptotics for solutions of the TFE--4 \ef{tfe} in a
neighbourhood of any point $(x_0,t_0)$. For nonlinear evolution
equations, such a completeness is a very difficult question. As
far as we know, the evolution completeness result proved in
\cite{CompG} for the 1D porous medium equation on a bounded
interval remains the only known such result for essentially
quasilinear PDEs (i.e., not a perturbed semilinear equation).

 Indeed, for the TFE--4 \ef{tfe}, such a completeness problem is difficult
 beyond any imagination. In particular, this would include a full
 analysis of  all the asymptotics of the non-stationary
 quasilinear fourth-order degenerate parabolic flow \ef{BlowT} containing no monotone,
 coercive, potential, or order-preserving operators.

However, our homotopy approach somehow implies certain (but not
that strong and promising) confidence concerning the evolutionary
completeness of $\Phi^-(n)$ for $n>0$: the point is that, for
$n=0$, the eigenfunction set of the Hermite polynomials \ef{s16}
is indeed {\em complete} and {\em closed} in any suitable weighted
space, where those notions are now understood as in classic theory
of bi-orthogonal polynomials and Riesz bases. So we may expect
that
the evolution completeness for small $n>0$ can be ``inherited"
from those brilliant spectral properties available for $n=0$
(Section \ref{S3}). This is the only issue we are aware of and can
rely on in this analysis.

Same speculations apply to the evolutionary completeness of global
similarity patterns $\Phi^+(n)$ for small $n>0$, which is now
connected with completeness/closure of eigenfunctions \ef{s14} of
$\BB$ for $n=0$; see \cite{EGKP} for  proofs.

 \subsection{A pessimistic comment}

  Overall, we must admit that, though we have obtained some
  multiplicity results for not-that-multi-dimensional eigenspaces
and have shown certain extensions of our techniques, any further
rigorous justification seems to be too excessive. Indeed, any
rigorous results will inevitably require to specify or evaluate
with sufficient accuracy of those numerical values of various
projections given by linear functionals as linear combinations of
 the Hermite
 polynomials \ef{s16}. In view of a complicated nature
of  non-self-adjoint  theory for the spectral pair
$\{\BB,\BB^*\}$, this is expected to be  entirely illusive.

On the other hand, it would be very important to trace our
$n$-bifurcation branches of nonlinear eigenvalue problems in both
global and blow-up setting by using some more general and powerful
techniques of nonlinear operator theory. However, no one can
expect this to be a simply task. We suspect that, in
view of principally non-variational structure of such nonlinear
eigenvalue problems, containing no monotone and/or strongly
coercive operators, any non-local (in $n$) sharp results on
existence/multiplicity of $n$-branches will not be obtained
reasonably soon.

\ssk

Therefore, overall, we claim that our $n$-branching approach,
which allowed us to explain the occurrence of nonlinear branches
from linear eigenfunctions at $n=0$, though not being fully
rigorous, is the only currently available way to detect branching
phenomena for such nonlinear eigenvalue problems embracing similar
classes of non-variational and non-monotone operators. It is clear
how these homotopy-branching methods can be extended to more
general and more higher-order quasilinear operators of different
types, once  a parameter homotopy to a proper linear spectral
problem for a suitable non-self-adjoint pair $\{\BB,\BB^*\}$ has
been well understood and carefully and rigorously studied.
However, we warn that the latter linear problem often can be a
very difficult one itself; one such example of a refined
scattering theory for $2m$-th order linear Schr\"odinger operators
is under attack  in \cite{GalKamLSE}.


\begin{appendix}
\section*{Appendix A. Necessary functional setting for branching at $n=0^+$}
 \label{ApA}
 \setcounter{section}{1}
\setcounter{equation}{0}

\begin{small}

Here, we are going to present some justification of the our main
branching analysis. Namely, we need to deal with expansions such
as \ef{nn1} and/or \ef{nn2}. Recall that, using this, we are not
going to, and in fact cannot, justify rigorously the existence of
nonlinear eigenfunctions as solutions of $\ef{eigpm}_+$, \ef{bc1}
at least for small $n>0$, but just the branching at $n=0$, {\em
under the hypothesis that a proper limit}
 \be
 \label{lim1}
 f(y;n) \to f_0(y) \,\,(=\psi_\b(y)) \asA n \to 0^+
  \ee
  exists in a necessary metric to be specified. According to our spectral theory of the non-self adjoint pair
  $\{\BB,\BB^*\}$, here $f_0$ denote some eigenfunction $\psi_\b$,
  and, in the most simple and interesting case, we assume that
   \be
   \label{lim2}
   f_0(y)=F(y) \,\,(=\psi_0(y)),
    \ee
    where $F(y)$ is the rescaled kernel of the fundamental
    solution \ef{s3} of the bi-harmonic operator.

 Thus, we need to check under which extra assumptions on \ef{lim1}, the following
limit takes place, in the weak sense,
 \be
 \label{l3}
  \tex{
   \frac {|f(y;n)|^n-1}n \rightharpoonup \ln |f_0(y)|,
   }
   \ee
where the right-hand side is assumed to be well defined (bounded)
a.e.
 First of all, it is obvious that such a convergence crucially
 depends on the structure of zeros of the limit functions
 $f_0(y)$, which is easy to demonstrate:

 \ssk

 \noi{\bf Example: a non-transversal zero.} Let $f_0(y)$ have a
 non-transversal zero at, say, $y=0^-$ (the interface point), and
  \be
  \label{l4}
   \tex{
  f(y;n)= {\mathrm e}^{ \frac 1{ny}} \quad \mbox{for
  $y<0$} \LongA  \frac {|f(y;n)|^n-1}n = \frac  {{\mathrm e}^{\frac
  1{y}}-1}n \to \iy.
  }
   \ee
  Actually, this means that
   \be
   \label{l5}
   f(y;n) \to 0 \equiv f_0(y) \asA n \to 0^+ \quad \mbox{for all $y
   <0$},
   \ee
i.e., $y=0$ is not a transversal zero of $f_0(y)$. Then the limit
\ef{l3} makes no sense and the branching analysis at $n=0$ does
not apply at all.

\ssk

Fortunately, such a situation cannot occur for the analytic kernel
$F(y)$ and all its derivatives, representing other eigenfunctions.
Of course, we cannot guarantee that non-transversal zeros of
$F(y)$ cannot occur at all. They can, but with a lower probability
as for any analytic function. However, we do know that such
non-transversal zeros are always isolated and cannot concentrate
on a given surface in $\ren$. Therefore, on any compact subset
such non-transversal zero surfaces {\em have zero measure}.
 However, this is not sufficient and an extra rough estimate
would be useful.

Evidently, \ef{l3} is violated in the pointwise sense on a {\em
bad} set of points $b_*(n)$ such that
 \be
 \label{l51}
 |f(y;n)|\approx |f_0(y)| \ll {\mathrm e}^{-\frac 1n} \quad
 \mbox{for all $n>0$ small}.
  \ee
Then, in this {\em worst} case,
 \be
 \label{l6}
  \tex{
 \frac{|f_0(y)|^n-1}n \sim - \frac 1n \to \iy \asA n \to 0^+.
 }
 \ee
Assume that a non-transversal (a multiple) zero again occur at
$y=0$ and the 1D behaviour is as follows:
 \be
 \label{l7}
 f_0(y) \sim y^k \whereA k=2,3,4,...\, .
  \ee
Then, in the weak sense, the integral representation of \ef{l3}
will provide us with the ``bad" (``worst") discrepancy of the
order
 \be
 \label{l8}
  \tex{
  \sim \frac 1n \, {\rm meas}\, b_*(n) \sim \frac 1n \, {\mathrm
  e}^{-\frac 1{nk}} \to 0 \asA n \to 0^+
  }
  \ee
  for any {\em finite} multiplicity of the zero at $y=0$. It is
  clear that any use of the $\ren$-geometry of such multiple zeros
  cannot affect the non-analytic exponential term in \ef{l8} and the convergence.

We complete our discussion as follows:

\begin{proposition}
 \label{Pr.NN}
 Let \ef{lim1} hold uniformly on compact subset, where the limit
 function $f_0(y)$ satisfy the above assumption of a.a.
 transversal zeros. Then \ef{l3} holds in the weak sense.
 \end{proposition}

Finally, let us also formally note that, in \ef{nn5} on the bad
set $b_*(n)$, we have the following:
\be
 \label{nn51}
  \tex{
   \big|(\n \D)^{-1}\big( \frac{|f|^n-1}n \,\n \D f\big)\big| \sim
\big| (\n \D)^{-1}\big( \frac 1 n \,
 \n \D f \big)\big| \sim \big| \frac 1n \, f(y)\big| \ll \frac 1n \,
 {\mathrm e}^{-\frac 1n} \to 0
   }
   \ee
as $n \to 0^+$. This confirms that the convergence \ef{nn5} takes
place a.e., provided that the zero set of $f_0(y)$ has zero
measure only on any compact subset in $\ren$, i.e., the
analyticity is not required (Sard's theorem for $C^p$ functions in
$\ren$ with any $p \ge 1$ may be used). Of course, this is just a
rough estimate and further study is needed.

\end{small}
\end{appendix}


\begin{thebibliography}{99}



\bibitem{PV}
 P.~Alvarez-Caudevilla and V.A.~Galaktionov,
  \emph{On a branching analysis of similarity solutions of a fourth-order thin film equation},
   submitted (arXiv:0911.2996).


  \bibitem{A}
 A.~Ambrosetti, \emph{Branching points for a class of variational operators},
 {J.~Anal. Math.,} {\bf 76} (1998), 321--335.

 \bibitem
  {AGi04}
   L.~Ansini and L.~Giacomelli, {\em Doubly nonlinear thin-film
   equations on one space dimension}, Arch. Ration. Mech. Anal.,
   {\bf 173} (2004), 89--131.


   \bibitem
    {Beck05}
  J.~Becker and G.~Gr\"un, {\em The thin-film equation: recent advances and some
new perspectives},
  {J.~Phys.: Condens. Matter}, {\bf 17} (2005), S291--S307.

\bibitem 
 {BF1}
 F.~Bernis and A.~Friedman, \emph{Higher order nonlinear degenerate
 parabolic equations}, J. Differ. Equat., \textbf{83} (1990), 179--206.

\bibitem 
 {BMc91}
 F.~Bernis and J.B.~McLeod, \emph{Similarity solutions of a higher order
nonlinear
 diffusion
 equation}, Nonl. Anal., TMA, \textbf{17} (1991), 1039--1068.



\bibitem 
 {BPW}
 F. Bernis, L.A. Peletier, and S.M. Williams, \emph{Source type solutions of a
fourth order nonlinear degenerate
 parabolic equation}, Nonl. Anal., TMA, \textbf{18} (1992), 217--234.


\bibitem{BertBern}
A.~J. Bernoff and A.~L. Bertozzi, \emph{Singularities in a
modified
  {K}uramoto-{S}ivashinsky equation describing interface motion for phase
  transition}, Physica D, \textbf{85} (1995), 375--404.

\bibitem{BerPugh1}
A.L. Bertozzi and M.C. Pugh, \emph{Long-wave instabilities and
saturation in
  thin film equations}, Comm. Pure Appl. Math., \textbf{{LI}} (1998), 625--651.

\bibitem{BerPugh2}
A.L. Bertozzi and M.C. Pugh, \emph{Finite-time blow-up of
solutions of some long-wave
  unstable thin film equations}, Indiana Univ. Math. J., \textbf{49}
  (2000),  1323--1366.


\bibitem{BS}  M.S.~Birman and M.Z.~Solomjak,
{Spectral Theory of Self-Adjoint Operators in Hilbert Spaces}, {D.
Reidel}, Dordecht/Tokyo (1987).


\bibitem
  {BK01N}
  M.~Bowen and J.R.~King, {\em Moving boundary problems and
  non-uniqueness for the thin film equation},  Euro
J.~Appl. Math.,  {\bf 12} (2001), 321--356.





\bibitem{BW06} M.~Bowen and T.P.~Witelski, \emph{The linear limit of
 the dipole problem for the thin film equation}, {SIAM J.~Appl.
 Math.,} {\bf 66} (2006), 1727--1748.


\bibitem
{CarrT02} J.A.~Carrillo and G.~Toscani, {\em Long-time asymptotic
behaviour for strong solutions of the thin film equations}, Comm.
Math. Phys., {\bf 225} (2002), 551--571.



\bibitem{CR} M.G. Crandall and P.H. Rabinowitz, \emph{Bifurcation
from simple eigenvalues,} J. Funct. Anal., {\bf 8} (1971),
321-340.

\bibitem{Deim} K. Deimling, {Nonlinear Functional Analysis},
Springer-Verlag, Berlin/Tokyo, 1985.


\bibitem{DGM} M.~Del Pino, J.~Garc\'ia-Meli\'an, and M.~Musso,
\emph{Local bifurcation from the second eigenvalue of the
Laplacian in a square}, Proc.  AMS, \textbf{131} (2003),
3499--3505.


\bibitem{EGKP}  Yu.V. Egorov, V.A. Galaktionov, V.A. Kondratiev, and S.I. Pohozaev,
\emph{Global solutions of higher-order semilinear parabolic
equations in the supercritical range}, {Adv. Differ. Equat.,}
\textbf{9} (2004),  1009--1038.

\bibitem{EGK1} J.D. Evans, V.A. Galaktionov, and J.R. King,
\emph{Blow-up similarity solutions of the fourth-order unstable
thin film equation}, {European J. Appl. Math.,} \textbf{18}
(2007),
 195--231.

\bibitem{EGK2} J.D. Evans, V.A. Galaktionov, and J.R. King,
\emph{Source-type solutions of the fourth-order unstable thin film
equation}, {European J. Appl. Math.,} \textbf{18} (2007),
273--321.

\bibitem{EGK3} J.D. Evans, V.A. Galaktionov, and J.R. King,
\emph{Unstable sixth-order thin film equation. I. Blow-up
similarity solutions}, {Nonlinearity}, \textbf{20} (2007),
1799--1841.

\bibitem{EGK4} J.D. Evans, V.A. Galaktionov, and J.R. King,
\emph{Unstable sixth-order thin film equation. II. Global
similarity patterns}, {Nonlinearity}, \textbf{20} (2007),
1843--1881.


\bibitem 
 {FB0}
 R. Ferreira and F. Bernis, \emph{Source-type solutions to thin-film equations
in higher dimensions},  Euro J. Appl. Math., \textbf{8} (1997),
507--534.




\bibitem 
 {GalGeom}
 V.A. Galaktionov, {\rm Geometric Sturmian  Theory of Nonlinear
 Parabolic Equations and Applications}, Chapman and Hall/CRC, Boca Raton,
Florida,
 2004.


\bibitem
 {CompG}
 V.A.~Galaktionov,
{\em Evolution completeness of separable solutions of non-linear
diffusion equations in bounded domains}, Math. Meth. Appl. Sci.,
{\bf 27} (2004), 1755--1770.

\bibitem
 {2mSturm}
 V.A.~Galaktionov,
 {\em Sturmian 
nodal set analysis for higher-order parabolic equations and
applications}, Adv. Differ. Equat., {\bf 12} (2007), 669--720.



 \bibitem
 {GalKamLSE}
  V.A. Galaktionov and I.V. Kamotski,
  {\em {R}efined scattering and {H}ermitiam spectral theory for linear
  {S}chr\"odinger equations with applications},
 in preparation (to be available shortly in arXiv.org).


\bibitem 
{GS1S-V}
   V.A.~Galaktionov and A.E.~Shishkov, {\em Saint-Venant's
principle in blow-up for higher-order quasilinear parabolic
 equations}, { Proc. Royal Soc. Edinburgh, Sect.~A}, {\bf 133A} (2003), 1075--1119.



 \bibitem
 {GiacOtto08}
 L.~Giacomelli, H.~Knüpfer, and F.~Otto,
  {\em Smooth zero-contact-angle solutions to a thin-film equation around the steady state},
   J.~Differ. Equat., {\bf 245} (2008), 1454--1506.


\bibitem 
{Grun04} G.~Gr\"un, {\em Droplet spreading under weak slippage --
existence for the Cauchy problem}, {Commun. Part. Differ. Equat.,}
{\bf 29} (2004), 1697--1744.


\bibitem{KZ} M.A. Krasnosel'skii and P.P. Zabreiko,
{Geometrical Methods of Nonlinear Analysis}, Springer-Verlag,
Berlin/Tokio, (1984).

\bibitem{KHK} S.~Kr\"{o}mer, T.J.~Healey, and H.~Kielh\"{o}fer,
\emph{Bifurcation with a two-dimensional kernel}, J.~Differ.
Equat., \textbf{220}, (2006), 234--258.

\bibitem 
{LIO}  J.-L.~Lions, {\rm Quelques m\'{e}thodes de r\'{e}solution
des probl\`{e}mes aux limites non lin\'{e}aires\/}, Dunod,
Gauthier-Villars, Paris, 1969.


\bibitem{Lo} J.L\'{o}pez-G\'{o}mez, {Spectral Theory and Nonlinear
Functional Analysis.} Chapman \& Hall / CRC, Research Notes in
Mathematics 426, Boca Raton, Fl, 2001.

\bibitem{PT} L.A.~Peletier and W.C. Troy, Spatial Patterns.
Higher Order Models in Physics and Mechanics, Birkh\"{a}usser,
Boston/Berlin, 2001.

\bibitem{R} P.H. Rabinowitz, \emph{A bifurcation theorem for
potential operators,} J. Funct. Anal., {\bf 25} (4) (1977),
412--424.


\bibitem  
{St}   C.~Sturm, {\em  M\'emoire sur une classe d'\'equations \`a
diff\'erences partielles,} {J.~Math. Pures Appl.,} {\bf 1} (1836),
373--444.



\bibitem{VainbergTr} M.A.~Vainberg and V.A.~Trenogin, {\rm Theory of
Branching of Solutions of Non-Linear Equations}, Noordhoff Int.
Publ., Leiden, 1974.



\bibitem 
{Zel56} Ya.B.~Zel'dovich, {\em The motion of a gas under the
action of a short term pressure shock}, Akust. Zh., {\bf 2}
(1956), 28-38; Soviet Phys. Acoustics, {\bf 2} (1956), 25--35.


\end{thebibliography}
\end{document}